\documentclass[11pt,reqno]{amsart}
\usepackage[shortlabels]{enumitem}
\usepackage{esint}
\usepackage{physics}
\usepackage{relsize}
\usepackage[backref]{hyperref}
\usepackage{mathtools}
\usepackage{amsfonts}
\usepackage[all]{xy}
\usepackage{geometry}
\usepackage{amsmath,amssymb} 
\usepackage{faktor}
\usepackage{amscd}
\usepackage{tikz-cd}

\newcommand{\subsectionnotoc}[1]{%
  \par\addvspace{.5\baselineskip plus .7\baselineskip}%
  \noindent{\normalfont\bfseries #1.}\enspace\ignorespaces
}

\newtheorem{theorem}{Theorem}[section]
\newtheorem{lemma}[theorem]{Lemma}
\newtheorem{definition}[theorem]{Definition}

\theoremstyle{corollary}
\newtheorem{corollary}[theorem]{Corollary}
\theoremstyle{conjecture}

\newtheorem{problem}[theorem]{Problem}

\theoremstyle{assumption}

\theoremstyle{proposition}
\newtheorem{proposition}[theorem]{Proposition}
\theoremstyle{remark}
\newtheorem{remark}[theorem]{Remark}
\numberwithin{equation}{section}
\everymath{\displaystyle}

\newcommand{\ii}{\ensuremath{\sqrt{-1}}}
\newcommand{\pp}{\bar\partial}

\newcommand{\noop}[1]{}
\DeclareMathOperator{\vol}{Vol}
\DeclareMathOperator{\Rm}{Rm}
\DeclareMathOperator{\Ric}{Ric}

\begin{document}

\title{On the geometry of non-collapsed polarized cscK surfaces}
\author{Junsheng Zhang}
\address{Courant Institute of Mathematical Sciences\\
  New York University, 251 Mercer St\\
  New York, NY 10012\\}
\curraddr{}
\email{jz7561@nyu.edu}
\author{Keshu Zhou}
\address{Department of Mathematics, University of California, Berkeley, CA 94720, USA} 
\email{keshu\_zhou@berkeley.edu}

\begin{abstract}
We show that the Gromov–Hausdorff convergence of non-collapsed polarized constant scalar curvature Kähler (cscK) surfaces  can be realized as convergence in a Hilbert scheme. We also derive uniform estimate of Bergman kernels on the effective regular set. As an application, we establish the Zariski openness of cscK metrics for certain smooth polarized families, following the approach of Donaldson.
\end{abstract}
\maketitle
\tableofcontents

\section{Introduction}

A polarized constant scalar curvature K\"ahler surface, abbreviated as cscK metric below, consists of the data 
\( (X, g, J, \omega, L, A) \), 
where \( (X, g, J, \omega) \) is a smooth compact Kähler surface, 
\( L \) is a Hermitian holomorphic line bundle on \( X \), 
and \( A \) is a Hermitian connection on \( L \) whose curvature satisfies 
\( \omega = \sqrt{-1}\,\Theta_A \) and 
moreover, the scalar curvature \( S_{\omega} \) of \( \omega \) is constant.
For constants $V, C_S>0$, let $\mathcal{K}(V,C_S)$ be the class of polarized cscK surfaces satisfying the following conditions:
\begin{equation}\label{assumption--volume upper bound and scalar bound}
    \mathrm{Vol}(X,\omega)\leq V, \quad |S_{\omega}|\leq1,\quad
C_{\omega}\leq C_S,
\end{equation}
where $C_{\omega}$ is the scaling-invariant Sobolev constant for the K\"ahler manifold $(X,\omega)$; see Definition \ref{def--sobolev constant}. 

Given any sequence $X_j \in \mathcal{K}(V,C_S)$,  
the compactness theory developed in  
\cite{tian-viaclovsky0,tian-viaclovsky1}  (see also \cite{anderson05}), 
implies that, after passing to a subsequence,  
$X_j$ converges in the polarized 
$\hat C^{\infty}$–Cheeger–Gromov sense  
to a polarized cscK orbifold $X_\infty$;  
see Section~\ref{orbifold compactness} for details. 

\subsection{Main results}
The following is our first main theorem in this paper:
\begin{theorem}\label{algebraic convergence}
         By passing to a subsequence, the Gromov-Hausdorff convergence of $X_j$ to $X_{\infty}$ can be realized as convergence within a Hilbert scheme.
\end{theorem}
 
By converging within a Hilbert scheme, we mean that there exist constants $k_0$ and $N_0$, 
a flat projective morphism $\pi:\mathcal{X} \subset \mathcal{B} \times \mathbb{P}^{N_0} \to \mathcal{B}$,
and closed points $b_j, b_\infty \in \mathcal{B}$ with $b_j \to b_\infty$ (in the Euclidean topology), 
such that the fibers 
\[
(\mathcal{X}_{b_j}, \mathcal{O}(1)) \cong (X_j, k_0 L_j), \quad 
(\mathcal{X}_{b_\infty}, \mathcal{O}(1)) \cong (X_\infty, k_0 L_\infty)
\] 
are isomorphic as polarized varieties. 

Historically, this type of result arose from the existence question of canonical metrics on K\"ahler manifolds, and would provide a bridge between the Riemannian convergence theory and algebraic geometry \cite{donaldson2010,DS1,CDS15,szekelyhidi16,datar-gabor16,chen-wang20,chen-sun-wang18,liu-gabor}. We would like to emphasize that originally results of this type were established based on obtaining a uniform lower bound for the Bergman function, namely the partial $C^0$-estimate proposed by Tian \cite{tian1990}. A key observation of this paper is that, in certain cases (especially in complex dimension $2$), if one can understand the Gromov--Hausdorff convergence in a sufficiently precise manner, then one can directly obtain algebraic convergence in a flat family, even \textit{without} establishing the partial $C^0$-estimate. This provides flexibility in geometric applications. On the other hand, whether the partial $C^0$-estimate holds here remains an interesting open question; see Remark \ref{gradient estimate imply lower bound} and Section~\ref{began function lower bound} for further discussion.

 By generalizing results in this paper, it is possible to study the existence problem of cscK metrics. This will be explored in a future paper. Moreover, we note that based on Theorem \ref{algebraic convergence}, it is also possible to study \emph{explicit} moduli compactifications using the moduli continuity method, as initiated in \cite{mabuchi-mukai93,OSS16} and further developed in \cite{spotti-sun18,Liu-Xu19,GMS21,liu2022}; see Sections~\ref{in non-q-gorenstein family} and ~\ref{further remarks} for further discussion.

Following \cite{donaldson14} and using Theorem \ref{algebraic convergence},  we prove the following Zariski openness of  cscK metrics in certain smooth polarized families.  In \cite{donaldson14}, Zariski openness of polarized cscK metrics is established under the additional assumption of a uniform 
Ricci curvature bound, which is not expected to hold for general cscK metrics. Our contribution here is to drop this assumption in complex dimension two.

\begin{theorem}\label{Zariski openness in controlled cone}
Let $(\underline{X},L)\to S$ be a smooth polarized family of surfaces with finite automorphism group. 
Suppose that for some (hence all) $s\in S$, $(X_s,L_s)$ satisfies the numerical condition 
\begin{equation}\label{controlled cone for cscK}
   c_1(X_s)^2 - \frac23 \frac{(c_1(X_s)\cdot c_1(L_s))^2}{c_1(L_s)^2} > 0,\quad  c_1(X_s)\cdot c_1(L_s)>0.
\end{equation} 
Then the following set is Zariski open in $S$:
\[
S^*=\{s\in S \mid (X_s,L_s) \text{ admits a cscK metric in } c_1(L_s)\}.
\] 
\end{theorem}

 With the Zariski openness result in place, and following the framework of \cite{Odaka13}, we obtain the following result on the moduli space.
\begin{theorem}\label{thm-moduli space}
 Smooth polarized cscK surfaces $(X,L)$ with a fixed Hilbert polynomial and satisfying \eqref{controlled cone for cscK} form a separated moduli algebraic space; moreover, this moduli space has only quotient singularities.
\end{theorem}

 For a polarized cscK surface \((X,L)\), we say that \(c_1(L)\) lies in the \textit{controlled cone} if the numerical condition \eqref{controlled cone for cscK} is satisfied. It is an explicit sufficient condition for a priori bound on the Sobolev constant of cscK surfaces \cite{tian92,tian-viaclovsky1,chen-lebrun-weber08}. In fact, the Zariski openness holds as long as we assume a uniform \emph{a priori} upper bound on the Sobolev constant; see Theorem~\ref{Zariski openness under sob bound}.
Note that the non-emptiness of the controlled cone restricts $X$ to be either a blow-up of $\mathbb{P}^2$ or of a Hirzebruch surface. By the results in \cite{chen-cheng,dervan15,hong-won,Cheltsov-MG}, there exist many examples of cscK metrics with polarizations lying in the controlled cone.
 
Such a Zariski openness result for smooth families was proved independently by Donaldson~\cite{donaldson14} and Odaka~\cite{Odaka13} 
for Kähler--Einstein metrics on Fano manifolds with finite automorphism groups.
The proof in \cite{Odaka13} relies on the resolution of the Yau--Tian--Donaldson (YTD)
conjecture for Fano manifolds established in \cite{CDS15}. In this paper, we adapt the differential-geometric approach of \cite{donaldson14}, which avoids such a YTD--correspondence theorem.  Note that the finiteness of the automorphism group is a necessary condition \cite{tian1997,Don-toric}, and under this assumption, in general it is known that the set $S^*$ is open in the Euclidean topology \cite{lebrun-Sim1994}. 

We also remark that the Zariski openness of K-stability for general log Fano pairs 
was proved in \cite{BLX,xu20} using purely algebraic geometry, which were based on the study of the behavior of several algebro-geometric invariants \cite{FO,li18} as well as deep results from birational geometry \cite{birkar19}. In recent years, there have been great advancements on the Yau--Tian--Donaldson conjecture for general polarizations; see \cite{chen-cheng,chen-cheng21,li22,BJ25,darvaszhang,trusiani2026} and the references therein. In particular, as pointed out by Dervan, by the resolution of the uniform version of the YTD conjecture in \cite{trusiani2026}, together with the results of \cite{dervan2025arc}, it is known that, for a smooth polarized family with finite automorphism group, the cscK locus is very general, i.e. a countable intersection of Zariski open subsets.
It would be interesting to see if one can establish Zariski openness algebraically in general, as in the Fano case.

\subsection{Non-collapsing and Sobolev constants bounds}
In our setting, we give a geometric characterization of the uniform Sobolev constant bound by establishing its equivalence with the volume non-collapsing condition.
It is well known that a Sobolev constant bound implies a volume non-collapsing condition \cite[Lemma~2.2]{Hebey2000}. 
Conversely, in the presence of a Ricci curvature lower bound, a uniform local volume non-collapsing condition also guarantees a uniform  bound on the Sobolev constant \cite{croke80}. 
By carrying out a detailed analysis of the bubbling phenomena and building on the work of \cite{tian-viaclovsky2}, 
we prove that the same equivalence holds for polarized cscK surfaces with bounded scalar curvature.
More precisely, let $\widetilde{\mathcal{K}}(V,\kappa)$ denote the class of polarized cscK surfaces satisfying 
\begin{equation}
    \mathrm{Vol}(X,\omega)\leq V, \quad |S_{\omega}|\leq1,\quad
\end{equation}
\begin{equation}\label{assumption--volume non-collapsing}
\vol(B(p,r))\geq \kappa r^4 \text{ for any } p\in X \text{ and } r\in (0,\frac12\mathrm{diam}(X,\omega)].
\end{equation}
\begin{theorem}\label{thm--non-collapsing implies Sob bound}
      There exists a constant $C$ depending only on $V$ and $\kappa$, such that 
    \begin{equation}
        \widetilde{\mathcal{K}}(V,\kappa)\subset \mathcal{K}(V,C).  \end{equation}
\end{theorem}

In particular, Theorem~\ref{algebraic convergence} remains valid for the class $\widetilde{\mathcal{K}}(V,\kappa)$, thereby establishing algebraic convergence for Gromov–Hausdorff convergent sequences in the non-collapsed cscK surface setting, as an analogue of \cite{DS1}.

\

\subsection{Ideas of the proof}
We now briefly outline the ideas behind the proof of Theorem \ref{algebraic convergence} and Theorem \ref{Zariski openness in controlled cone}. The main tool for establishing previous related results is the H\"ormander's $L^2$-estimate, which typically requires a uniform lower bound on Ricci curvature. Therefore, our main difficulty lies in the lack of uniform Ricci curvature bound. A key new ingredient in overcoming this is to encode the dynamic behavior during the formation of singularities. This phenomenon is commonly referred to as bubbling in geometric analysis, and we manage to connect it with complex analytic/algebraic geometry.

For Theorem~\ref{algebraic convergence}, we exploit the special geometry in complex dimension two, namely the limit is a 
polarized orbifold. In particular, it has at worst rational singularities, which allows us to apply the Riemann–Roch formula directly, and the crux is to show that the constant term of the Hilbert polynomial does not jump in the limit. 
This is achieved by a detailed study of the bubble formation, i.e. the energy identity (Proposition~\ref{prop--energy identity}) and the Betti number identities (Proposition~\ref{prop--first betti number preserved}), combined with the  Chern--Gau\ss--Bonnet and signature formulas on orbifolds.

As in \cite{donaldson14}, besides Theorem~\ref{algebraic convergence} and results in \cite{Stoppa2009} and their generalization in \cite{szekelyhidi2015}, the main analytic input for Theorem~\ref{Zariski openness in controlled cone} is a quantitative estimate of the Bergman kernel on the effective regular set, which requires overcoming the lack of a Ricci curvature bound in  this setting.
Recall that the Bergman kernel function, also known as the density of states, is defined by:
\begin{equation}
\varrho_{k,X}=\sum_{\alpha}\vert s_{\alpha}\vert^2,
\end{equation}
where $\{s_{\alpha}\}$ is any $L^2$-orthonormal basis of $H^0(X,L^k)$ and the $L^2$-norm is defined by fiberwise norm on $L^k$ and volume form $(n!)^{-1}(k\omega)^n$ on $X$.

\begin{theorem}\label{effective bergman function bound} For $X\in \mathcal{K}(V, C_S)$, there exists an  constants $\beta, C$ depending only on $V$ and $C_S$ such that 
\begin{itemize}
    \item there exists an exhaustion  $X=\bigcup_{r\in (0,\frac{1}{10})}\Omega_r$ with 
   $ \vol(X\setminus \Omega_r)\leq Cr^4;$
\item for $k\geq r^{-2}(\log r^{-1})^{\beta}$, we have $  \|\rho_k-(1+\frac{S_\omega}{4\pi k})\|_{L^{\infty}(\Omega_r)}\leq Ck^{-2}r^{-4}$ and 
\begin{equation}
\quad \|\nabla\rho_k\|_{L^{\infty}(\Omega_r)}\leq Ck^{-1}r^{-3},\quad \|\nabla^2\rho_k\|_{L^{\infty}(\Omega_r)}\leq Ck^{-1}r^{-4}.
\end{equation}
\end{itemize}
\end{theorem}

To prove Theorem \ref{effective bergman function bound}, we present a novel construction of a quasi-plurisubharmonic (psh) weight function to (partially) carry out H\"ormander's $L^2$-estimate. More precisely, we exploit the ALE nature of bubble limits, where we can find logarithmic growth psh weights that indeed control the Ricci curvature. By scaling down and matching constants, we graft those weights from the whole bubble tree and glue them together to construct quasi-psh
 functions on $X_j$, which effectively control the Ricci curvature. This comes with the cost that these weights must blow up dynamically at a logarithmic rate around bubbling points, so the key point is to uniformly control their blow-up order near singularities. Nevertheless, this already allows us 
to apply Hörmander’s $L^2$-estimate and obtain quantitative asymptotic estimates for the 
Bergman kernel on effective regular region. We emphasize that the requirement on $k$ is stronger than \cite[Proposition~7]{donaldson14}, due to the dynamically 
logarithmic singularities of the weight function near the singular set. It turns out, however, that these estimates are sufficient to 
carry out the argument of \cite{donaldson14} and thereby establish Zariski openness. The uniform estimate of volume of effective singular set can be derived, based on a direct bubbling analysis and the volume estimate established in \cite{tian-viaclovsky0,tian-viaclovsky1}.

As indicated above, the proof of our results above utilizes the special mechanism of Riemannian degeneration, i.e. the bubbling phenomenon. It seems interesting that, at least at a philosophical level, the need from complex analysis to study bubbles might suggest a new 
connection between geometric analysis and algebraic geometry, as pioneered 
in \cite{sun25}; see also \cite{odaka2024,spotti2025, deB-S2024,tazoe2025} and references therein.

\

\subsection{Outline of Paper}
The paper is organized as follows. 
In Section~\ref{sec--prelim}, we collect preliminary results, including a uniform 
$L^2$-bound on the curvature, the orbifold compactness theorem, and convergence 
results for holomorphic sections. Section~\ref{sec--bubble} reviews and supplements 
the bubble analysis for converging sequences of polarized cscK surfaces with a uniform bound Sobolev constants. 
In Section~\ref{convergence in a flat family}, we show that the Hilbert 
polynomial remains unchanged on the limit, thereby realizing the Gromov–Hausdorff convergence 
as convergence in a Hilbert scheme.
In Section~\ref{sec-construction of weightes}, we use the results from Section \ref{sec--bubble} to construct psh functions that control the Ricci curvature and have uniform growth near singular regions, forming the main technical part of the paper. 
In Section~\ref{sec--zariski openness}, we prove 
Theorem~\ref{effective bergman function bound} using the weight functions 
constructed in Section~\ref{sec-construction of weightes}. We then establish 
Theorem~\ref{Zariski openness in controlled cone} by following the argument in 
\cite{donaldson14} and carefully adapting it to our setting.
In Section~\ref{sec-noncollapsing imply sobolev}, we prove 
Theorem~\ref{thm--non-collapsing implies Sob bound} via a contradiction argument, 
relying on bubbling analysis without assuming a Sobolev constant bound. 
In Section~\ref{sec--discussion}, we prove Theorem \ref{thm-moduli space} following the argument in \cite{Odaka13} and discuss some questions around Theorem~\ref{algebraic convergence}. 
We show that the consistency of the Hilbert polynomial is equivalent to the irreducibility of the algebraic limit in projective space—a point emphasized by Donaldson~\cite{donaldson2010}.  
We also outline two possible approaches for constructing cscK metrics that converge in non-\( \mathbb{Q} \)-Gorenstein families.

\subsectionnotoc{Notation and Conventions} 
\begin{itemize}
    \item Throughout the paper, $\Psi\left(\epsilon_1, \ldots, \epsilon_k \mid a_1, \ldots, a_l\right)$ denotes a function such that for fixed parameters $a_i$ we have $\lim _{\epsilon_1, \ldots, \epsilon_k \rightarrow 0} \Psi=0$.
\item In a metric space, we denote by 
\(
A_p(r_0,r_1):=B(p,r_1)\setminus \overline{B(p,r_0)}
\)
the open annulus centered at $p$ with inner radius $r_0$ and outer radius $r_1$.
\item We use $dV_g$ and $d\mu_g$ interchangeably to denote the volume measure associated with a Riemannian metric $g$. When there is no risk of confusion, we will omit the subscript $g$.
\end{itemize}
 
\subsectionnotoc{Acknowledgments} 
 The authors would like to thank Song Sun for suggesting the question, for many helpful discussions, and for his constant support. They also thank Yuji Odaka for many helpful discussions and for his interest in this work, and Simon Donaldson for his valuable feedback on the draft. The first author thanks Ruadhaí Dervan for bringing to his attention the very generality of the cscK locus and for helpful discussions. The second named author thanks Chung-Ming Pan for discussion on singular cscK metrics. This project was essentially completed in the summer of 2025, while both authors were visiting the Institute for Advanced Study in Mathematics (IASM) at Zhejiang University. They are grateful to IASM for its warm hospitality. The second named author was partially supported by the NSF grant DMS-2304692 and the Shoshichi Kobayashi Memorial Fund in UC Berkeley.


\section{Preliminaries}\label{sec--prelim}

\subsection{Matsusaka big theorem and uniform curvature $L^2$-bound}\label{sec-boundedness}
Here we recall the well-known Matsusaka's big theorem:
\begin{theorem}[\cite{matsusaka1,matsusaka2, kollar-matsusaka}]
    For a polarized projective manifold $(X,L)$ of dimension $n$, there exists $m_0$ depending only on $c_1(X)\cdot c_1(L)^{n-1}$ and $c_1(L)^n$, such that $L^m$ is very ample for any $m\geq m_0$.
\end{theorem}

Now according to our assumption \eqref{assumption--volume upper bound and scalar bound}, $c_1(L)^2$ and $c_1(X)\cdot c_1(L)$ are uniformly bounded in $\mathcal{K}(V, C_S)$. As a result, there exists $k_1$ such that $L^{k_1}$ is very ample for all $(X,L)\in\mathcal{K}(V,C_S)$.
We may replace $L$ with $L^{k_1}$ and therefore in the following discussion, we just assume $k_1=1$. 
As a consequence of \cite[3.28.9]{kollar2013}, we know that 
\begin{equation}\label{eq-finite diffeo-type}
     \text{  $\mathcal{K}(V, C_S)$ contains finitely many diffeomorphism types}
\end{equation}
 and we obtain that there exists $\Lambda>0$ such that for any $X\in \mathcal{K}(V,C_S)$, we have
\begin{equation}
\vert c_1(X)^2\vert+\vert c_2(X)\vert\leq\Lambda.
\end{equation}
Now by the well-known result of Calabi \cite{Calabi1}, we have
\begin{equation}\label{c1 and ric L2}
    \begin{aligned}
    \int_X(\vert \Ric_{\omega}\vert^2-S_{\omega}^2)\omega^2&=-8\pi^2c_1(X)^2,\\    
    \int_X(\vert \Rm_{\omega}\vert^2-\vert \Ric_{\omega}\vert^2)\omega^2&=8\pi^2(2c_2(X)-c_1(X)^2).\end{aligned}
\end{equation}
Here we use the convention that 
\begin{equation}\label{eq-convention of scalar curvature}
    \Ric_{\omega}=-\ii\partial\bar\partial \log \omega^2, \quad S_{\omega}=\tr_{\omega}(\Ric_{\omega})=\frac{2\,\Ric_{\omega}\wedge\omega}{\omega^2}.
\end{equation} In particular, $[\Ric_{\omega}] \in 2\pi c_1(X)$ in 
$H^2(X,\mathbb{R})$.
Therefore, the assumption \eqref{assumption--volume upper bound and scalar bound} implies that there exist a constant $\Lambda$ depending only on $V$ and $C_S$ such that for any $X\in \mathcal{K}(V,C_S)$, we have  
\begin{equation}\label{eq--uniform curvature L2 bound}
\Vert \Rm \Vert_{L^2}\leq \Lambda.
\end{equation}

Since $(X,L)\in \mathcal{K}(V,C_S)$ admits only finitely many possible Hilbert polynomials and finitely many diffeomorphism types, given a sequence $(X_j,L_j)$ we may, after passing to a subsequence, assume that all $(X_j,L_j)$ share the same Hilbert polynomial and that the $X_j$ are diffeomorphic.
\subsection{Orbifold compactness}\label{orbifold compactness}
Here we review the Riemannian convergence theory for non-collapsed cscK surfaces developed by Tian-Viaclovsky \cite{tian-viaclovsky0,tian-viaclovsky1} (see also  \cite{anderson05}). Their results apply to a broader class of Riemannian $4$-manifolds, but we restrict to K\"ahler setting in our paper.

First, we clarify the notion of convergence used in this paper. 
Let $M_\infty$ be a compact Riemannian orbifold with isolated singularities. 
We say that a sequence of Riemannian manifolds $(M_j,g_j)$ converges to 
$(M_\infty,g_\infty)$ in the $\hat C^{\infty}$--Cheeger--Gromov sense if: 
\begin{itemize}
    \item  $(M_j,d_j)$ converges to $(M_\infty,d_\infty)$ in the Gromov--Hausdorff sense; and therefore there exists a distance function $\mathbf{d}_j$ on $M_j \sqcup M_\infty$ 
such that $M_j$ and $M_\infty$ are $\delta_j$--close in the Hausdorff sense, 
with $\delta_j \to 0$;

\item
Given any $\delta>0$ and any compact subset 
$K \subset M_\infty^{\mathrm{reg}}$, there exist (for $j$ sufficiently large) 
open embeddings 
\[
\chi_j : U \to M_j
\]
from an open neighborhood $U \supset K$ into $M_j$ such that  
\begin{itemize}
    \item $\mathbf{d}_j(\chi_j(x),x) \leq \delta$ for all $x\in K$, and  
    \item $\chi_j^* g_j \to g_\infty$ in $C^{\infty}$ over $K$.  
\end{itemize}
\end{itemize}
The notion of pointed $\hat C^{\infty}$--Cheeger--Gromov convergence is defined 
similarly by considering larger and larger metric balls centered at chosen base points.

In the polarized case, we also include the convergence of complex structures and connections on the line bundles \cite{DS1}. A polarized K\"ahler orbifold is a compact Riemannian orbifold $(M_{\infty},g_{\infty})$ with extra data $(J_{\infty},\omega_{\infty},L_{\infty},A_{\infty})$, where 
\begin{itemize}
    \item $J_{\infty}$ is complex structure on $M_{\infty}^{reg}$ with respect to which the metric is K\"ahler with K\"ahler form $\omega_{\infty}$;
    \item $L_{\infty}$ is a hermitian holomorphic line bundle over $M_{\infty}^{\mathrm{reg}}$;
    \item $A_{\infty}$ is a metric connection on $L_{\infty}$ with  $\omega_{\infty}=\ii \Theta_{A}$.
\end{itemize}
    Since $M_{\infty}$ has only orbifold singularity, it is $\mathbb Q$-factorial \cite[Proposition 5.15]{KM}. Hence $L_\infty$ extends as a $\mathbb{Q}$-line bundle on $X_\infty$. In particular, $X_\infty$ is a projective variety by \cite{baily}.

We say $(M_j, g_j, J_j, \omega_j, L_j,A_j)$ converges to $(M_{\infty}, J_{\infty}, \omega_{\infty},L_{\infty},A_{\infty})$ in the polarized $\hat C^{\infty}$-Cheeger-Gromov sense,
if  for any $\delta>0$ and compact subset $K \subset M_{\infty}^{\mathrm{reg}}$ we have maps $\chi_j$ as before but in addition so that the pulled back complex structures $\chi_j^*\left(J_j\right)$ converge to $J_{\infty}$ in $C^{\infty}$ over $K$ and we have smooth bundle isomorphisms $\widehat{\chi}_j: L_{\infty} \rightarrow \chi_j^*\left(L_j\right)$ with respect to which the connections $\chi_j^*(A_j)$ converge smoothly. As in the preceding discussion, the notion of pointed polarized 
$\hat C^{\infty}$-Cheeger--Gromov convergence can also be defined.

Throughout this paper, when we say that tensors on $M_j$ converge in the $C^{\infty}_{loc}$-sense, we mean that they are pulled back via the maps $\chi_j$ (and possibly also $\hat{\chi}_j$), so that, as tensors on $M_{\infty}$, defined on an exhausting sequence of subsets of $M_{\infty}^{\mathrm{reg}}$,  they converge in $C^{\infty}_{loc}(M_{\infty}^{\mathrm{reg}})$.

By locally trivializing the line bundles using sections with controlled norm and taking limits of the transition functions, we can obtain the following  well known result.

\begin{lemma}
    Suppose a sequence of polarized K\"ahler surfaces converges in the 
    $\hat C^{\infty}$-Cheeger--Gromov sense. Then, after passing to a subsequence, 
    it converges in the polarized $\hat C^{\infty}$-Cheeger--Gromov sense.  
\end{lemma}

The main result in \cite{tian-viaclovsky1,tian-viaclovsky2} is a
compactness theory for non-collapsed cscK surfaces. Recall that, by the discussion
in the previous section (see \eqref{eq-finite diffeo-type}--\eqref{eq--uniform
curvature L2 bound}), elements in $\mathcal{K}(V,C_S)$ have only finitely many
diffeomorphism types and satisfy a uniform $L^2$-curvature bound.
\begin{theorem}\label{thm--TVcompactness}
  For a sequence $(X_j,\omega_j) \in \mathcal{K}(V,C_S)$, after passing to a subsequence, 
$X_j$ converges in the polarized $\hat{C}^{\infty}$-Cheeger--Gromov sense to a polarized 
cscK orbifold $X_\infty$. Moreover, both the number of orbifold points and the orders 
of the local orbifold groups are bounded by a quantity depends only $V$ and $C_S$.
\end{theorem}


A key ingredient in the proof of Theorem \ref{thm--TVcompactness} is the following $\epsilon$-regularity and a volume ratio upper bound estimate.

\begin{theorem}[\cite{tian-viaclovsky1,tian-viaclovsky2}]\label{thm-volume bound and epsilon regularity}
Given $V, C_{S}$, there exists constant $V_1$ and $\epsilon_0$ such that for $X\in \mathcal{K}(V,C_S)$, we have
\begin{itemize}
    \item for any $p\in X$ and $r>0$, $\vol(B(p,r))\leq V_1r^4$;
    \item if $\left(\int_{B(p,r)}|\mathrm{Rm}|^2 \,dV_g\right)^{\frac{1}{2}}\leq \epsilon_0$ for some $r\in (0,\mathrm{diam}(X))$, then 
    \begin{equation}
        \sup_{B(p,r/2)}|\mathrm{Rm}|\leq \frac{V_1}{r^2}\left(\int_{B(p,r)}|\mathrm{Rm}|^2 \, dV_g\right)^{\frac{1}{2}}.
    \end{equation}
\end{itemize}
\end{theorem}

Note that an important corollary of Theorem \ref{thm-volume bound and epsilon regularity} is that volume convergence holds for $\hat{C}^{\infty}$-Cheeger--Gromov convergence, which will be used later.

\begin{definition}\label{def-regularity scale}
     For $x \in X$ and $X\in \mathcal{K}(V,C_S)$, we define the regularity scale $r_x$ and $\widetilde{r}_x$ by
$$
r_x := \max \left\{r\mid \sup _{B_r(x)}|\mathrm{Rm}| \leq r^{-2}\right\}, \quad \widetilde{r}_x:=\min\left\{r\mid \left(\int_{B_r(x)}|\mathrm{Rm}|^2\,dV_g\right)^{\frac{1}{2}}\geq \epsilon_0\right\}.$$
    \end{definition}

As a consequence of Theorem \ref{thm-volume bound and epsilon regularity}, we know that there exists a constant $C$ depending only on $V,C_S$ such that for any $X\in \mathcal{K}(V,C_S)$ and $x\in X$, we have
\begin{equation}\label{eq-two regularity comparable}
    \frac{1}{{C}}r_x\leq \widetilde{r}_x\leq {C}r_x.
\end{equation}

\subsection{Sobolev inequality}\label{sec-sobolev inequality}

\begin{definition}\label{def--sobolev constant}
    The Sobolev constant of a closed Riemannian $4$-manifold $(M,g)$ is defined to be the best constant $C_g$ such that 
    \begin{equation}
        \Vert f\Vert_{L^4}^2\leq C_g(\Vert \nabla f\Vert_{L^2}^2+\mathrm{Vol}(M,g)^{-1/2}\Vert f\Vert_{L^2}^2)
    \end{equation}
    for all $f\in W^{1,2}(M)$.
\end{definition}

\begin{definition}\label{def-local sobolev constant}
    The local Sobolev constant of a Riemannian $4$-manifold $(M,g)$ at scale $r$ is defined to be the best constant $C_S(r)$ such that
    \begin{equation}
        \Vert f\Vert_{L^4(B(p,r))}\leq C_S(r)\Vert \nabla f\Vert_{L^2(B(p,r))}
    \end{equation}
    for all $p\in M$ and all $f\in W_0^{1,2}(B(p,r))$.
\end{definition}

There are certain cases that Sobolev constant can be \textit{a prior} bounded for cscK  metrics. This is first discovered in \cite{tian92} and later is generalized in \cite{chen-lebrun-weber08}. We have a functional $\mathcal{A}([\omega])$ defined on the K\"ahler cone \cite{chen-lebrun-weber08}:
\begin{equation}
    \mathcal{A}([\omega])=\frac{\left(c_1 \cdot[\omega]\right)^2}{[\omega]^2}-\frac{1}{32 \pi^2} \mathcal{F}(\xi,[\omega]),
\end{equation}where $\xi$ is the extremal vector field for the class $[\omega]$, which can be determined without knowing that an extremal metric exists \cite{futaki-mabuchi}, and $\mathcal{F}$ denotes the Futaki invariant.

\begin{definition}\label{controlled cone}
    We say a K\"ahler class $[\omega]$ is inside the controlled cone if
    \begin{equation}
       3 c_1(X)^2-2\mathcal{A}([\omega])>0 \quad \text {and }\quad c_1(X)\cdot [\omega]>0.
    \end{equation}
\end{definition}

\begin{lemma}[\cite{tian-viaclovsky1,chen-lebrun-weber08}]\label{contolled  cone imply sobolev bound}
    Let $(X,\omega)$ be an compact extremal K\"ahler surface. Suppose $[\omega]$ is inside the controlled cone and $S_{\omega}\geq 0$. Then the Sobolev constant has the following bound
    \begin{equation}
 C_{\omega}\leq \frac{\max\{6,\Vert S_{\omega}\Vert_{L^{\infty}} \vol(X,\omega)^{\frac{1}{2}}\}}{\sqrt{32\pi^2(3c_1^2(X)-2\mathcal{A}([\omega]))}}.
    \end{equation}
\end{lemma}

Note that in our setting when $\omega$ has constant scalar curvature, if $[\omega]$ lies in the controlled cone, its scalar curvature is automatically positive. Also, the condition of lying 
in the controlled cone reduces to the numerical condition \eqref{controlled cone for cscK}. Consequently, for a smooth polarized family of cscK surfaces $(\underline{X},L)\to S$, 
if $c_1(L_s)$ lies in the controlled cone for some $s \in S$, 
then $c_1(L_{s'})$ lies in the controlled cone for all $s' \in S$.


\subsection{Convergence of holomorphic sections}\label{convergence of holomorphic sections}

Here we review some standard facts about convergence of holomorphic sections. When working with the line bundle $L^k$, $k\in \mathbb N$, we use the notation $L^{2,\#}$, etc. to denote norms defined by rescaled metrics $k\omega$ as in \cite{DS1}.

\begin{proposition}\label{L-infinity estimate}
    The following holds:
    \begin{itemize}
        \item [(1)] There exists  constant $K_0$ depending only on $V$ and $C_S$, such that if $X\in\mathcal{K}(V,C_S)$ and $s$ is a holomorphic section of $L^k$, we have
        \begin{equation}\label{eq--global c0 bound}
            \Vert s\Vert_{L^{\infty,\#}}\leq K_0\Vert s\Vert_{L^{2,\#}}.
        \end{equation}
        \item [(2)] Suppose $X_j$ is a sequence in $\mathcal{K}(V,C_S)$ converging to an orbifold $X_\infty$. Then for any interior region $\Omega_\infty\subset\subset X^{\mathrm{reg}}_\infty$ and regions $\Omega_j\subset\subset X_j$ converging to $\Omega_\infty$, there exists a constant $K_1(\Omega_\infty)$, such that for any holomorphic section $s$ of $L_j^k$, we have
        \begin{equation}
            \Vert \nabla s\Vert_{L^{\infty,\#}(\Omega_j)}\leq K_1(\Omega_\infty)\Vert s\Vert_{L^{2,\#}}.
        \end{equation}
    \end{itemize}
\end{proposition}

\begin{corollary}
    Given any sequence of collection of holomorphic sections $s_j^{(i)}\in H^0(X_j,L_j^k)$ with $\langle s_j^{(m)},s_j^{(l)}\rangle_{L^{2,\#}}=\delta_{ml}$. By passing to some subsequence, it converges in $C^{\infty}_{loc}$ to a collection of holomorphic section $s_\infty^{(i)}\in H^0(X_\infty,L_\infty^k)$ with $\langle s_\infty^{(m)},s_\infty^{(l)}\rangle_{L^{2,\#}}=\delta_{ml}$. In particular, all limit sections are non-zero.
\end{corollary}
\begin{proof}
    This is standard. We point it out to highlight the fact that it uses the uniform  \textit{global} $L^{\infty}$-bound \eqref{eq--global c0 bound} and volume estimate in Theorem \ref{thm-volume bound and epsilon regularity}. 
\end{proof}

By a diagonal argument, given any sequence of orthonormal bases of $H^0(X_j,L_j^{k})$, 
after passing to a subsequence we may assume that it converges both in $C^{\infty}_{\mathrm{loc}}$ 
and in $L^2$ to an orthonormal set, which linearly independently spans a subspace of $H^0(X_\infty,L_\infty^{k})$. 
In particular, 
by ampleness of $L_j,L_\infty$, we obtain the following estimate:
\begin{corollary}\label{cor--dimension non-decreasing}
    For any $j$ large and sufficiently large divisible $k$, we have 
    \begin{equation}\label{dim non-increasing}
   \chi(X_j,L_j^k) \;\leq\; \chi (X_\infty,L_\infty^k).
\end{equation}
\end{corollary}

\subsection{Cohomology and intersection theory on orbifolds}\label{sec-orbifold-intersection-numbers}
We recall some standard results from the Chern--Weil theory on orbifolds \cite{satake1956,satake1957,kawasaki1979} and the usual algebraic intersection theory of Cartier divisors \cite[Chapter~2]{fulton2013}. See also \cite[Section 2]{lebrun2015edge} for a summary. By an orbifold smooth form, we mean a form on the regular locus whose pull-back to each local uniformizing chart extends to a smooth invariant form.

We first clarify different cohomology theory on orbifolds. Recall that for any topological 
space $X$, we have the singular homology/cohomology, and the $i$-th Betti number $b_i(X)$ denotes the
dimension of the $i$-th singular homology group with coefficients in 
$\mathbb{R}$. By the universal coefficient theorem, this coincides with the 
dimension of the $i$-th singular cohomology group. 
Since an orbifold 
is locally contractible, the singular cohomology groups of $X$ agree with the sheaf 
cohomology groups with coefficients in the constant sheaf 
$\underline{\mathbb{R}}$ \cite[Theorem 4.47]{Voisin_2002}. Moreover, by \cite{satake1956, Baily1956}, for a compact orbifold, this cohomology group is also isomorphic to the orbifold de Rham cohomology and also to the space of orbifold harmonic forms.
There is a great advantage of using different perspectives. For example, the Hodge decomposition and Hodge index theorem for K\"ahler surfaces can be easily extended to orbifold setting \cite{Baily1956}, since the analytic foundation of Hodge theory is essentially the same in the orbifold category, and as we see below, the intersection form can also be identified by integral of differential forms over orbifolds. 

Let $X$ be a projective orbifold surface. Let $E$ and $F$ be holomorphic orbifold line bundles on $X$. We use the same symbols $E$ and $F$ for the induced $\mathbb Q$-Cartier divisor classes on $X$. More explicitly, if $r_E,r_F>0$ are chosen so that $E^{\otimes r_E}$ and $F^{\otimes r_F}$ descend to  line bundles, i.e. Cartier divisors, on $X$, then, as $\mathbb Q$-Cartier divisor classes,
\[
    E=\frac{1}{r_E}[E^{\otimes r_E}],\qquad
    F=\frac{1}{r_F}[F^{\otimes r_F}].
\]
Choose orbifold Hermitian metrics $h_E$ and $h_F$, and set
\[
    \alpha_E=\frac{\sqrt{-1}}{2\pi}\Theta_{h_E}(E),\qquad
    \alpha_F=\frac{\sqrt{-1}}{2\pi}\Theta_{h_F}(F).
\]
The following result is standard and will be used repeatedly throughout the paper without further mention.
\begin{lemma}\label{lem-intersection}
  We have
\begin{equation}\label{eq--orbifold-intersection-by-forms}
    E\cdot F=\int_X \alpha_E\wedge\alpha_F .
\end{equation}
Here $E\cdot F$ is the algebraic intersection number of the two $\mathbb Q$-Cartier divisor classes on $X$ as a normal projective surface. 
\end{lemma}

In particular, if $L$ is a orbifold line bundle with an orbifold smooth curvature form $\omega=\sqrt{-1}\Theta_h(L)$, then
\begin{equation}\label{eq--orbifold-L-square-by-volume}
    L^2=\frac{1}{(2\pi)^2}\int_X\omega^2.
\end{equation}
If $\omega$ is a K\"ahler orbifold metric, then
\begin{equation}\label{eq--orbifold-KL-by-scalar}
    K_X\cdot L=-\frac{1}{(2\pi)^2}\int_X\mathrm{Ric}_{\omega}\wedge\omega.
\end{equation}

\begin{proof}[Proof of Lemma \ref{lem-intersection}]
By the linear property on both sides. It is sufficient to consider the case that $E$ and $E$ themselves are Cartier divisors.

 Let $\pi:Y\to X$ be a resolution. By definition of the algebraic intersection of Cartier divisors on a normal surface,
\[
    E\cdot F=(\pi^*E\cdot\pi^*F)_Y .
\]
Choose Hermitian metrics $h_E^0$ and $h_F^0$ on the line bundles $E$ and $F$ which are smooth on the complex space $X$, for instance by using local embeddings into smooth ambient spaces. Let their Chern forms on $X^{\mathrm{reg}}$ be
\[
    \beta_E=\frac{\sqrt{-1}}{2\pi}\Theta_{h_E^0}(E),\qquad
    \beta_F=\frac{\sqrt{-1}}{2\pi}\Theta_{h_F^0}(F).
\]
Then $\pi^*\beta_E$ and $\pi^*\beta_F$ extend smoothly across the exceptional divisor and are Chern forms of smooth Hermitian metrics on $\pi^*E_0$ and $\pi^*F_0$. Since $Y$ is smooth, Chern--Weil theory gives
\[
    (\pi^*E_0\cdot\pi^*F_0)_Y
    =\int_Y \pi^*\beta_E\wedge\pi^*\beta_F .
\]

On $X^{\mathrm{reg}}$, the two metrics $h_E$ and $h_E^0$ are metrics on the same line bundle $E|_{X^{\mathrm{reg}}}$ and hence differ by a function; similarly for $F$. Thus there are functions $\varphi_E$ and $\varphi_F$, which are smooth on $X^{\mathrm {reg}}$ and globally bounded on $X$ such that
\[
    \alpha_E=\beta_E+\ii\partial\bar\partial\varphi_E,\qquad
    \alpha_F=\beta_F+\ii\partial\bar\partial\varphi_F.
\] 

Then the desired equality follows from an integration by part argument. Choose logarithmic cut-off functions $\chi_\varepsilon$ with respect to an orbifold smooth K\"ahler metric on $X$ such that $0\leq\chi_\varepsilon\leq 1$, $\chi_\varepsilon=0$ in an $\varepsilon^2$-neighborhood of $D$, $\chi_\varepsilon=1$ outside an $\varepsilon$-neighborhood of $D$, and as $\epsilon\rightarrow 0$, 
\[
    \|\nabla\chi_\varepsilon\|_{L^2(X)}+\|\ii\partial\bar\partial\chi_\varepsilon\|_{L^1(X)}\longrightarrow 0 .
\]
The norm is taken with respect to the orbifold smooth K\"ahler metric on $X$. If $\gamma$ is either $\beta_E$ or $\beta_F$ which is smooth on $Y$, then Stokes' theorem gives
\[
    \int_X\chi_\varepsilon\,\ii\partial\bar\partial u_E\wedge\gamma
    =\int_X u_E\,\ii\partial\bar\partial\chi_\varepsilon\wedge\gamma ,
\]
because $\gamma$ is closed. The right hand side tends to zero, since $u_E$ is bounded and $\|\ii\partial\bar\partial\chi_\varepsilon\|_{L^1(X)}\to 0$. Letting $\varepsilon\to 0$ therefore gives
\[
    \int_X \ii\partial\bar\partial u_E\wedge\beta_F=0 .
\]
The same argument gives $\int_X\beta_E\wedge \ii\partial\bar\partial u_F=0$. For the remaining term, apply the same cut-off argument with the closed form $\ii\partial\bar\partial u_F$ on:
\[
    \int_X\chi_\varepsilon\,\ii\partial\bar\partial u_E\wedge \ii\partial\bar\partial u_F
    =\int_X u_E\,\ii\partial\bar\partial\chi_\varepsilon\wedge \ii\partial\bar\partial u_F .
\]
Since $\ii\partial\pp u_F$ is the difference of an orbifold smooth form and a smooth form, it has bounded norm with respect to the orbifold K\"ahler metric. Letting $\epsilon\rightarrow 0$ and using $\|\ii\partial\bar\partial\chi_\varepsilon\|_{L^1(X)}\to 0$, we obtain
\[
    \int_Y \ii\partial\bar\partial u_E\wedge \ii\partial\bar\partial u_F=0 .
\]
Consequently,
\[
    \int_X\alpha_E\wedge\alpha_F=
  \int_X \beta_E\wedge\beta_F
    =E\cdot F.
\]
\end{proof}



\section{Bubbling of cscK metrics}\label{sec--bubble}
In this section we collect some basic facts about bubbling of cscK metrics, which will be used later. Most of the results are standard and have been proved in \cite{bando1990a,bando1990b,tian-viaclovsky1}.
We include some details here to highlight certain properties needed later.
Moreover, in our polarized setting, we are able to derive some stronger conclusions.
 
 First we clarity what we mean by \textit{curvature singularity}. Given a sequence $(X_j,\omega_j)$ inside $\mathcal{K}(V,C_S)$, converging to $(X_\infty,\omega_\infty)$ in $\hat C^{\infty}$-Cheeger–Gromov sense, we say a point $p_\infty\in X_\infty$ is a curvature singularity, if for any $r>0$ and any $p_j\to p_\infty$ in Gromov-Hausdorff topology, we have 
\begin{equation}\label{eq--def of curvature sing}
\liminf_{j\to\infty}\int_{B(p_j,r)}\vert \Rm_{g_j}\vert^2d\mu_{g_j}>\epsilon_0.
\end{equation}
 We denote by $S(X_\infty)$ the set of curvature singularities. It is clear that $S(X_\infty)$ is discrete and contains $X^{\mathrm{sing}}_\infty$, the set of orbifold points of $X_\infty$. In general, $S(X_\infty)$ can be \emph{strictly} larger: there may exist points arising from the convergence locally modeled on rescalings of the Burns metric. 
Later, as an application of the bubbling analysis, we will show that under our \emph{polarized} assumption, one in fact has
$ S(X_\infty)=X^{\mathrm{sing}}_\infty$; see Theorem \ref{thm-characterization of orbifold points}. During the proof of the bubble decomposition Theorem \ref{thm-bubble decomposition}, we allow for the possibility that some of them may in fact correspond to regular points of the limit.

The analysis of bubbles is local, so we can focus on a small neighborhood of one curvature singularity $p_{\infty}$.  To improve clarity and ease the notation, we need the notion of a rooted tree.
Let $\mathcal{T}$ be a rooted tree with a finite number of vertices. The vertices of $\mathcal{T}$ are partially ordered such that $I \leq  J$ if and only if the unique path from the root to $J$ passes through $I$. 
If $I\leq J$ and $I\neq J$, we denote this by by $I<J$. The depth of a vertex, denoted by  $\mathrm{dep}(I)$ is defined to be the length of the path to its root.  In particular, the root has depth zero. The depth of a tree is the maximum depth of any vertex.
For any vectex $I$, we let $C(I)$ denote its set of children; if it is not the root, we let $P(I)$ denote its unique parent.


The following theorem is proved in \cite{bando1990a,bando1990b,tian-viaclovsky1}.

\begin{theorem}[\cite{bando1990a,bando1990b,tian-viaclovsky1}]\label{thm-bubble decomposition}
   Passing to a subsequence, there exist a rooted tree with finitely many vertices \( \mathcal{T} \) and a constant \( \eta > 0 \) such that, for each \( j \) and each vertex \( I \in \mathcal{T} \), there exist points \( p_{j,I} \in X_j \) with \( p_{j,I} \to p_{\infty} \), scaling factors \( \lambda_{j,I} \to \infty \) as \( j \to \infty \), and constants \( \delta_I \in (0,1) \) satisfying the following properties:
    \begin{itemize}
    \item for the root vertex $I$,  we denote $p_{j,I}=p_{j,1}$ and we make the convention that $\lambda_{j,P(1)}=1$, then for any vertex $I\in \mathcal{T}$, we have 
    \begin{equation}\label{curvature estimate on annulus}
        \int_{A_{p_{j,I}}(\lambda_{j,I}^{-1},\delta_{I}\lambda_{j,P(I)}^{-1})}|\mathrm{Rm}_{g_j}|^2 d\mu_{g_j}\leq \epsilon_0,
    \end{equation}
        \item for any $I<J$, we have $p_{j,J}\in B(p_{j,I}, \lambda_{j,I}^{-1})$ and $\lim_{j\rightarrow \infty} \frac{\lambda_{j,I}}{\lambda_{j,J}}=0,$ 
        \item for any $I<J$, we have either $p_{j,J}=p_{j,I}$ or $\lim_{j\rightarrow \infty}\lambda_{j,I}d_j(p_{j,J},p_{j,I})\geq \eta>0$,
        \item for any $I$,  $(X_j, \lambda_{j,I} ^2\omega_j, p_{j,I})$ converges to a Stein ALE scalar-flat non-flat K\"ahler orbifold surface $(Z_{I},\omega_{Z_{I}}, p_{Z_I})$ in the pointed $\hat C^{\infty}$-Cheeger–Gromov sense.
        \item if a vertex $I$ has no children, then  $(X_j, \lambda_{j,I} ^2\omega_j, p_{j,I})$ converges to a smooth Stein ALE scalar-flat non-flat K\"ahler surface $(Z_{I},\omega_{Z_{I}}, p_{Z_I})$ in the pointed $\hat C^{\infty}$-Cheeger–Gromov sense.
        \item for any $I$, $C(I)$ has a natural one-to-one correspondence to the singular set $S(Z_{I})$. Moreover, if $J\in C(I)$ corresponds to a point $q_J\in S(Z_I)$, then the asymptotic cone of $(Z_J,\omega_J)$ is isomorphic to the local tangent cone of $Z_I$ at $q_J$.
    \end{itemize}
\end{theorem}

Note that the precise meaning of ``asymptotically Euclidean (AE)'' or ``asymptotically locally Euclidean (ALE)'' may vary in the literature. 
Since we are only concerned with scalar-flat Kähler metrics, we adopt the following convention: 
there exist asymptotic coordinates and constant $\tau>1$ in which, for any integer $k\ge 0$,
\[
\sum_{\ell=0}^k |x|^\ell\bigl|\nabla^\ell(g_{ij}-\delta_{ij})\bigr|
=O\bigl(|x|^{-\tau}\bigr) \qquad (|x|\to\infty).
\]
This fast fall-off condition can be obtained by combining \cite[Proposition 5.2]{tian-viaclovsky0} and \cite[Theorem 1.1]{BKN1989}. In fact, $\tau$ can be chosen arbitrarily close to $2$, although we will not use this fact. For the ALE case, we require a similar estimate holds on the covering chart of the end.

In the theory of orbifold compactness of Einstein metrics, it's known that a point in the limit space is singular if and only if it's a curvature singularity. The proof uses Bishop-Gromov volume comparison  \cite{anderson90, cheeger-colding97}, thus relies crucially on the bound on Ricci curvature. It is well-known that this phenomenon fails for general cscK metrics due to the existence of Burns metric \cite{lebrun1991}. As an application of bubbling analysis, we prove this equivalence under our \textit{polarized} assumption. For each $I\in\mathcal I$, denote by $\mathcal C_I\coloneqq\mathbb C^2/\Gamma_I$ the asymptotic cone of $Z_I$. Then we have:
\begin{theorem}\label{thm-characterization of orbifold points}
    A point $p_\infty\in X_\infty$ is an orbifold point if and only if it is a curvature singularity. The same conclusion holds for any blow-up limits. In particular, for each bubble limit $Z_I$, the orbifold group $\Gamma_I\neq \mathrm{Id}$.
\end{theorem}

The last statement in Theorem \ref{thm-characterization of orbifold points}, combining the curvature estimate Proposition \ref{prop-curvature estimate on neck region}, allows us to construct constant mean curvature (CMC) hypersurfaces foliations along neck region as in \cite{ozuch2022}, thus obtaining the following strong conical approximation:

\begin{theorem}\label{neck structure theorem}
     For each $I\in \mathcal{T}$, there exists a flat cone $\mathbb{R}^4/\Gamma_I$ for some finite subgroup $\Gamma_I\subset SO(4)$ acting freely away from the origin and there exist $\delta>0$, $1<K_0<K_1$, such that we can find smooth embeddings
   \begin{equation}
       \Phi_{j,I}:A_{p_{j,I}}(K_0\lambda^{-1}_{j,I},K_0{(\lambda_{j,I}\lambda_{j,P(I)})}^{-\frac{1}{2}})\to \mathbb R^4/\Gamma_I,
   \end{equation} 
   and
   \begin{equation}
       \Psi_{j,P(I)}:A_{p_{j,I}}(K_1^{-1}{(\lambda_{j,I}\lambda_{j,P(I)})}^{-\frac{1}{2}},K_1^{-1}\lambda_{j,P(I)}^{-1})\to \mathbb R^4/\Gamma_I,
   \end{equation} 
   such that for $k\in \mathbb N$, we have estimates:
   \begin{equation}
    \left| \nabla^k\left((\Phi_{j,I}^{-1})^*(\lambda_{j,I}^2g_j)-g_{\Gamma_I}\right)\right|_{g_{\Gamma_I}}\leq C_k r_I^{-k-\delta},
   \end{equation}
    \begin{equation}
      \left|\nabla^k\left((\Psi_{j,P(I)}^{-1})^*(\lambda_{j,P(I)}^{2}g_j)-g_{\Gamma_I}\right)\right|_{g_{\Gamma_I}}\leq C_kr_{I}^{-k+\delta}.
   \end{equation}
\end{theorem}

The upshot here is that there is a polynomial rate of convergence to the model cone metrics. This rate will be crucial for our construction of psh weights in Section \ref{sec-construction of weightes}.

\subsection{Proof of Theorem \ref{thm-bubble decomposition}}
 The proof is based on inductive on scales. We first construct a minimal bubble, i.e. a bubble limit which is not a cone, but its asymptotic cone is isomorphic to the local tangent cone of $X_\infty$ at $p_\infty$. This corresponds to the root vertex in our notation. The same construction also works for smaller scales, and the termination of bubble tree follows from energy condition.

Let $\epsilon_0$ be the small constant appeared in Theorem \ref{thm-volume bound and epsilon regularity}, i.e the constant such that the $\epsilon$-regularity holds. Fix an $\epsilon\in (0, \epsilon_0]$ such that the gap theorem in \cite[Corollary 4.5]{chen-weber11} holds. Let $p_\infty\in S(X_\infty)$ be a curvature singularity. We first choose $5r_\infty>0$, such that 
\begin{equation}\label{eq-isolated curvature singularity}
   B(p_\infty,5r_\infty)\cap S(X_\infty)=\{p_\infty\} \text{ and } \int_{B(p_\infty,5r_\infty)}\vert \Rm_{g_\infty}\vert^2d\mu_{g_\infty}\leq\frac{\epsilon}{2}
\end{equation}
where $\epsilon$ is a constant to be determined below. By definition of $S(X_\infty)$, we can choose $p_j^{\prime}\to p_\infty$ such that 
$\sup_{B(p_j^{\prime},5r_\infty)}\vert \Rm_{g_j}\vert\to\infty.$
Let $p_j\in B(p_j^{\prime},5r_\infty)$ be the point such that 
\begin{equation}
    \vert \Rm_{g_j}\vert(p_j)=\sup_{B(p_j^{\prime},5r_\infty)}\vert \Rm_{g_j}\vert,
\end{equation} then by \eqref{eq-isolated curvature singularity} and the definition of curvature singularities \eqref{eq--def of curvature sing}, we must have that in the Gromov-Hausdorff sense 
    $p_j\to p_\infty.$
Therefore for $j$ large we may assume $B(p_j,r_\infty)\subset B(p_j^{\prime},5r_\infty)$. In particular, $$\vert \Rm_{g_j}\vert(p_j)=\sup_{B(p_j,r_\infty)}\vert \Rm_{g_j}\vert,$$ i.e. $p_j$ is the point with the largest curvature near its neighborhood. This fact will be used later in the proof of termination of bubble trees.

We next choose $r_{j}$ such that
\begin{equation}\label{equa-choice of energy scale}
    \int_{A_{p_j}(r_j,r_\infty)}\vert \mathrm{Rm}_{g_j}\vert^2d\mu_{g_j}=\epsilon,
\end{equation}
then $r_j\to 0$. Let $\overline\lambda_j=r_{j}^{-1}$ and $\underline{\lambda}=r_\infty^{-1}$. According to compactness theory in \cite{tian-viaclovsky1}, $(X_j,\overline\lambda_{j}^2\omega_j,p_j)$ converges to an ALE scalar flat K\"ahler orbifold $(Z,\omega_Z, p_Z)$ in the pointed $\hat{C}^{\infty}$-Cheeger-Gromov sense. Our goal is to prove that $(Z,\omega_Z, p_Z)$ is the desired minimal bubble associated to this pointed sequence.

\begin{lemma}\label{lem-no intermediate scales}
    Given any sequence $\lambda_j\in[\underline{\lambda},\overline\lambda_j]$ with $\lambda_j\to\infty$ and $\lambda_j^{-1}\overline{\lambda}_j\to\infty$, any subsequential limit of $(X_j,\lambda_j^2\omega_j,p_j)$ is a cone limit.
\end{lemma}
\begin{proof}
    Let $Z_1$ be any subsequential limit of $(X_j,\lambda_j^2\omega_j,p_j)$. By the curvature $L^2$-bound and the $\epsilon$-regularity Theorem \ref{thm-volume bound and epsilon regularity}, we know that $Z_1$ has only finitely many orbifold singularities. 
    
    We observe first that there is a universal Sobolev constant bound on $(Z_1,\omega_{Z_1})$, i.e. it depends only on the Sobolev constant bound in $\mathcal K(V,C_S)$ and is independent of the choice of scales $\lambda_{j}$. Given any $f\in C^{\infty}_c(Z_1\setminus S(Z_1))$, we may regard it as a smooth function on $X_j$ under pointed $\hat C^{\infty}$-Cheeger-Gromov convergence. We use Sobolev inequality of $\omega_j$, i.e.
    \begin{equation}
        \Vert f\Vert_{L^4(X_j,\omega_j)}\leq C_S(\Vert df\Vert_{L^2(X_j,\omega_j)}+\mathrm{Vol}(X_j,\omega_j)^{-\frac{1}{4}}\Vert f\Vert_{L^2(X_j,\omega_j)}).
    \end{equation}
    Rescaling the metric, we obtain
    \begin{equation}
        \Vert f\Vert_{L^4(X_j,\lambda_j^2\omega_j)}\leq C_S(\Vert df\Vert_{L^2(X_j,\lambda_j^2\omega_j)}+\lambda_j^{-1}\Vert f\Vert_{L^2(X_j,\lambda_j^2\omega_j)}).
    \end{equation}
    Since $f\in C^{\infty}_c(Z_1\setminus S(Z_1))$ and $\lambda_j^2 \omega_j\to \omega_{Z_1}$ in $C^{\infty}_{loc}(Z_1\setminus S(Z_1))$, the term $\Vert f\Vert_{L^2(X_j,\lambda_j^2\omega_j)}$ converges to $\Vert f\Vert_{L^2(Z_1,\omega_{Z_1})}$, so is bounded. Since $\lambda_j\to\infty$, by taking limit we see
    \begin{equation}
         \Vert f\Vert_{L^4(Z_1,\omega_{Z_1})}\leq C_S\Vert df\Vert_{L^2(Z_1,\omega_{Z_1})}.
    \end{equation}
    Now using orbifold smoothness, it's easy to see this extends across $S(Z_1)$. Then by \cite{tian-viaclovsky1}, we know that $Z_1$ is an ALE scalar flat K\"ahler orbifold.
    
    Since $\lambda_j\to\infty$ and $\lambda_j^{-1}\overline{\lambda}_j\to\infty$, the choice of scale \ref{equa-choice of energy scale} guarantees that given any $R>1$,  the $L^2$-energy is small on $B_{\lambda_j^2\omega_j}(p_j,R)\setminus B_{\lambda_j^2\omega_j}(p_j,R^{-1})$. Thus $Z_1$ has only one curvature singularity, and the convergence is smooth away from that point. By Fatou's lemma, we have
    \begin{equation}
        \int_{Z_1}\vert \Rm_{g_{Z_1}}\vert^2d\mu_{g_{Z_1}}\leq\epsilon.
    \end{equation}
    The universal bound on Sobolev constant allows us to apply the gap theorem in \cite[Corollary 4.5]{chen-weber11}, which implies that by choosing $\epsilon$ small, $Z_1$ must be flat, hence being a flat cone. 
\end{proof}

Let $\mathfrak{C}$ be the set of isomorphism classes of all cone limits constructed in Lemma \ref{lem-no intermediate scales} under pointed Gromov-Hausdorff topology. Then we have:
\begin{corollary}\label{cor-no intermediate scales}
    $\mathfrak{C}$ consists of a unique point. In particular, we have $\mathcal{C}_\infty(Z)=\mathcal{C}_0(X_\infty)$. 
\end{corollary}
\begin{proof}
    Since every element of $\mathfrak{C}$ is of the form $\mathbb C^2/\Gamma$, which is rigid as flat cones, it suffices to show that $\mathfrak{C}$ is connected. This is well-known; for example we can adapt the proof of \cite[Lemma 3.2]{DS2}.
\end{proof}

We are now ready to prove Theorem \ref{thm-bubble decomposition}.
    The proof is the same as \cite{bando1990a,bando1990b}. We wish to show at each step, there is a certain amount of energy extracted around every curvature singularity, so the bubbling must terminate after finitely many steps. There are two possibilities, and the key fact is that by our choice of $p_j$, the curvature must have a concentration around $p_j$. Denote by $$\tilde{g}_j=\bar \lambda_j^2g_j.$$ 
    
\noindent\textbf{Case (i)}. $p_Z=\lim p_j$ is the unique curvature singularity on $Z$.
    Then $Z$ has at most one orbifold singularity and the metric convergence is smooth outside $B(p_Z,1/2)$, and therefore by the choices of scales \ref{equa-choice of energy scale}, we have  
    \begin{equation}
        \int_{B_{\tilde g_j}(p_j, 1/2)}|\Rm_{\tilde g_j}|^2d\mu_{\tilde g_j}\leq \int_{X_j}|\Rm_{\tilde g_j}|^2d\mu_{\tilde g_j} -\epsilon.
    \end{equation}


\noindent\textbf{Case (ii)}. There are more than one curvature singularities. Note that other curvature singularities might approach the boundary of $B(p_Z,1)$, so $Z$ itself could \textit{a priori} be flat. This is different from \cite{bando1990b}, where for Einstein metrics flatness of limit always reduces to the first case. Denote by $\{q^{(i)}\}_{i=1}^{k}$ the set of curvature singularities of $Z$ with one of them being $p_Z$, and $q_j^{(i)}\in X_j$ be a sequence of points converging to $q^{(i)}$ in the Gromov-Hausdorff sense. Note the number of curvature singularities is \textit{a priori} bounded. Since in this case the number of curvature singularities is at least two, the energy splits in the following sense. 
If we choose $\delta_0>0$ so that the balls $B(q^{(i)},\delta_0)$ are disjoint, then by the definition of curvature singularities \eqref{eq--def of curvature sing}, for any $i$ and sufficiently large $j$, we have
\begin{equation}
    \int_{B_{\tilde g_j}(q_j^{(i)},\delta_0)}\vert\Rm_{\tilde g_j}\vert^2d\mu_{\tilde g_j}\leq\int_{X_j}\vert\Rm_{\tilde g_j}\vert^2d\mu_{\tilde g_j}-\frac{1}{2}\epsilon_0.
\end{equation}
 Thus there is also a definite amount of energy extracted in this step.

Now we can take $p_{j,1}=p_j$ and $\lambda_{j,1}=\overline\lambda_j$, so $(Z,\omega_Z,p_Z)$ corresponds to the root vertex of the bubble tree. We repeat the procedure above, namely at each step we examine the neighborhood of every curvature singularity of previous bubble, and construct the next minimal bubble. The whole process terminates at finitely many steps since at each step we extract a definite amount of energy. This finishes the proof of most of statements in Theorem \ref{thm-bubble decomposition}, except for non-flatness of bubbles. The latter relies on our polarized assumption. To be more precise, it follows from Theorem \ref{thm-characterization of orbifold points} and Proposition \ref{prop-curvature estimate on neck region} proved below, using the argument of \cite{bando1990b}.

\subsection{Proof of Theorem \ref{thm-characterization of orbifold points} and Theorem \ref{neck structure theorem}}


The key ingredient in this subsection is the following refined curvature estimate along the neck region, which serves as the analogue of \cite[Proposition~3]{bando1990a} in the setting of constant scalar curvature K\"ahler metrics.

In the following, we consider an annulus contained in an element of $\mathcal{K}(V, C_S)$. When we say that there exist constants satisfying certain properties, we mean that all such constants depend only on the parameters $V$ and $C_S$.
\begin{proposition}\label{prop-curvature estimate on neck region}
    There exist positive constants $\epsilon_1,\delta,K_0,C_k$ such that if 
    \begin{equation}
        \int_{A_p(r_1,r_0)}\vert\mathrm{Rm}_g\vert^2d\mu_g\leq\epsilon_1 \quad \text{for some $r_1<r_0$},
    \end{equation}
    then for any $K_0r_1\leq r\leq K_0^{-1}r_0$ and $K_0\leq K\leq\sqrt{r_0/r_1}$, we have the following refined curvature estimate:
    \begin{equation}\label{equa-curvature estimate on neck region}
        \begin{aligned}
      \int_{A_p(Kr_1,K^{-1}r_0)}\vert\mathrm{Rm}_g\vert^2d\mu_g&\leq C_0K^{-\delta}\\
             r^{2+k}\sup_{\partial B_r(p)}\vert \nabla^k\mathrm{Rm}_g\vert&\leq C_k\max\left\{(\frac{r_1}{r})^{\delta},(\frac{r}{r_0})^{\delta}\right\}, \text { for any\ } k\in \mathbb N_{\geq 0}.
        \end{aligned}
        \end{equation}
\end{proposition}

\begin{proof}

    Denote by $\widetilde{\Ric}=\Ric_{\omega}-\frac{1}{2}S_{\omega}\omega$, namely the traceless Ricci curvature. Since cscK metrics have harmonic traceless Ricci tensor, there is a refined Kato's inequality for $\widetilde{\Ric}$ \cite{branson00,CGH00,cibotaru-zhu12}; see also \cite[Lemma 6.2]{tian-viaclovsky1}. As a result, we obtain the following differential inequality:
    \begin{equation}
        \Delta\vert\widetilde{\Ric}\vert^{\frac{1}{2}}\geq-C\vert\Rm\vert\vert\widetilde{\Ric}\vert^{\frac{1}{2}}.
    \end{equation}
Therefore the same proof in \cite[Proposition 3]{bando1990a} gives neck estimate for Ricci curvature for some $\delta>0$:
\begin{equation}
    r^2\sup_{\partial B_r(p)}\vert\Ric_g\vert\leq C\max\{(\frac{r_1}{r})^{\delta},(\frac{r}{r_0})^{\delta}\}.
\end{equation}

By Corollary \ref{cor-no intermediate scales}, the metric is $C^{\infty}_{loc}$-close to the same model cone metric for arbitrary scales along this region. As a result, we can construct broken Hodge gauge on consecutive annuli and squeeze out to estimate the full curvature tensor as in \cite[Lemma 6.5]{tian-viaclovsky0}; see also \cite[Theorem 3.7]{chen-weber11}. We include a short proof here for readers' convenience.

According to the proof of Theorem 3.7 in \cite{chen-weber11}, which in turn depends on the idea of estimating under broken Hodge gauge as in \cite{uhlenbeck82,tian1990,tian-viaclovsky0}, we can find some $\epsilon_1>0$, such that if
\begin{equation}
    \int_{\mathrm{supp}\phi}\vert\Rm_g\vert^2d\mu_g<\epsilon_1,
\end{equation}
then 
\begin{equation}
    \int\phi^2\vert\Rm_g\vert^2d\mu_g\leq C\int_{\mathrm{supp}\nabla\phi}\vert\Rm_g\vert^2d\mu_g+C\int_{\mathrm{supp}\nabla\phi}\vert\Ric\vert^2.
\end{equation}
We choose cut-off $0\leq \phi\leq 1$, such that $\mathrm{supp}(\phi)\subset A_p(2^{i-1}r_1,2^{-i+1}r_0)$, $\phi\equiv 1$ on $A_p(2^ir_1,2^{-i}r_0)$ as well as $\mathrm{supp}(\nabla\phi)\subset A_p(2^{i-1}r_1,2^ir_1)\cup A_p(2^{-i}r_0,2^{-i+1}r_0)$. Then we obtain:
\begin{equation}
    \begin{aligned}
        \int_{A_p(2^ir_1,2^{-i}r_0)}\vert\Rm_g\vert^2&d\mu_g\leq C\int_{A_p(2^{i-1}r_1,2^ir_1)\cup A_p(2^{-i}r_0,2^{-i+1}r_0)}\vert\Rm_g\vert^2d\mu_g\\
        &+C\int_{A_p(2^{i-1}r_1,2^ir_1)\cup A_p(2^{-i}r_0,2^{-i+1}r_0)}\vert\Ric_g\vert^2d\mu_g.
    \end{aligned}
\end{equation}
Note that the second Ricci term can be controlled by our previous estimate, i.e.
\begin{equation}
    \int_{A_p(2^{i-1}r_1,2^ir_1)\cup A_p(2^{-i}r_0,2^{-i+1}r_0)}\vert\Ric_g\vert^2d\mu_g\leq C 2^{-2\delta i}.
\end{equation}
Let $i\geq i_0$ be integers and $2^{i_0}\geq K_0$. Denote by
\begin{equation}
    \beta_i=\int_{A_p(2^ir_1,2^{-i}r_0)}\vert\Rm_g\vert^2d\mu_g,
\end{equation}
then the above estimate gives:
\begin{equation}
    \beta_i\leq C(\beta_{i-1}-\beta_i)+C2^{-2\delta i},
\end{equation}
so we get:
\begin{equation}
    \beta_i\leq\frac{C}{1+C}\beta_{i-1}+\frac{C}{1+C}2^{-2\delta i}.
\end{equation}
Denote by $\mu=\max\{\frac{C}{1+C},2^{-2\delta}\}\in(0,1)$, then we have
\begin{equation}
    \beta_i\leq\mu\beta_{i-1}+\mu^{i+1}.
\end{equation}
Thus we get
\begin{equation}
    \beta_i-(i+1)\mu^{i+1}\leq\mu(\beta_{i-1}-i\mu^{i}).
\end{equation}
We may iterate to see that for any $l$ large, we have
\begin{equation}
    \beta_l\leq((l+1)\mu+C)\mu^{l}\leq l\mu^l.
\end{equation}
Since $\mu<1$, we can choose some $\delta^{\prime}$ such that $l\mu^l\leq2^{-l\delta^{\prime}}$ for $l$ large. This finishes the proof for the curvature integral estimate. The pointwise estimate follows from $\epsilon$-regularity in \cite[Theorem 3.1]{tian-viaclovsky0}.
\end{proof}

\

As an immediate consequence of Proposition \ref{prop-curvature estimate on neck region}, we obtain the following energy identity, which appears frequently in other geometric analysis problems.

\begin{proposition}\label{prop--energy identity}
    The following identities hold:
    \begin{equation}
        \lim_{j\to\infty}\int_{X_j}\vert \mathrm{Rm}_{g_j}\vert^2d\mu_{g_j}=\int_{X_\infty}\vert \mathrm{Rm}_{g_\infty}\vert^2d\mu_{g_\infty}+\sum_{I\in\mathcal T}\int_{Z_I}\vert \mathrm{Rm}_{g_{Z_I}}\vert^2d\mu_{g_{Z_I}},
    \end{equation}
    \begin{equation}
        \lim_{j\to\infty}\int_{X_j}\vert \mathrm{Ric}_{g_j}\vert^2d\mu_{g_j}=\int_{X_\infty}\vert \mathrm{Ric}_{g_\infty}\vert^2d\mu_{g_\infty}+\sum_{I\in\mathcal T}\int_{Z_I}\vert \mathrm{Ric}_{g_{Z_I}}\vert^2d\mu_{g_{Z_I}}.
    \end{equation}
\end{proposition}
\begin{proof}
    By Fatou's lemma, it's easy to see the right hand side is no more than the left hand side. For the converse, it suffices to show that there is no energy lost along neck region, which is guaranteed by Proposition \ref{prop-curvature estimate on neck region}. More precisely, for any fixed $K$ large, we have:
    \begin{equation}
        \begin{aligned}
             \lim_{j\to\infty}&\int_{X_j\setminus\bigcup_{I\in\mathcal T} A_{p_{j,I}}(K\lambda_{j,I}^{-1},K^{-1}\lambda_{j,P(I)})}\vert \mathrm{Rm}_{g_j}\vert^2d\mu_{g_j}\\
             &\leq \int_{X_\infty}\vert \mathrm{Rm}_{g_\infty}\vert^2d\mu_{g_\infty}+\sum_{I\in\mathcal T}\int_{Z_I}\vert \mathrm{Rm}_{g_{Z_I}}\vert^2d\mu_{g_{Z_I}}.
        \end{aligned}
       \end{equation}
       Now we apply Proposition \ref{prop-curvature estimate on neck region}, which tells us
       \begin{equation}
           \int_{X_j\setminus\bigcup_{I\in\mathcal T} A_{p_{j,I}}(K\lambda_{j,I}^{-1},K^{-1}\lambda_{j,P(I)})}\vert \mathrm{Rm}_{g_j}\vert^2d\mu_{g_j}=\int_{X_j}\vert \mathrm{Rm}_{g_j}\vert^2d\mu_{g_j}-\Psi(K^{-1}),
       \end{equation}
       so
       \begin{equation}
           \lim_{j\to\infty}\int_{X_j\setminus\bigcup_{I\in\mathcal T} A_{p_{j,I}}(K\lambda_{j,I}^{-1},K^{-1}\lambda_{j,P(I)})}\vert \mathrm{Rm}_{g_j}\vert^2d\mu_{g_j}=\lim_{j\to\infty}\int_{X_j}\vert \mathrm{Rm}_{g_j}\vert^2d\mu_{g_j}-\Psi(K^{-1}).
       \end{equation}
       The converse inequality follows by letting $K\to\infty$. The proof for the Ricci curvature is identically the same.
 \end{proof}

Then we point out the following \textit{``no compact holomorphic cycle"} property under our polarized assumption. Geometrically, this means that during formation of singularities $X_j\rightarrow X_{\infty}$, no holomorphic cycles is contracted. We thank Song Sun for pointing out the following argument, which also appears in \cite[Proposition 3.1]{han-viaclovsky20}.
\begin{proposition}\label{prop-no compact cycles}
    Let $(Z,\omega_Z,p_Z)$ be a bubble limit, which is a complete non-compact ALE scalar flat K\"ahler orbifold with finitely many singular points. Then $Z$ contains no holomorphic compact cycles. In particular, it is Stein.
\end{proposition}

\begin{proof}
    The key is to rule out compact cycles with curvature singularities sitting inside. Assume by contradiction that there is one such irreducible compact cycle $\mathcal{C}\subset Z$, then $[\mathcal{C}]$ defines a non-trivial integral homology class in $H_2(Z;\mathbb{Z})$. Since $\mathcal{C}$ is holomorphic, we must have
    \begin{equation}
        0<\int_{\mathcal{C}}\omega_\infty<\infty.
    \end{equation}
   \textbf{Claim}: There exists $m\in\mathbb{N}^{+}$ and a smooth integral $2$-chain $\Sigma$ which is compactly supported in $Z_0:=Z\setminus S(Z)$, such that $[\Sigma]=m[\mathcal{C}]$ in $H_2(Z;\mathbb{Z})$.

   In fact, consider the natural exact sequence of homology groups associated to the inclusion $Z_0\subset Z$:
   \begin{equation}
       H_2(Z_0;\mathbb{Z})\to H_2(Z;\mathbb{Z})\to H_2(Z,Z_0;\mathbb{Z}).
   \end{equation}
   By excision theorem, the relative homology group $H_2(Z,Z_0;\mathbb{Z})$ reduces to the direct sum of local homology groups around each orbifold point, i.e.
   \begin{equation}
       H_2(Z,Z_0;\mathbb{Z})\cong\bigoplus_{i} H_2(B_{\epsilon}(p_i), B_{\epsilon}(p_i)\setminus\{p_i\};\mathbb{Z})\cong\bigoplus_{i} H_1(Y_{i};\mathbb{Z}),
   \end{equation}
   where $\{p_i\}$ are orbifold points and $Y_i$ are links around $p_i$. The claim follows from the fact that $H_1(Y_i;\mathbb{Z})$ is a finite group. We can assume $\Sigma$ is a smooth cycle, since $Z_0$ is a smooth manifold.

   Now choose the sequence of scales $\lambda_j\to\infty$ and base points $p_j$ such that $(X_j,\lambda_j^2\omega_j,p_j)$ converges to $Z$ in the pointed $\hat C^{\infty}$-Cheeger-Gromov sense. Since $\Sigma$ is compactly supported in $Z_0$, we can choose a sequence of embeddings $\chi_j$ from a neighborhood of $\mathrm{supp}(\Sigma)$ to $X_j$, such that $\chi_j^*(\lambda_j^2\omega_j)$ converges smoothly to $\omega_\infty$. It follows that
   \begin{equation}
0<\lim_{j\to\infty}\lambda_j^2[\omega_j]\cap(\chi_{j})_*[\Sigma]<\infty.
   \end{equation}
However, both $[\omega_j]$ and $(\chi_{j})_*[\Sigma]$ are integral classes, so $[\omega_j]\cap[\Sigma]$ must be zero for all $j$ large, giving the desired contradiction.
\end{proof}

Finally, we need the following lemma, which is presumably well known to experts and can be directly derived from the proof of \cite[Lemma 2.11]{han-viaclovsky20}. For the reader's convenience, we include a proof here.
\begin{lemma}\label{lem-AE stein orbifold must be flat}
    An AE Stein scalar-flat K\"ahler orbifold must be isometric to the standard flat $\mathbb C^2$.
\end{lemma}
\begin{proof}
  Let $Z$ be an AE Stein scalar-flat orbifold. Choose a resolution $\widetilde{Z}$ of $Z$. 
Then, by a standard gluing construction on the K\"ahler potential (see \cite[Section 2]{AP}), 
$\widetilde{Z}$ admits an AE K\"ahler metric, and hence is an iterated blow-up of $\mathbb{C}^2$ 
by \cite[Proposition 4.3]{hein-lebrun2016}. Therefore we obtain a birational map from $Z$ to $\mathbb C^2$, which is defined away from orbifold points. Due to normality of $Z$, we obtain a morphism from $Z$ to $\mathbb C^2$. However, since $Z$ contains no compact cycle, using the normality of $Z$ again, we obtain that it must be biholomorphic to $\mathbb C^2$. In particular, $\omega_Z$ is smooth everywhere. The argument in \cite[Theorem 6.2]{hein-lebrun2016} then implies that $Z$ must be flat.
\end{proof}

Now we prove our main theorems in this subsection.

\

\noindent\textit{Proof of Theorem \ref{thm-characterization of orbifold points}}:
    Suppose the limit metric $\omega_\infty$ is smooth at $p_\infty$, but $p_{\infty}$ is a curvature singularity. Let $Z$ be a minimal bubble at $p_\infty$. Then by Proposition \ref{prop-no compact cycles} and Corollary \ref{cor-no intermediate scales}, it is an AE Stein orbifold with finitely many orbifold points. Thus by Lemma \ref{lem-AE stein orbifold must be flat}, $Z$ must be the standard flat metric on $\mathbb C^2$. At this stage we do not necessarily arrive at a contradiction, since as we mentioned in the proof of Theorem \ref{thm-bubble decomposition}, the minimal bubble $Z$ itself could be flat. However, since $Z$ is smooth everywhere, we may examine again curvature singularities of $Z$ and go down to next scales. 
    
    More precisely, all children of $Z$ are again AE, thus being isometric to the standard flat $\mathbb C^2$ by using Lemma \ref{lem-AE stein orbifold must be flat} again. Repeating this process, we see that all bubble limits in the tree would be flat. Choose a sequence $p_j\to p_\infty$, so the above implies, by a local version of Proposition \ref{prop--energy identity}, that
    \begin{equation}
        \lim_{j\to\infty}\int_{B(p_j,\delta)}\vert\Rm_{g_j}\vert^2d\mu_{g_j}=\int_{B(p_\infty,\delta)}\vert\Rm_{g_\infty}\vert^2d\mu_{g_\infty}
    \end{equation}
    for all $\delta$ small. However, this is impossible since $p_\infty$ is a curvature singularity. The proof for blow-up limits is similar. In particular, all asymptotic cones are non-trivial.
\qed

 \

\noindent\textit{Proof of Theorem \ref{neck structure theorem}}:
The desired gauge for Riemannian metrics is constructed in \cite{ozuch2022} via the study of CMC hypersurface foliations. The key ingredients of proof are the following:
\begin{itemize}
    \item Refined estimate for curvature along neck region. This is proved in Proposition \ref{prop-curvature estimate on neck region}.
    \item Orbifold groups along neck are non-trivial. This is proved in Theorem \ref{thm-characterization of orbifold points}.
    \item Closeness to model cone metric in any intermediate scales. This is proved in Corollary \ref{cor-no intermediate scales}.
\end{itemize}

\qed

\begin{remark}
    The construction in \cite{ozuch2022} relies heavily on the condition that $\Gamma_I\neq\mathrm{Id}$, which is justified by polarized assumption in our case. It would be interesting to know whether such a gauge on the neck exists even without the polarization assumption.
\end{remark}

\subsection{Consequences of bubbling analysis}

In this last subsection, we prove a quantitative statement that  
the geometric regular set is contained in the analytic regular set.  

The analytic regular set is defined via the regularity scale $r_x$, i.e. curvature bounds,  
while the geometric regular set is described using balls centered at the points $\{p_{j,I}\}$  
constructed in Theorem~\ref{thm-bubble decomposition}.  
For later applications to Bergman kernel estimates, it is important  
that curvature bound holds on this geometric regular region.
Without loss of generality, we may assume $X_{\infty}$ has a unique curvature singularity $p_{\infty}$.

\begin{proposition}\label{prop-curvature bound on geometric regular part}
    There exists a constant $C_1>0$ such that for any $r\in (0,1)$ and $j\in \mathbb N$, 
    \begin{equation}
     \left\{x\in X_j\mid r_x\leq r\right\}\subset \bigcup_{I\in \mathcal T}B(p_{j,I},C_1r).
    \end{equation}
\end{proposition}
\begin{proof}






   By the $\epsilon$-regularity \eqref{eq-two regularity comparable}, it suffices to prove the statement holds if we replace $r_x$ by $\widetilde{r}_x$. For this, we argue by contradiction. Suppose that there is a sequence $r_j\rightarrow 0$ and point $x_j\in X_j$ such that  
   \begin{equation}\label{upper bound of regular scale}
       x_j\notin \bigcup_{I\in \mathcal T}B(p_{j,I},jr_j)
\text{ and }
       \widetilde{r}_{x_j}\leq r_j.
   \end{equation} 
   Since $\mathrm{diam}(X_j,\omega_j)$ is uniformly bounded, the sequence $jr_j$ is bounded. By Theorem \ref{thm-characterization of orbifold points}, we know that under the Gromov-Hausdorff convergence of $X_j$ to $X_{\infty}$, $x_j$ converges to the curvature singularity $p_{\infty}$. Therefore it determines a path $\gamma$ in the tree $\mathcal{T}$. Indeed a vertex $I\in \gamma$ if and only if 
   \begin{equation}
\lambda_{j,P(I)}d_j(p_{j,I},x_j)\rightarrow 0 \text{ as } j\rightarrow \infty,
   \end{equation}Then in particular, for any $I\in \gamma$, we know that $\lambda_{j.P(I)}jr_j\rightarrow 0$ as $j\rightarrow \infty$. Let $I_0\in \gamma$ be the vertex which has the largest depth. 
   
   We first suppose that $I_0$ has children.
   Let $C(I_0)$ denote the set of children of the vertex $I_0$. Then by the choice of $I_0$, by passing to a subsequence, there are two cases: 
   \begin{itemize}
       \item Case 1.  $\lambda_{j,I_0}d_j(p_{j,I},x_j)\rightarrow \infty$ as $j\rightarrow \infty$ for some $I\in C(I_0)$. Since $\lambda_{j,I_0}d_j(p_{j,I},p_{j,I_0})$ stay bounded as $j\rightarrow \infty$, we know that $\lambda_{j,I_0}d_j(p_{j,I_0},x_j)\rightarrow \infty$ as $j\rightarrow \infty$, so $d_j(p_{j,I_0},x_j)\geq\frac{1}{2}\lambda_{j,I_0}^{-1}$ for $j$ large. By the curvature estimate \eqref{curvature estimate on annulus}, we know that there exists a positive constant $\delta_0$ such that 
       \begin{equation}
           \widetilde{r}_{x_j}\geq \delta_0(d_j(p_{j,I_0},x_j)-\lambda_{j,I_0}^{-1}).
       \end{equation}
       Therefore
       \begin{equation}
           r_j\geq\widetilde{r}_{x_j}\geq\delta_0(d_j(p_{j,I_0},x_j)-\lambda_{j,I_0}^{-1})\geq\frac{1}{2}\delta_0d_j(p_{j,I_0},x_j)
       \end{equation}
       for $j$ large, contradicting with the assumption \eqref{upper bound of regular scale}.
       \item Case 2. $\lambda_{j,I_0}d_j(p_{j,I},x_j)$ has a finite positive limit as $j\rightarrow \infty$ for all $I\in C(I_0)$. Then we know that under the pointed $\hat C^{\infty}$-Cheeger–Gromov convergence from $(X_j, \lambda_{j,I_0}^2\omega_j,p_{j,I_0})$ to $Z_{I_0}$, the convergence is smooth around the sequence $x_j$. Therefore we know that there exists a constant $\delta_0'$ such that for $j$ sufficiently large
       \begin{equation}\label{eq-curvature scale lower bound inteh second case}
\lambda_{j,I_0}\widetilde{r}_{x_j}\geq \delta_0'.
       \end{equation}Then combining \eqref{upper bound of regular scale} and \eqref{eq-curvature scale lower bound inteh second case}, we would have 
       \begin{equation}
           \delta_0' j\leq \lambda_{j,I_0}\widetilde{r}_{x_j}j\leq \lambda_{j,I_0} r_jj\leq \lambda_{j,I_0}d_j(p_{j,I},x_j)
       \end{equation}which contradicts with the assumption that $\lambda_{j,I_0}d_j(p_{j,I},x_j)$ has a finite positive limit.
   \end{itemize}

    If $I_0$ has no children, then we can argue similar as before by considering two cases:
    \begin{itemize}
     \item Case $1^\prime$.  $\lambda_{j,I_0}d_j(p_{j,I_0},x_j)\rightarrow \infty$ as $j\rightarrow \infty$. Then the argument is the same as that in the Case 1 of the previous discussion;
          \item Case $2^\prime$. $\lambda_{j,I_0}d_j(p_{j,I_0},x_j)$ has a finite limit, which could be zero. Then since $I_0$ has no children,  $(X_j, \lambda_{j,I_0}^2\omega_j,p_{j,I_0})$ to $Z_{I_0}$ converges in pointed  $C^{\infty}$-Cheeger–Gromov sense to a smooth ALE scalar-flat K\"ahler surface $Z_{I_0}$. Then the argument is the same as that in the Case 2 of the previous discussion.
    \end{itemize}
\end{proof}

\section{Convergence in a flat family}\label{convergence in a flat family}
Recall that we have a sequence  
$(X_j, L_j) \in \mathcal{K}(V, C_S),$
which share the same Hilbert polynomial and whose underlying manifolds 
\(X_j\) are diffeomorphic to each other, converging to a polarized 
Kähler orbifold \((X_\infty, L_\infty)\). 
Moreover, for each orbifold point $\{p_k\}_{k=1}^N\subset X_{\infty}$, Theorem~\ref{thm-bubble decomposition} provides a bubble tree \(\mathcal{T}_k\), together with associated bubbles, i.e., Stein ALE scalar-flat Kähler orbifolds \(\{Z_I\}_{I \in \mathcal{T}_k}\). As a consequence of the bubble analysis in the previous section, we obtain some 
topological results, which is likely well-known to experts. 

\begin{lemma}\label{lem--ALE has vanishing b1 and b3}
    Let $Z$ be an ALE K\"ahler orbifold surface. Then $b_1(Z)=b_3(Z)=0$.
\end{lemma}

\begin{proof}
    First, it suffices to prove this when $Z$ is smooth. This can be seen by choosing a resolution. Since $Z$ is of complex dimension two and has at worst rational singularities, the exceptional locus is given by a bunch of rational curves, and we can then use a Mayer--Vietoris exact sequence argument\footnote{We remark that there are some more general results on fundamental groups in \cite[Section 7]{kollar93}, although in our setting we can argue simply as above.}. Then we use the result of \cite[Corollary 1.4]{Weber2019} to see that $b_1(Z)=b_3(Z)=0$, as $Z$ has only one end which is non-parabolic.
\end{proof}

\begin{proposition}\label{prop--first betti number preserved}
    $b_1(X_j)=b_1(X_{\infty})$ and $b_2(X_j)= b_2(X_{\infty})+\sum_{I} b_2(Z_I)$.
\end{proposition}
\begin{proof} Both are the consequences of the  Mayer-Vietoris exact sequence. 
    Since $X_{\infty}$ is an orbifold, we know that there exists a decomposition $$ X_\infty=U_{\infty}\bigcup (\sqcup_k V_{\infty,k}),$$ where $U_{\infty}$ is a connected smooth manifold with boundary and its  boundaries given by $S^3/\Gamma_k$ and $V_{\infty,k}$ is a small neighborhood of the orbifold point $p_k$, diffeomorphic to a cone over $S^3/\Gamma_k$ and $U_\infty\cap V_{\infty,k}$ is diffeomorphic to $\mathbb R\times S^3/\Gamma_k$. By Theorem \ref{thm-bubble decomposition}, for $j$ large we have decompositions $ X_j=U_j\bigcup (\sqcup_k V_{j,k})$, where $U_j$ is diffeomorphic to $U_\infty$ and each $V_{j,k}$ is a connected non-compact manifold with a single end and $U_j\cap V_{j,k}$ is diffeomorphic to $\mathbb R\times S^3/\Gamma_k$. Moreover, we can write
\begin{equation}
    V_{j,k}=\bigcup_{I\in \mathcal T_k}V_{j,k;I},
\end{equation} and we can achieve that
\begin{itemize}
    \item if $V_{j,k;I}\cap V_{j,k;J}$ is non-empty, then it is diffeomorphic to $\mathbb R\times S^3/\Gamma$ for some finite subgroup $\Gamma$ of $U(2)$ acing freely on $S^3$;
    \item $V_{j,k;I}$ is diffeomorphic to the ALE orbifold $Z_I$ minus a finite number of ball centered near its orbifold points. Note that if $Z_I$ is smooth, then $V_{j,k;I}$ is diffeomorphic to $Z_I$.
\end{itemize}

    Then we can use  Mayer-Vietoris to compare the homology groups of $X_j$ and $X_\infty$.
Note that  $H_2(U_j\cap (\sqcup_k V_{j,k}), \mathbb R)$, $H_2(V_{j,k;I}\cap V_{j,k;J},\mathbb R)$, $H_1(U_j\cap (\sqcup_k V_{j,k}), \mathbb R)$ and $H_1(V_{j,k;I}\cap V_{j,k;J},\mathbb R)$ are all zero, so we have 
\begin{equation}
b_2(X_j)=b_2(U_j)+\sum_{I}b_2(V_{j,k;I}), 
\end{equation}and 
\begin{equation}
b_1(X_j)=b_1(U_j)+\sum_{I}b_1(V_{j,k;I}), 
\end{equation}
A similar argument can be applied to $X_\infty$ and noting that $V_{\infty,k}$ is contractible, we have  $b_2(X_\infty)=b_2(U_{\infty})$ and $b_1(X_\infty)=b_1(U_{\infty})$.
Since $U_j$ is diffeomorphic to $U_{\infty}$, we obtain that for $j$ large, in order to prove the result, it is sufficient to show that 
\begin{equation}\label{topology same as bubble}
    b_1(V_{j,k;I})=0\quad \text{ and }\quad b_2(V_{j,k;I})=b_2(Z_I).
\end{equation}
According to Lemma \ref{lem--ALE has vanishing b1 and b3}, we know that $b_1(Z_I)=0$. Applying the Mayer--Vietoris exact sequence once again to 
$Z_I$, we then obtain that \eqref{topology same as bubble} holds for each $I$.
\end{proof}

 Recall that the classical Riemann--Roch theorem for smooth polarized surfaces states
\begin{equation}
    \chi(X_j,kL_j)
    =\frac{L_j^2}{2}k^2-\frac{K_{X_j}\cdot L_j}{2}\,k
    +\chi(X_j,\mathcal O_{X_j}),
\end{equation}
and we have assumed that the Hilbert polynomial is independent of $j$, denoted by $P$ below. 
Since $X_{\infty}$ is an orbifold, it is normal with at worst rational singularities \cite[Proposition~5.15]{KM}. 
Then an analogous formula holds for any \emph{Cartier} divisor $D$ on $X_{\infty}$. 
The argument for this is standard, but for the reader’s convenience we include it here.
Choose a resolution $\pi:Y\to X_\infty$, so that $Y$ is a smooth projective surface. 
By the rational singularities assumption, we have 
\[
\pi_*\mathcal O_Y=\mathcal O_{X_\infty}, \qquad 
R^i\pi_*\mathcal O_Y=0 \ \ (i>0).
\]
For any Cartier divisor $D$, the projection formula gives
\[
\pi_*(k\pi^*D)\cong kD, 
\qquad 
R^i\pi_*(k\pi^*D)\cong kD\otimes R^i\pi_*\mathcal O_Y=0 \ \ ,i>0, k\geq 0.
\]
Hence the Leray spectral sequence implies
\[
H^i(Y,k\pi^*D)=H^i(X_\infty,\pi_*(k\pi^*D))=H^i(X_\infty,kD), \quad i\geq 0,k\geq 0.
\]
It follows that
$\chi(X_\infty,kD)=\chi(Y,k\pi^*D)$ for all $k$. Applying Riemann--Roch on $Y$, we obtain
\begin{equation}\label{RR on orbifold}
    \begin{aligned}
        \chi(X_\infty,kD)
        &=\chi(Y,k\pi^*D) =\frac{(\pi^*D)^2}{2}k^2-\frac{\pi^*D\cdot K_Y}{2}\,k+\chi(Y,\mathcal O_Y) \\
        &=\frac{D^2}{2}k^2-\frac{D\cdot K_{X_\infty}}{2}\,k+\chi(X_\infty,\mathcal O_{X_\infty}).
    \end{aligned}
\end{equation}
In particular, $\chi(X_\infty,kD)$ is a polynomial in $k$.

Specializing to our setting, we may assume (after replacing $L_\infty$ by a positive multiple if necessary) that $L_\infty$ is a genuine line bundle rather than merely a $\mathbb Q$-line bundle. Thus the above formula applies, and we set
\begin{equation}
    P_{\infty}(k)\coloneqq \chi(X_{\infty},kL_\infty).
\end{equation}
By formula \eqref{RR on orbifold}, we immediately obtain:
\begin{lemma}\label{lem--quadratic and linear terms are preserved}
    The quadratic and linear coefficients in $P(k)$ and $P_\infty(k)$ are equal.
\end{lemma}
\begin{proof}
    By the uniform volume estimate in Theorem~\ref{thm-volume bound and epsilon regularity} and $C^{\infty}_{loc}$-convergence away from the singular set, we see that $\mathrm{Vol}(X_j,\omega_j)=\mathrm{Vol}(X_\infty,\omega_\infty)$ for $j$ large. By Lemma \ref{lem-intersection}, we conclude that $L_j^2=L_\infty^2$, i.e. the quadratic coefficients are preserved. Similarly, by the cscK condition and smooth convergence of scalar curvature on regular set, we know that $S_{\omega_j}=S_{\omega_\infty}$ for $j$ large, thus the linear coefficients are also preserved.
\end{proof}

\begin{theorem}\label{thm-dim constant}
    For any $k\geq 1$, we have $P(k)=P_{\infty}(k)$.
\end{theorem}

\begin{proof}
According to Lemma \ref{lem--quadratic and linear terms are preserved} and Corollary \ref{cor--dimension non-decreasing}, the key point is to show that 
the constant term of the Hilbert polynomial does not jump in the limit.
To this end, we make use of the Chern--Gau\ss--Bonnet formula and the signature formula on orbifolds, together with the energy identity (Proposition~\ref{prop--energy identity}) and the Betti number identities (Proposition~\ref{prop--first betti number preserved}).

Note that the Hodge decomposition and Dolbeault's theorem hold for orbifolds; see, for example, 
\cite[Section~2]{baily} and \cite[Section~2.5]{PS2008}. 
In particular, we have
\[
\dim H^1(X_{\infty},\mathcal{O}_{X_{\infty}})=\frac{1}{2}b_1(X_\infty)=\frac{1}{2}b_1(X_j)=\dim H^1(X_j,\mathcal{O}_{X_j}).
\]
Thus, in order to compare the constant terms of the Hilbert polynomials, it suffices to show that 
$h^{2,0}:=\dim H^2(X,\mathcal{O}_X)$ remains unchanged in the limit. Note that by \eqref{dim non-increasing}, it is sufficient to show that \begin{equation}\label{not jump}
    h^{2,0}(X_\infty)\leq h^{2,0}(X_j).
\end{equation}

On $X_{\infty}$, we have the Chern--Gau\ss--Bonnet formula and the signature formula as follows \cite{Hitchin,kawasaki1978}. Although it does not play a role in what follows, we note that all quantities in the formula below are taken in the \emph{Riemannian} setting. In particular, for a Kähler metric the (Riemannian) scalar curvature equals twice the complex scalar curvature defined in \eqref{eq-convention of scalar curvature}.
\begin{equation}\label{signature}
    \chi(X_{\infty})=\frac{1}{8 \pi^2}\left(\int_{X_{\infty}}\left|W\right|^2 d V-\frac{1}{2} \int_{X_{\infty}}\left|\Ric \right|^2 d V+\frac{1}{6} \int_{X_{\infty}} S^2 d V\right)+\sum_k(1- \frac{1}{|\Gamma_k|}),
\end{equation}and \begin{equation}\label{cherngaussbonnet}
    \tau(X_{\infty})=\frac{1}{12 \pi^2}\left(\int_{X_{\infty}}\left|W^{+}\right|^2 d V-\int_{X_{\infty}}\left|W^{-}\right|^2 d V\right)-\sum_k \eta\left(S^3 / \Gamma_k\right),
\end{equation}where
$\Gamma_k\subset U(2)$ denotes the structure group at the orbifold point $\{p_k\}$, and $\eta(S^3/\Gamma_k)$ denotes the corresponding eta-invariants. By Hodge index theorem, on $X_{\infty}$ we also have 
\begin{equation}
    \tau(X_{\infty}) = 2\bigl(2h^{2,0}(X_{\infty})+1\bigr) - b_2(X_{\infty}).
\end{equation}
Combining this with \eqref{cherngaussbonnet} and \eqref{signature}, we obtain
\begin{equation}
    12h^{2,0}(X_{\infty})
    = b_2(X_{\infty})
    - \frac{1}{8\pi^2}\int_{X_{\infty}} |\Ric|^2dV
    - 3\sum_k \eta\!\left(S^3/\Gamma_k\right)
    + 2\sum_k \Bigl(1 - \frac{1}{|\Gamma_k|}\Bigr)
    + C,
\end{equation}
where $C$ stands for the expression involving only $b_1$ and scalar curvature, thus is preserved under convergence, i.e. for each $X_j$ we have
\begin{equation}
    12h^{2,0}(X_j)
    = b_2(X_j)
    - \frac{1}{8\pi^2}\int_{X_j} |\Ric|^2dV
    + C.
\end{equation}
Therefore to show \eqref{not jump}, by Proposition \ref{prop--first betti number preserved} and Proposition \eqref{prop--energy identity}, it is sufficient to show that 
\begin{equation}\label{inequ for ric and b_2}
    \sum_{I}\frac{1}{8\pi^2}\int_{Z_I}|\Ric|^2dV\leq \sum_{I}b_2(Z_I) - 3\sum_k \eta\!\left(S^3/\Gamma_k\right)
    + 2\sum_k \Bigl(1 - \frac{1}{|\Gamma_k|}\Bigr).
\end{equation}

On $Z_I$, we have the we have the Chern--Gau\ss--Bonnet formula and the signature formula as follows:
\begin{equation}\label{signature for ALE}
    \chi(Z_I)=\frac{1}{8 \pi^2}\left(\int_{Z_I}\left|W^-\right|^2 d V-\frac{1}{2} \int_{Z_I}\left|\Ric \right|^2 d V\right)+\sum_l(1- \frac{1}{|\Gamma_{I,l}|})+\frac{1}{|\Gamma_{I,\infty}|},
\end{equation}
\begin{equation}\label{cherngaussbonnet for ALE}
    \tau(Z_I)=\frac{-1}{12 \pi^2}\int_{Z_I}\left|W^{-}\right|^2 d V-\sum_l \eta\left(S^3 / \Gamma_{I,l}\right)+\eta(S^3/\Gamma_{I,\infty}),
\end{equation}where the signature of $Z_I$ is defined to be the signature of $\widetilde{Z}_I$, 
the one-point compactification at infinity of $Z_I$ endowed with the same orientation. It is a real-analytic orbifold, and the conformal class $[g_{Z_I}]$ extends to $\widetilde{Z}_I$ as a real-analytic anti-self-dual conformal metric \cite[Proposition 12]{chen-lebrun-weber08}. By Mayer-Vietoris exact sequence, $b_2(\widetilde{Z}_I)=b_2(Z_I)$. Combining \eqref{cherngaussbonnet for ALE}, \eqref{signature for ALE} and Lemma \ref{lem--ALE has vanishing b1 and b3}, we obtain that 
\begin{equation}
\begin{aligned}
      \frac{1}{8\pi^2}\int_{Z_I}|\Ric|^2dV=&b_2(Z_I)-6b_+(\widetilde{Z}_I)+2\sum_l(1-\frac{1}{|\Gamma_{I,l}|})-2(1-\frac{1}{|{\Gamma_{I,\infty}}|})\\
&+3\eta(S^3/\Gamma_{I,\infty})-3\sum_l\eta(S^3/\Gamma_{I,l}).
\end{aligned}
\end{equation}
Each $Z_I$ is an ALE scalar-flat Kähler orbifold, 
with local uniformizing groups $\Gamma_{I,l}$ at its orbifold points 
and with asymptotic group $\Gamma_{I,\infty}$ at infinity. By Theorem \ref{thm-bubble decomposition}, we know that
\begin{itemize}
    \item for each orbifold point $p_k \in X_\infty$, let $R_k \in \mathcal{T}_k$ denote the corresponding root vertex of the bubble tree. Then we have
\begin{equation}
    \Gamma_{R_k,\infty} = \Gamma_k;
\end{equation}
\item For each vertex $I$, its children $C(I)$ are in one-to-one correspondence with the orbifold points $\{p_{I,l}\}$ on $Z_I$, 
and the structure group at infinity of each child equals the local structure group of the corresponding orbifold point.

\end{itemize}
Taking sum over all bubble limits $Z_I$ and use above result, we know that all the orbifold term will get cancelled except for the the minimal bubbles and we obtain
\begin{equation}
       \sum_{I}\frac{1}{8\pi^2}\int_{Z_I}|\Ric|^2dV= \sum_{I}b_2(Z_I) -6\sum_{I}b_+(\widetilde{Z}_I)- 3\sum_k \eta\!\left(S^3/\Gamma_k\right)
    + 2\sum_k \Bigl(1 - \frac{1}{|\Gamma_k|}\Bigr).
\end{equation}
Since $b_+(\widetilde{Z}_I)\geq 0$, we obtain that \eqref{inequ for ric and b_2} holds.
\end{proof}
\begin{remark}
  As a consequence of Theorem \ref{thm-dim constant}, we have $b_+(\widetilde{Z}_I)=0$ for each bubble $Z_I$. 
This can also be seen directly as follows: by \cite{hein2021}, $Z_I$ admits another compactification 
$\widehat{Z}_I$ obtained by adding a divisor $D_I$ at infinity, and $\widehat{Z}_I$ is a 
projective orbifold with $D_I^2>0$. Moreover, by \cite[Proposition 2.11]{hein2021}, 
we have $H^i(\widehat{Z}_I,\mathcal{O}_{\widehat{Z_I}})=0$ for all $i\geq 1$. Therefore $b_+(\widehat{Z}_I)=1$, which implies $b_+(\widetilde{Z}_I)=0$.
\end{remark}

\begin{remark}\label{gradient estimate imply lower bound}
Note that Theorem~\ref{thm-dim constant} is proved without using H\"ormander's $L^2$-estimate. 
If some (weak) gradient estimate for holomorphic sections of $kL$ is available, 
then a uniform lower bound for Bergman kernels follows directly. 
Such a bound is known, for instance, for non-collapsed K\"ahler-Einstein surfaces where Ricci curvature is uniformly bounded.
\end{remark}

\begin{remark}
    When $K_{X_\infty}\cdot L_\infty<0$, a more straightforward proof without using Chern-Gau\ss-Bonnet theorem and signature formula can be given.  
    Recall that by Betti number identities, it is sufficient to show that 
  \begin{equation}
     h^{2,0}(X_{\infty}) \leq  h^{2,0}(X_{j}).
  \end{equation}
 Taking a resolution $\pi:Y\to X_\infty$, we can directly show that  $h^{2,0}(X_{\infty})= h^{2,0}(Y)=h^0(Y,K_Y)=0$. To see this, take any section $s\in H^0(Y,K_Y)$, then it descends to a section $s\in H^0(X_\infty^{\mathrm{reg}},K_{X_\infty})$. Let $m_0$ be the Cartier index of $K_{X_\infty}$, so $K_{X_\infty}^{m_0}$ is a line bundle over $X_\infty$. Since $X_\infty$ is normal, $s^{m_0}$ extends globally as a holomorphic section of $K_{X_\infty}^{m_0}$. Now using the fact that $K_{X_\infty}\cdot L_\infty<0$, we must have $s^{m_0}=0$, thus $s=0$. Note that the assumption here is satisfied for cscK metrics inside the controlled cone(for which the use of Betti number identities is even unnecessary). In particular, this includes K\"ahler-Einstein metrics on del Pezzo surfaces.
\end{remark}
    
  \begin{remark}
   When $K_{X_{\infty}}$ is a Cartier divisor, equivalently when $X_{\infty}$ has only canonical singularities, a proof using H\"ormander's $L^2$-estimate can be given as follows. Note that this relies on knowledge of constant terms in Riemann-Roch formula for Cartier divisors(cf. \cite{reid87}), so cannot be applied in generality.
   
Since $K_{X_{\infty}}$ is a Cartier,  both $\chi(X_\infty,kL_\infty)$ and $\chi(X_\infty,kL_\infty+K_{X_\infty})$ 
are quadratic polynomials with the same constant term. 
On the other hand, consider the convergence of $L^2$-holomorphic sections of $kL_j+K_{X_j}$. 
As in Proposition~\ref{L-infinity estimate}, the $L^\infty$ norm can be controlled by the $L^2$ norm. 
Furthermore, applying Hörmander’s estimate to $kL_j+K_{X_j}$ does not require a Ricci curvature bound. 
Hence, by a standard cut-off argument (for example, see the proof of Proposition~\ref{main analytic input}), 
every section of $H^0(X_\infty,kL_\infty+K_{X_\infty})$ arises as the limit of sections from $X_j$. 
It follows that
\[
P_\infty(0)=\chi(X_\infty,K_{X_\infty})=\chi(X_j,K_{X_j})=P(0).
\]
  \end{remark}

Now we can prove one of our main theorems in the paper:

\begin{theorem}[=Theorem \ref{algebraic convergence}]
        The Gromov-Hausdorff convergence $(X_j,L_j)$ in $\mathcal K(V,C_S)$ can be realized as convergence inside a Hilbert scheme.
\end{theorem}

\begin{proof}
    Given Theorem \ref{thm-dim constant}, the proof of this statement is fairly straightforward. We choose $k$ such that $L_j^k$ is very ample for all $j$. A $L^2$-orthonormal basis of $H^0(X_j,L_j^k)$ then defines an embedding $T_{k,j}:X_j\to\mathbb P^{N_k}$, which by Theorem \ref{thm-dim constant} $C^{\infty}_{loc}$-converges in Cheeger-Gromov's sense to an embedding $$T_{k,\infty }:X_\infty\to\mathbb P^{N_k}$$ defined by $H^0(X_\infty,L_\infty^k)$. Note that here we crucially use the fact that the holomorphic sections of $L_{\infty}^k$ on $X_{\infty}$ arising from the limit form the complete linear series. Consequently, $T_{k,\infty}$ is an embedding.

    On the other hand, consider algebraic varieties $T_{k,j}(X_j)$ in $\mathbb P^{N_k}$. Each of them corresponds to a point $b_j$ in the Hilbert scheme. By passing to some subsequences we can assume that $b_j$ converges to a limit point $b_\infty$ in analytic topology. Let $\widetilde W_k$ be the scheme corresponding to $b_\infty$, and $W_k$ its reduced scheme structure, which a priori could have more than 1 irreducible components. By abusing notation we also regard $W_k$ as its underlying algebraic set. Since the convergence of $T_{k,j}$ to $T_{k,\infty}$ is taken in $C^{\infty}_{loc}$-sense, at the moment we can only conclude that $T_{k,\infty}$ maps $X_\infty$ isomorphically to a unique irreducible component $B_k$ of $W_k$.

    To see that $W_k$ is actually irreducible, we use the estimate of volumes. Since $T_{k,\infty}$ is an isomorphism defined by $H^0(X_\infty,L_\infty^k)$, we have
    \begin{equation}
        \omega_{FS,k}=\omega_\infty+(2\pi k)^{-1}\partial\pp\log\rho_{k,X_\infty}.
    \end{equation}
    Thus
    \begin{equation}
        \mathrm{Vol}(B_k)=\frac12\int_{X_\infty}\omega_{FS,k}^2=\frac12\int_{X_\infty}\omega_\infty^2.
    \end{equation}
    Moreover, we have mentioned that the total volume is preserved under limits, i.e. for each fixed $j$, 
    \begin{equation}
        \int_{X_j}\omega_j^2=\lim_{j\to\infty}\int_{X_j}\omega_j^2=\int_{X_\infty}\omega_\infty^2.
    \end{equation}
   Therefore, we obtain
        \begin{equation}
            \mathrm{Vol}(B_k)=\vol(X,\omega_j).
        \end{equation}On the other hand, by the convergence in a Hilbert scheme, we know that 
\begin{equation}
\vol(X,\omega_j)=\mathrm{Vol}(\widetilde{W}_k)
\end{equation}
        Combining the above two equalities, we obtain that  $\widetilde{W}_k$ is irreducible and generically reduced. 

        Now $\widetilde{W}_k$ is irreducible and generically reduced. Moreover, its reduced structure $W_k$ is isomorphic to $X_\infty$, hence being normal. We can apply Hironaka's lemma \cite[Chapter III, Lemma
        9.12]{hartshorne} to conclude that $\widetilde{W}_k$ is indeed reduced, and hence  $X_\infty\cong\widetilde{W}_k\cong W_k$.
    \end{proof}

 \begin{remark}
           We emphasize that compared to the Fano case \cite[Lemma 2.4]{OSS16}, we do not expect the total space of smoothing to be $\mathbb Q$-Gorenstein; see discussions in Section \ref{in non-q-gorenstein family}. 
        \end{remark}

\section{Construction of PSH weights}\label{sec-construction of weightes}

As always, we fix a sequence $(X_j, L_j,\omega_j)\in \mathcal{K}(V,C_S)$, converging to a polarized K\"ahler orbifold $X_\infty$. Fix an orbifold point $p_\infty\in X_\infty$ and let $p_{j,I}\in X_j$, $I\in \mathcal{T}$ be the points constructed in the Theorem \ref{thm-bubble decomposition}  and let $$d_{j,I}(\cdot)=d(p_{j,I},\cdot).$$ Then the main result we are going to prove in this section is the following.
\begin{theorem}\label{thm--psh weight}
There exist constants $\delta>0$ and $L>0$ such that for sufficiently large $j$, 
we can find a psh function $u_j$ on $B(p_{j,1},\delta)$ satisfying:  
\begin{itemize}
       \item $\sqrt{-1}\partial\bar\partial u_j+\mathrm{Ric}(\omega_j)\geq \frac12 \omega_j 
    \quad \text{on } B(p_{j,1},\delta)$;
    \item $L\sum_{I}\log d_{j,I}\leq u_j\leq 0 
    \quad \text{on } B(p_{j,1},\delta)$;
    \item $|\nabla u_j|+|\ii\partial\pp u_j|\leq L 
    \quad \text{on } A_{p_{j,1}}(\tfrac{\delta}{2},\delta)$.
\end{itemize}

\end{theorem}

By a direct cut-off argument, we obtain the following global weight, which will serve as the basis for applying H\"ormander’s $L^2$ estimates on $X_j$. To simplify notation, we will assume throughout that 
$X_{\infty}$ has a single orbifold point $p_{\infty}$. 
The general case involves no essential new difficulties, 
only more cumbersome notation.

\begin{theorem}\label{global weights}
    There exist constants $C,\delta>0,k_0$ such that for $j$ large,  there are quasi-psh functions $u_j$ on $X_j$ satisfying
  \begin{equation}
        \sqrt{-1}\partial\pp u_j+\mathrm{Ric}(\omega_j)+k_0\omega_j\geq\omega_j\ \text{on} \ X_j,
    \end{equation}
     \begin{equation}
        C\sum_{I}\log d_{j,I}\leq u_j\leq 0\ \text{on}\ B(p_{j,1},\delta), \text{ and } u_j= 0 \text{ on }X_j\setminus  B(p_{j,1},\delta),
        \end{equation}
\end{theorem}

\begin{proof}
      According to Theorem \ref{thm--psh weight}, there are local psh functions $u_{j}$  are uniformly bounded in $C^2$ on annulus $ A_{p_{j,1}}(\frac{\delta}{2},\delta)$, where $p_{j,1}\to p_{\infty}$ correspond to the root of the bubble tree. Therefore, we can cut-off $u_{j}$ along these annuli to obtain a global quasi-psh function $u_j$ on $X_j$, such that 
    \begin{itemize}
        \item  $u_j=0$ on $X_j\setminus B(p_{j,1},\frac34\delta)$;
       
        \item $\sqrt{-1}\partial\pp u_j\geq -C\omega_j$ on $ A_{p_{j,1}}(\frac{\delta}{2},\delta)$ for some uniform $C>0$ independent of $j$.
    \end{itemize}
Then there exist constants \(k_0,C>0\), independent of \(j\), such that for all \(j\gg 1\) the desired estimates hold.
\end{proof}



We prove Theorem~\ref{thm--psh weight} by induction on scales. We begin with the smallest scales, i.e. neighborhoods of points $p_{j,I}$ where the vertex $I$ in the tree $\mathcal{T}$ has no children. After rescaling the metric by the factor $\lambda_{j,I}^2$, Theorem~\ref{thm-bubble decomposition} implies that the manifolds converge in the pointed $C^{\infty}$--Cheeger--Gromov sense to a smooth ALE scalar-flat Kähler surface. In Proposition \ref{psh weight for ALE}, using the ALE structure of the limit, we can construct psh functions on the limit with controlled growth. Since the convergence is smooth on deepest bubbles, we can pull them back to obtain psh functions on balls centered at $p_{j,I}$ of radius comparable to $\lambda_{j,I}^{-1}$. On the other hand, in Proposition \ref{psh weight for neck}, we construct psh functions on the neck region using Theorem \ref{neck structure theorem}. This allows us to connect two incomparable scales by patching them together, yielding psh functions on larger balls centered at $p_{j,I}$ with radius comparable to a small constant multiple of $\lambda_{j,P(I)}^{-1}$, where $P(I)$ denotes the parent of $I$ in the tree $\mathcal{T}$.

We then move to the next scale, namely balls centered at $p_{j,P(I)}$ with radius comparable to $\lambda_{j,P(I)}^{-1}$. 
After rescaling the metric by $\lambda_{j,P(I)}^2$, the manifolds converge in the pointed $C^\infty$ Cheeger--Gromov sense to a Stein ALE space with finitely many orbifold points. 
On this limit, one can construct psh functions with controlled growth and logarithmic poles at the orbifold points. 
Using the weights already constructed on the smaller balls centered at $p_{j,I}$, we can patch these together with the psh functions obtained at the previous scale. 
As a result, we obtain psh functions on balls centered at $p_{j,P(I)}$ of radius comparable to any large constant multiple of $\lambda_{j,P(I)}^{-1}$. 
By repeating this procedure inductively along the tree $\mathcal{T}$, we eventually reach the root vertex, and hence construct psh functions on balls of definite radius centered at $p_{j,1}$.

\subsection{PSH weights on bubbles and neck}
As preliminary steps, we construct psh weights on bubble limits as well as neck regions. The former construction deals with a fixed space, while the latter requires a uniform estimate before taking limits. 

The first result addresses the problem about constructing psh functions on bubble limits of $X_j$.

\begin{proposition}\label{psh weight for ALE}
      Let $(Z,\omega_Z,p_Z)$ be a Stein, and complete ALE K\"ahler orbifold surface with finite orbifold points $\{x_i\}_{i=1}^{k}$. Then there exists smooth psh function $\varphi_Z$ and $\rho_Z$ on $Z^{\mathrm{reg}}$ such that 
      \begin{itemize}
      \item $\rho_Z$ is a smooth positive function on $Z$ that equals a large constant on a compact set containing all orbifold points $\{x_i\}_{i=1}^{k}$ and 
 \begin{equation}
         \rho_Z=(1+\Psi(\rho_Z^{-1}))d(p_Z,\cdot). 
  \end{equation}
      \item  outside a compact set $K$, we have $\ii\partial\pp \rho_Z^2\geq \omega_Z$,
          \item  $|\varphi_Z-2\log d(x_i,\cdot)|\leq 1$ on small neighborhoods of each $x_i$,
          \item $\varphi_Z=a_Z\log (\rho_Z^2)$ outside a compact set for some constant $a_Z>0$,
          \item $\ii\partial\pp \varphi_Z\geq \min\{1,\rho_Z^{-2}(\log \rho_Z^2)^{-2}\}\omega_Z$
          \item  $\ii\partial\pp \varphi_Z+\mathrm{Ric}(\omega_Z)\geq \min\{1,\rho_Z^{-2}(\log \rho_Z^2)^{-2}\}\omega_Z$ .
      \end{itemize}
\end{proposition}
\begin{proof}
    The existence of $\rho_Z$ satisfying the first two properties is a standard result, see for example \cite[Lemma 2.15]{CH1} and \cite[Section 3]{SunZ}. Since $\omega_Z$ is a smooth orbifold K\"ahler metric, by working with a local orbifold chart, we can find local plurisuharmonic functions $\varphi_i$ defined in a neighborhood $U_i$ of $x_i$ satisfying that $U_i\cap U_l=\emptyset$ for $i\neq l$, and on each $U_i$, we have $|\varphi_i-2\log d(x_i,\cdot)|\leq 1$, $\ii\partial\pp \varphi_i\geq \omega_Z$ and $\ii\partial\pp\varphi_i+\mathrm{Ric}(\omega_Z)\geq \omega_Z$. Since $Z$ is Stein, there exists a strictly psh function $\varphi_0$ on $Z$. Then we define 
    \begin{equation}\label{eq--definition of varphi_z}
\varphi_Z=\sum_{i=1}^k\chi_i\varphi_i+A\chi_0\varphi_0+a_Z\log(\rho_Z^2),
    \end{equation}where $\chi_i$ are cut-off functions that equals 1 in a small neighborhood of $x_i$ and 0 outside $U_i$, $\chi_0$ is cut-off function which equals 1 on the compact set $K$ and 0 outside a larger compact set. We choose $A$ large and then $a_Z$ large such that $\varphi_Z$ as defined in \eqref{eq--definition of varphi_z} satisfies the last two inequalities claimed in the statement.
\end{proof}

On the other hand, based on the strong conical approximation in Theorem \ref{neck structure theorem}, we can construct logarithmic growth psh functions on the neck region to dominate the Ricci curvature.
\begin{proposition}\label{psh weight for neck} 
By increasing $K_0$ and $K_1$ if necessary, there exists a constant $C>0$ such that for all sufficiently large $j$ and each $I\in\mathcal{T}$, on $A_{p_{j,I}}\!\left(K_0\lambda_{j,I}^{-1},\,K_1^{-1}\lambda_{j,P(I)}^{-1}\right)$ there exists psh functions 
$\psi^{\mathcal C_I}_{j}$ and $\phi_j^{\mathcal{C}_I}$ satisfying
\begin{equation}\label{eq:psi_estimate}
\begin{aligned}
\ii\partial\bar\partial \psi^{\mathcal C_I}_{j} + \Ric(\omega_j) &\ge 0,
&\qquad
\ii\partial\bar\partial \phi_j^{\mathcal{C}_I} &\geq \tfrac{1}{2}\,\omega_j, \\[3pt]
C\log d_{j,I}^2 \leq \psi_{j}^{\mathcal{C}_I} &\leq 0,
&\qquad
|\phi_{j}^{\mathcal{C}_I}| &\leq C.
\end{aligned}
\end{equation}Moreover on $A_{p_{j,I}}(2^{-1}K_1^{-1}\lambda_{j,P(I)}^{-1}, K_1^{-1}\lambda_{j,P(I)}^{-1})$, we have 
\begin{equation}
  |\nabla \psi_j^{\mathcal{C}_I}|_{\lambda_{j,P(I)}^2\omega_j}+ |\ii\partial\pp \psi_j^{\mathcal{C}_I}|_{\lambda_{j,P(I)}^2\omega_j} +  |\nabla \phi_j^{\mathcal{C}_I}|_{\lambda_{j,P(I)}^2\omega_j}+ |\ii\partial\pp \phi_j^{\mathcal{C}_I}|_{\lambda_{j,P(I)}^2\omega_j}\leq C.
\end{equation}
\end{proposition}
\begin{proof}

 We first construct psh functions on the inner and outer regions of the neck. Using a gluing argument, we then combine them into a global psh function on the neck with controlled growth and the required Ricci curvature control. The actual constant $K_0,K_1$ below may vary from line to line, but would be uniformly fixed through the construction.
 
The weight is grafted from the K\"ahler cone via the gauge constructed in Theorem~\ref{neck structure theorem} and on the cone, we choose the logarithmic growth weight function introduced in \cite[Lemma 3.7]{SunZ}. Note that the statement of Theorem~\ref{neck structure theorem} is purely Riemannian. Therefore, we also need to analyze the asymptotic behavior of the compatible complex structures associated with the K\"ahler metrics.

We follow the notation of Theorem \ref{neck structure theorem}. We pass to universal cover of $A_{p_{j,I}}$ and $\mathbb R^4/\Gamma_I$ and obtain a natural embedding $\widetilde\Phi_{j,I}:\widetilde{A}_{p_{j,I}}\to\mathbb R^4$. By abusing terminology, we say that $\widetilde\Phi_{j,I}$ is $\Gamma_I$-equivariant. Denote by $g_0$ the Euclidean metric on $\mathbb R^4$ and $r_0$ the distance function to the origin on $(\mathbb R^4,g_0)$. We also denote by $\tilde g_j, J_{\tilde g_j},\tilde\omega_j$ the induced structures on $\widetilde A_{p_{j,I}}$, hence on $\mathrm{Im} (\widetilde{\Phi}_{j,I})$ by abusing notation.
We claim that for each $j$, there exists a complex structure $J_j$ on $\mathbb R^4$ compatible with the Euclidean metric $g_0$, such that
\begin{equation}\label{equa-complex structure error on neck}
    \vert\nabla_{g_0}^k(J_{\tilde g_j}-J_j)\vert_{g_0}\leq C_k r_0^{-k-\delta}.
\end{equation}
To see this, note that by Theorem \ref{neck structure theorem} with a possibly different choice of $K_0$, the following holds on the annulus $A_{0}(2^{-1}K_0,2K_0\lambda_{j,I}^{\frac12}\lambda_{j,P(I)}^{-\frac12})\subset\mathbb R^4$ for $j$ large:
\begin{equation}
    \vert\nabla^k_{g_0}(\lambda_{j,I}^2\tilde g_j-g_0)\vert_{g_0}\leq C_k r_0^{-k-\delta}.
\end{equation}
 Choosing a point $p$ on the outer boundary of the annulus, we define $J_j\vert_p\coloneqq J_{\tilde g_j}\vert_p$. Then we extend $J_j$ to the whole $\mathbb R^4$ by parallel transport, so $(g_0,J_j)$ defines a K\"ahler structure. This is why we need to pass to the universal cover. For any $q$ in the annulus, we choose a $g_0$-minimizing geodesic $\gamma$ connecting $p$ and $q$, so we can estimate
\begin{equation}
    \vert(J_{\tilde g_j}-J_j)(q)\vert_{g_0}\leq\int_\gamma\vert\nabla_{g_0}J_{\tilde g_j}\vert_{g_0}\leq\int_\gamma\vert\nabla_{g_0}(\lambda_{j,I}^2\tilde g_j)\vert_{g_0}=C r_0(q)^{-\delta}.
\end{equation}
In the last step we use the fact that $\vert\nabla_{g_0}(\lambda_{j,I}^2\tilde g_{j})\vert_{g_0}=\vert\nabla_{g_0}(\lambda_{j,I}^2\tilde g_j-g_0)\vert_{g_0}=C_1 r_0^{-1-\delta}$. The higher estimate can be derived in a similar way. In the following, we denote by $\omega_{0,j}$ the associated K\"ahler form on $(\mathbb R^4,g_0,J_j)$.

Now we choose the psh weight on $\mathbb C^2$ as in \cite[Lemma 3.7-(1)]{SunZ}, i.e.
\begin{equation}
    \psi_1=\log(r_0^2+1)-\frac{1}{100}(\log(r_0^2+3))^{\frac12},
\end{equation}
which enjoys the property that by choosing $K_0$ large, the following holds on $\mathbb C^2\setminus B_{g_0}(0,K_0)$:
\begin{equation}\label{equa-positivity of model psh weights}
    dd^c_{J_j}\psi_1\geq C^{-1}r_0^{-2}(\log r_0)^{-\frac32}\omega_{0,j}.
\end{equation}
We then estimate the error term caused by complex structures. By the estimate \eqref{equa-complex structure error on neck} and straightforward computation, it's easy to see that
\begin{equation}
    \vert(dd^c_{J_{\tilde g_j}}-dd^c_{J_j})\psi_1\vert_{\lambda_{j,I}^2\tilde g_j}\leq Cr_0^{-2-\epsilon}
\end{equation}
for some $\epsilon>0$. By positivity estimate \eqref{equa-positivity of model psh weights} and Ricci curvature bound in Proposition \ref{prop-curvature estimate on neck region}, we get a desired strictly psh function $\psi_1$ on $\tilde A_{p_j,I}$. Since $\psi_1$ is $\Gamma_I$-invariant, it descends to a psh function $\psi_{j,1}$ on $A_{p_j,I}$. We can similarly graft back the psh function $r_0^2$. In particular, we get psh functions $\psi_{j,1}$ and $\phi_{j,1}$ on $A_{p_{j,I}}({K_0}\lambda_{j,I}^{-1}, {K_0}{({\lambda_{j,I}\lambda_{j,P(I)}})^{-\frac12}})$ such that 
\begin{equation}\label{eq-inner psh on cone}
\begin{aligned}
    \psi_{j,1}=(1+\Psi(\frac{1}{\lambda_{j,I}d_{j,I}}))\log (\lambda_{j,I}^2d_{j,I}^2),& \quad \ii\partial\pp\psi_{j,1}+\Ric(\omega_j)\geq (\lambda_{j,I}d_{j,I})^{-2-\delta}\lambda_{j,I}^2\omega_j\\
        \phi_{j,1}=(1+\Psi(\frac{1}{\lambda_{j,I}d_{j,I}}))\lambda_{j,I}^2d_{j,I}^2, &\quad \ii\partial\pp\phi_{j,1}\geq \frac{1}{2}\lambda_{j,I}^2\omega_j.
\end{aligned}
    \end{equation}
    In the same manner, on a small ball $B(0,K_0^{-1})$ inside $\mathbb C^2$, the function $\psi_2=\log r_0^2-\frac{1}{100}(-\log r_0^2)^{1/2}$ satisfies
    \begin{equation}
        dd^c_{J_j}\psi_2\geq C^{-1}r_0^{-2}(-\log r_0)^{-\frac32}\omega_{0,j}.
    \end{equation}
 Then we get psh functions $\psi_{j,2}$ and $\phi_{j,2}$ on $A_{p_{j,I}}({K_1^{-1}({\lambda_{j,I}\lambda_{j,P(I)}})^{-\frac12}},{K_1^{-1}\lambda_{j,P(I)}^{-1}})$ such that 
    \begin{equation}\label{eq--outer psh on cone}
    \begin{aligned}
         \psi_{j,2}=(1+\Psi(\lambda_{j,P(I)}d_{j,I}))\log(\lambda_{j,P(I)}^2d_{j,I}^2),& \quad \ii\partial\pp\psi_{j,2}+\Ric(\omega_j)\geq \frac{\lambda_{j,P(I)}^2\omega_j}{(\lambda_{j,P(I)}d_{j,I})^{2-\delta}}\\
        \phi_{j,2}=(1+\Psi(\lambda_{j,P(I)}d_{j,I}))\lambda_{j,P(I)}^2d_{j,I}^2, &\quad \ii\partial\pp\phi_{j,2}\geq \frac{1}{2}\lambda_{j,P(I)}^2\omega_j.
    \end{aligned}
    \end{equation}

 We then apply a maximum construction to obtain a global plurisubharmonic function on the entire neck region. To this end, we need to show that on the outer boundary of the inner region the maximum is realized by the inner function, and on the inner boundary of the outer region it is realized by the outer function \cite[Chapter I--Lemma 5.17]{Demailly}. By exploiting the asymptotics of these functions, this can be arranged after rescaling $\psi_{j,2}$ and $\phi_{j,2}$ by a large positive constant and applying an appropriate translation. Here we rely on the fact that both $\psi_{j,1}$ and $\psi_{j,2}$ have leading terms given by the logarithm of the distance function. Moreover, since the scale $\lambda_{j,I} \to \infty$, the function $\psi_{j,1}$ tends to infinity along the outer boundary of the inner region. We therefore subtract a large constant to make it comparable to the logarithm of the distance function.
 
 More precisely, we fix $K_0$ and $K_1$ large such that $\frac{2}{K_1}\ll 1\ll \frac{K_0}{2}$ and the error term $\Psi$ in \eqref{eq-inner psh on cone} and \eqref{eq--outer psh on cone} is bounded by $\frac{1}{10}$. Then it is easy to show that we find a large constant $L$ such that when $d_{j,I} = 2{K_1^{-1} (\lambda_{j,I} \lambda_{j,P(I)}})^{-\frac12}$, we have 
\begin{equation}
\psi_{j,1} - \log \lambda_{j,I}^2 > L \big( \psi_{j,2} - \log \lambda_{j,P(I)}^2 \big) 
+ (2-L) \log \Bigg( 2K_1^{-1} (\lambda_{j,I} \lambda_{j,P(I)})^{-\frac12}\Bigg) =:\widetilde{\psi}_{j,2}
\end{equation}and  when $d_{j,I} = 2^{-1}K_0 (\lambda_{j,I} \lambda_{j,P(I)})^{-\frac12}$, we have
\begin{equation}
\psi_{j,1} - \log \lambda_{j,I}^2 
< \widetilde{\psi}_{j,2}.
\end{equation}Then we can define a psh function $\psi_{j}^{\mathcal{C}_I}$ on $A_{p_{j,I}}\!\left(K_0\lambda_{j,I}^{-1},\,K_1^{-1}\lambda_{j,P(I)}^{-1}\right)$ satisfying the desired estimate via

 \[ \psi_{j}^{\mathcal{C}_I} = \begin{cases} \psi_{j,1}-\log \lambda_{j,I}^2   & \text{ on } A_{p_{j,I}}(\frac{K_0}{\lambda_{j,I}},\frac{2}{K_1\sqrt{\lambda_{j,I}\lambda_{j,P(I)}}})\\
          \max\{ \psi_{j,1}-\log \lambda_{j,I}^2,\widetilde{\psi}_{j,2}\} & \text{ on } A_{p_{j,I}}(\frac{2}{K_1\sqrt{\lambda_{j,I}\lambda_{j,P(I)}}},\frac{K_0}{2\sqrt{\lambda_{j,I}\lambda_{j,P(I)}}})\\
          \widetilde{\psi}_{j,2} & \text{ on } A_{p_{j,I}}(\frac{K_0}{2\sqrt{\lambda_{j,I}\lambda_{j,P(I)}}},\frac{1}{K_1\lambda_{j,P(I)}}).
       \end{cases}
    \]
The construction of $\phi_j^{\mathcal{C}_I}$ is similar and we omit the detail here. Then we finish the proof of Theorem \ref{neck structure theorem}.
\end{proof}

\subsection{Gluing procedure}
We now state the following two theorems, each associated with a vertex $I \in \mathcal{T}$. 
The proof proceeds by induction on the depth (equivalently, the height) of the vertex in the tree $\mathcal{T}$.

\begin{theorem}\label{thm A}
    For any $I\in \mathcal{T}$, any $K\geq K_0$, there exists $j_0(I,K)$ such that for $j\geq j_0$, there exists psh function $\varphi_{j,I}$ on $B(p_{j,I},2K\lambda_{j,I}^{-1})$ such that 
    \begin{equation}
        \mathrm{Ric}(\omega_j)+\ii\partial\pp \varphi_{j,I}\geq\frac{1}{2}\omega_j \text{ on } B(p_{j,I},2K\lambda_{j,I}^{-1})
    \end{equation}
    \begin{equation}\label{asymptotics of weight in thmA}
\varphi_{j,I}=b_I(\log(\lambda_{j,I}^2d_{j,I}^2)+\Psi(K^{-1},j^{-1})\log K) \text{ on } A_{p_{j,I}}(\frac{\sqrt{K}}{2}\lambda_{j,I}^{-1},2K\lambda_{j,I}^{-1})
    \end{equation}
    \begin{equation}
        |\varphi_{j,I}|\leq C(I,K)
    \end{equation}
\end{theorem}

\begin{theorem}\label{thmB}
    For any $I\in \mathcal{T}$, there exist $L(I)\geq 2\max\{K_0,K_1\}$ and $j_0(I)$ such that for any $j\geq j_0$,  there exists psh function $\psi_{j,I}$ on $B(p_{j,I},2L^{-1}\lambda_{j,P(I)}^{-1})$,such that 
    \begin{equation}
           \mathrm{Ric}(\omega_j)+\ii\partial\pp \varphi_{j,I}\geq\frac{1}{2}\omega_j \text{ on } B(p_{j,I},2L^{-1}\lambda_{j,P(I)}^{-1})
    \end{equation}
    \begin{equation}
\psi_{j,I}=\phi_{j}^{\mathcal{C}_I}+L\psi_{j}^{\mathcal{C}_I}-2L\log(\frac{\lambda_{j,I}L}{\lambda_{j,P(I)}})+(2b_I+1)\log L \text{ on } A_{p_{j,I}}(L\lambda_{j,I}^{-1},2L^{-1}\lambda_{j,P(I)}^{-1})
    \end{equation}
    \begin{equation}
        L\sum_{J\leq I}\log d_{j,J}\leq \psi_{j,I}\leq C(I).
    \end{equation}
\end{theorem}

We are going to do inductions on the height of a vertex  in the tree $\mathcal T$. There are three steps:
\begin{itemize}
    \item Step 1. Theorem \ref{thm A} holds for vertices $I$ with height 0, i.e., leaves in the tree $\mathcal{T}$;
    \item Step 2. If Theorem \ref{thm A} holds for a vertex $I$, then Theorem \ref{thmB} holds for this vertex $I$;
    \item Step 3. If Theorem \ref{thmB} holds for all siblings of a vertex $I$, then Theorem \ref{thm A} holds for its parent $P(I)$.
\end{itemize}

\noindent \textit{Proof of Step 1.} This is a direct consequence of Proposition \ref{psh weight for ALE}. Recall that for a leaf vertex $I\in \mathcal{T}$, $(X_j, \lambda_{j,I}^2\omega_j, p_{j,I})$ converges to a smooth Stein ALE scalar flat K\"ahler surface $(Z_I,\omega_{Z_I},p_{Z_I})$ in the pointed Cheeger-Gromov sense. Then the psh functions constructed in Lemma \ref{psh weight for ALE} can be pulled back to $X_j$. Note that the Ricci form is scaling invariant, we obtain that for any $K\geq 10$, for $j$ large,  we have nonnegative smooth psh functions $\varphi_{j,I}$ on $B(p_{j,I}, 2K\lambda_{j,I}^{-1})$ satisfying
\begin{equation}\label{eq--positivity on bubble region}
    \ii\partial\pp \varphi_{j,I}+\mathrm{Ric}(\omega_j)\geq\frac{1}{10}K^{-2}(\log K)^{-2}\lambda_{j,I}^2\omega_j
\end{equation} and 
\begin{equation}\label{psh on smooth bubble}
 \varphi_{j,I}=2a_I\log (\lambda_{j,I}d_{j,I})+\Psi(K^{-1},j^{-1})\log K \text{ on } A_{p_{j,I}}(\frac{\sqrt K}{2}\lambda_{j,I}^{-1},2K\lambda_{j,I}^{-1}).
\end{equation}
Since $\lambda_{j,I}\rightarrow \infty$ as $j\rightarrow\infty$, the constant in the right hand side of \eqref{eq--positivity on bubble region} is bigger than 1, when $j$ is sufficiently large. Therefore we obtain Theorem \ref{thm A} holds for vertices with height 0.
\hfill$\square$

\

\noindent\textit{Proof of Step 2.} Assuming Theorem \ref{thm A} holds for a vertex $I$, then we need to patch psh functions on the scale $\lambda_{j,I}$ and  psh functions on the cone region to get psh functions on the scale $\lambda_{j,P(I)}$.  By Theorem \ref{neck structure theorem}, we know that for any $L\geq 4K_0^2$, then for $j$ sufficiently large, we have psh functions 
\begin{equation}
\phi_j^{\mathcal C_I}+L\psi^{\mathcal C_I}_{j} \text{ on }A_{p_{j,I}}(\frac{\sqrt L}{2}\lambda_{j,I}^{-1},K_1^{-1}\lambda_{j,P(I)}^{-1})
\end{equation} satisfying
\begin{equation}\label{eq--eastimate on neck region}
     \ii\partial\pp\left(\phi_{j}^{\mathcal C_I}+L\psi^{\mathcal C_I}_{j}\right)+\mathrm{Ric}(\omega_j)\geq\frac{1}{2}\omega_j, \text{ with }    |\phi_j^{\mathcal{C}_I}|\leq C(I) \text{ and } C(I)\log d^2_{j,I}\leq \psi_j^{\mathcal{C}_I}\leq 0,
\end{equation}for some constant $C$ depends only on $I$ and  independent of $j$.

Let $\varphi_{j,I}$ be the psh function constructed in Theorem \ref{thm A}.
 Then we consider functions 
 \begin{equation}\label{eq--weight after translation}
\psi_{j,I,1}=\varphi_{j,I}-2L\log(\frac{\lambda_{j,I}}{\lambda_{j,P(I)}}) 
\end{equation} and 
\begin{equation}\label{eq-outside weight}
\psi_{j,I,2}=\phi_j^{\mathcal{C}_I}+L\psi^{\mathcal C}_{j}-2L\log (\frac{\lambda_{j,I}L}{\lambda_{j,P(I)}})+(2b_I+1)\log L,
 \end{equation}where $b_I$ is the constant in \eqref{asymptotics of weight in thmA}.
  We are going to show that we can take $L$ large such that for $j$ sufficiently large we have 
 \begin{equation}\label{eq--condition for glueing1}
\psi_{j,I,1}< \psi_{j,I,2} \text{ on }\{d_{j,I}\geq L\lambda_{j,I}^{-1}\} 
 \end{equation}
and 
\begin{equation}\label{eq--condition for glueing2}
  \psi_{j,I,1}> \psi_{j,I,2}  \text{ on } \{d_{j,I}\leq \sqrt{L}\lambda_{j,I}^{-1}\}.
\end{equation}
 To show this, consider the difference $\psi_{j,I,1}-\psi_{j,I,2}$ and notice that from the estimate \eqref{eq--eastimate on neck region}, we know that for $L$ large and then $j$ large, on the region $A_{p_{j,I}}(\frac{\sqrt L}{2}\lambda_{j,I}^{-1},2L\lambda_{j,I}^{-1})$, we have 
\begin{equation}\label{eq--intermediate bound}
  \psi_{j,I,1}-\psi_{j,I,2}= (b_I-L)\log (\lambda_{j,I}^2d_{j,I}^2)+2L\log L-(2b_I+1+\Psi(L^{-1},j^{-1})\log L+C
\end{equation}where $C$ denotes a function on $A_{p_j}(\frac{\sqrt L}{2}\lambda_j^{-1},2L\lambda_j^{-1})$, which is bounded independent of $L$ and $j$. Fix $L\geq K_0$ large and for large $j$ such that the error function $\Psi(L^{-1},j^{-1})$ in \eqref{psh on smooth bubble} is less that $\frac{1}{2}$ and $\frac{
1
}{2}\log L$ is bigger that the function $C$ in \eqref{eq--intermediate bound} and then for $j$ sufficiently large so that Theorem \ref{thm A} holds for $L$. Then we can check that \eqref{eq--condition for glueing1} and \eqref{eq--condition for glueing2} are satisfied.

Then we can define a psh functions on $B(p_{j,I}, L^{-1}\lambda_{j,P(I)}^{-1})$ by 
  \[ \psi_{j,I}(x) = \begin{cases} \psi_{j,I,1}   & d_{j, I}\leq \sqrt{L}\lambda_{j,I}^{-1} \\
          \max\{ \psi_{j,I,1},\psi_{j,I,2}\} & d_{j, I}\in [\sqrt{L}\lambda_{j,I}^{-1},L\lambda_{j,I}^{-1}]\\
          \psi_{j,I,2} & d_{j, I}\geq L\lambda_{j,I}^{-1},
       \end{cases}
    \]and the desired estimate in Theorem \ref{thmB} follows from Theorem \ref{thm A} and \eqref{eq--eastimate on neck region} and \eqref{eq-outside weight}.
\hfill$\square$

\

\noindent\textit{Proof of Step 3} Assuming Theorem \ref{thmB} holds for all siblings of a vertex $I$, then we show that Theorem \ref{thm A} holds for the vertex $P(I)$, the parent of $I$. Recall that $(X_j, \lambda_{j,P(I)}^2\omega_j, p_{j,P(I)})$ converges in the $\hat C^{\infty}$-Cheeger-Gromov sense to a Stein ALE scalar flat orbifold $(Z_{P(I)},\omega_{Z_{P(I)}}, p_{Z_{P(I)}})$ with finite orbifold points $\{x_i\}_{i=1}^k$ and for every $i$, points $p_{j,I_i}\in X_j$ are chosen such that $p_{j,I_i}\rightarrow x_i$, where $\{I_i\}_{i=1}^k$ consists of all children of $P(I)$ and we make the convention that $I_1=I$. 
As a consequence of Lemma \ref{psh weight for ALE} and the $\hat C^{\infty}$-Cheeger-Gromov convergence, for any $L,K\gg 1$, for $j$ sufficiently large, we have psh function $\widetilde{\varphi}_{j,P(I)}$ on $B(p_{j,P(I)},K \lambda_{j,P(I)}^{-1})\setminus\bigcup_{i}B(p_{j,I_i},\frac{L^{-2}}{2}\lambda_{j,P(I)}^{-1})$ such that 
\begin{equation}
    \mathrm{Ric}(\omega_j)+\ii\partial\pp \widetilde{\varphi}_{j,P(I)}\geq K^{-2}(\log K)^{-2}\lambda^2_{j,P(I)}\omega_j,
\end{equation}
\begin{equation}\label{inner control of outside weight in step3}
    \widetilde{\varphi}_{j,P(I)}=\log (\lambda^2_{j,P(I)}d_{j, I_i}^2)+\Psi(j^{-1}| L)\log L \text{ on } A_{p_{j,I_i}}(\frac{L^{-2}}{2}\lambda_{j,P(I)}^{-1},2L^{-1}\lambda_{j,P(I)}^{-1}).
\end{equation} Moreover, we have 
\begin{equation} \label{outside weight in step3}\widetilde{\varphi}_{j,P(I)}=a_I\left(\log(\lambda_{j,P(I)}^2d_{j,P(I)}^2)+\Psi(K^{-1},j^{-1})\log K\right) \text{ on } A_{p_{j,I}}(\frac{\sqrt{K}}{2}\lambda_{j,P(I)}^{-1},2K \lambda_{j,P(I)}^{-1}).
\end{equation}

For a child $I_i$ of $P(I)$, applying Theorem \ref{thm A}, we get psh functions $\psi_{j,I_i}$. Then similar to the proof of Step 2, we are going to patch these $\psi_{j,I_i}$ with a multiple of $\varphi_{j, P(I)}$ after adjusting constants. Firstly we choose $L$ large such that the balls $B(p_{j,I_i}, L^{-1}\lambda_{j,P(I)}^{-1})$ are disjoint and  $$L\geq 2\max\{L_i\mid i=1,\cdots k\},$$ where $L_i$ denotes the constant obtained in Theorem \ref{thmB} for the vertices $I_i$.  Then for this fixed $L$, we choose $j$ large such that the error term $\Psi(j^{-1}|L)$ in \eqref{inner control of outside weight in step3} is bounded by $\frac{1}{4}\log L$. For a large constant $A$,  we consider functions 

\begin{equation}
\widehat{\psi}_{j,I_i}=\psi_{j,I_i}+(2L_i-2b_{I_i}+1-3A)\log L_i
\end{equation}and 
\begin{equation}
    \widehat \varphi_{j,P(I)}=A \widetilde{\varphi}_{j,P(I)}
\end{equation} Then we can show that by choosing $A$ large (independent of $j$), we have 
\begin{equation}
    \widehat\psi_{j,I_i}< \widehat \varphi_{j,P(I)} \text{ on } \partial B(p_{j,I_i}, L^{-1}\lambda_{j,P(I)}^{-1})
\end{equation} and 
\begin{equation}
    \widehat\psi_{j,I_i}> \widehat \varphi_{j,P(I)} \text{ on } \partial B(p_{j,I_i}, L^{-2}\lambda_{j,P(I)}^{-1}).
\end{equation}
Then for any $K$ large and $j$ sufficiently large, we can define a $\varphi_{j,P(I)}$ on $B(p_{j,P(I)}, K\lambda_{j,P(I)}^{-1})$ by 

\[ \varphi_{j,P(I)} = \begin{cases} \widehat \psi_{j,I_i}   & d_{j, I_i}\leq L^{-2}\lambda_{j,P(I)}^{-1} \\
          \max\{ \widehat \psi_{j,I_i},\widehat \varphi_{j,P(I)}\} & d_{j, I_i}\in [L^{-2}\lambda_{j,P(I)}^{-1},L^{-1}\lambda_{j,P(I)}^{-1}]\\
          \widehat \varphi_{j,P(I)}& d_{j,I_i}\geq L^{-1}\lambda_{j,P(I)}^{-1}.
       \end{cases}
    \]
    The desired estimate on $\varphi_{j,P(I)}$ follows from the estimate on $\psi_{j, I_i}$ obtained in Theorem \ref{thm A} and \eqref{outside weight in step3}.
\qed

\section{Zariski openness of cscK metrics}\label{sec--zariski openness}

In this section, we prove the Zariski openness of  cscK surfaces in certain smooth polarized family, following \cite{donaldson14}. Let $\pi:(\underline{X},L)\to S$ be a smooth polarized family of surfaces, where $S$ is a quasi-projective variety and $\pi$ is a flat morphism. Here polarized family means that $L_s=L\vert_{X_s}$ is an ample line bundle for each $s\in S$, while smooth family means that $X_s$ is non-singular for each $s\in S$. By passing to some Zariski open subset we may assume the base $S$ is non-singular, so by general theory $\pi$ is a holomorphic submersion. In particular, all fibers are diffeomorphic to each other.
Now define $S^*$ as the subset that $X_s$ admits a cscK metric $\omega_s$ in the class $ c_1(L_s)$. 

We introduce the following condition:
\begin{definition}
    The family $\underline{X}\to S$ is said to satisfy Property (Sob) if there is a constant $C$ such that for any $s\in S^*$, the Sobolev constant $C_{\omega_s}$ is bounded by $C$.
\end{definition}

Following the approach in \cite{donaldson14}, we prove the following theorem. 
The improvement is that in complex dimension $2$, the bounded Ricci curvature 
assumption in \cite{donaldson14} can be removed.
Note that belonging to the controlled cone is preserved under smooth families, 
and by Lemma \ref{contolled cone imply sobolev bound}, the Sobolev constant is a priori bounded whenever the Kähler class lies in the controlled cone. 
Therefore, the following theorem implies Theorem \ref{Zariski openness in controlled cone}.
\begin{theorem}\label{Zariski openness under sob bound}
    Suppose $S$ has Property (Sob) and that for each $s\in S$,  $\mathrm{Aut}(X_s, L_s)$ is finite. Then $S^*$ is a Zariski open subset of $S$.
\end{theorem}

In the following, we first give the proof of Theorem \ref{effective bergman function bound}, and then following Donaldson's argument to prove Theorem \ref{Zariski openness under sob bound}. 


\begin{proof}[Proof of Theorem \ref{effective bergman function bound}]
 It suffices to show that, given any sequence in $ \mathcal{K}(V,C_S)$, there exists a subsequence $X_{j}$ and a constant $C,\beta$ such that the desired estimate holds.
This subsequence $X_j$ converges to a orbifold $X_\infty$, and we can apply the bubbling analysis from Section~\ref{sec--bubble}. 
For each orbifold point $p_k \in X_\infty$, $k=1,\dots,N$, Theorem~\ref{thm-bubble decomposition} provides a bubble tree $\mathcal{T}_k$ and corresponding points $p_{j,I} \in X_j$. In the following, we omit the index $j$.

We define for   $r \in (0, 10^{-1})$.
\begin{equation}\label{geometric regular}
    \Omega_r := X \setminus \bigcup_{k}\bigcup_{I \in \mathcal{T}_k} B(p_I, r).
\end{equation}
Note that the number of orbifold points on $X_\infty$ is uniformly bounded by Theorem \ref{thm--TVcompactness}
By Theorem \ref{thm-bubble decomposition} and Theorem \ref{thm-characterization of orbifold points}, the number of vertices in the tree $\mathcal{T}$ is also bounded, therefore the bound on $Z_r:=X\setminus\Omega_r$ follows from Theorem~\ref{thm-volume bound and epsilon regularity}.

    Let $u_j$ be the psh functions on $X_j$ constructed in Theorem \ref{global weights}. 
    In particular, we have uniform estimate 
\begin{equation}\label{estimate on ball}
    C\log r\leq u_j\leq  0\text{ on }\Omega_r\subset X_j.
\end{equation}

The proof of this result follows the same approach as in \cite[Section~6]{donaldson14}, based on peaked sections method \cite{tian90}. 
The only difference is that, when applying H\"ormander’s $L^2$-estimate to perturb an almost holomorphic section into a holomorphic one, 
we must use the additional weight $u_j$ constructed in Theorem \ref{global weights}. 
Consequently, we require $k$ to be larger than $r^{-2}(\log r^{-1})^{\beta}$, rather than merely $r^{-2}$ as in \cite[Proposition~7]{donaldson14}.

Given an $x\in\Omega_r\subset X_j$. By the definition of $\Omega_r$ in \eqref{geometric regular}, Proposition \ref{prop-curvature bound on geometric regular part} and the standard elliptic estimate for cscK metrics, we know that the geodesic 1/2-ball centered at $x$ under the rescaled metric $r^{-2}g_j$ has bounded geometry. Then a local construction (see \cite[Section 6]{donaldson14} and \cite[Section 7.2]{gabor-book}) produce a smooth section $\sigma_{l,1}$ of $L_j^k$ for $k$ large satisfying 
 \begin{itemize}
 \item $\sigma_{k,1}$ is supported on $B(x,2^{-1}r)$ and $|\sigma_{k,1}|(x)=1$;
     \item $\Vert \pp \sigma_{k,1}\Vert_{L^{2,\#}}=\epsilon(kr^2);$
     \item $\Vert  \sigma_{k,1}\Vert_{L^{2,\#}}=1+\epsilon(kr^2).$
 \end{itemize}
Here $L^{2,\#}$ denotes the $L^2$-norm with respect to the rescaled metric $k\omega$ and $\epsilon(x)$ denotes a quantity (independent of $X_j$) bounded by $$e^{-x^{\alpha}}p(x),$$ where $\alpha>0$ and $p$ is a polynomial in $x$. 
We emphasize that the exponentially decaying factor will play an important role in what follows.
 Applying H\"ormander's $L^2$-estimate with the weight $u_j$, we get smooth sections $\sigma_{k,2}$ satisfying $\pp \sigma_{k,2}=\pp \sigma_{k,1}$ and 
\begin{equation}
\int_{X_j}|\sigma_{k,2}|^2e^{-u_j}\omega_j^2\leq \frac{2}{k}\int_{X_j}|\pp \sigma_{k,1}|^2e^{-u_j}\omega_j^2.
\end{equation}
Writing in the rescaled norm and noting that $u_j\leq 0$ on $X$ and $e^{-u_j}\leq r^{-C}$ on $B(x,2^{-1}r)$, we obtain that 
\begin{equation}
    \|\sigma_{k,2}\|_{L^{2,\#}}\leq 2r^{-C} \|\pp \sigma_{k,1}\|_{L^{2,\#}}.
\end{equation}By choosing $\beta$ sufficiently large, we can ensure that whenever 
$k \geq r^{-2}(\log r^{-1})^{\beta}$, one still has 
\begin{equation}
    \|\sigma_{k,2}\|_{L^{2,\#}} = \epsilon(kr^2),
\end{equation}possibly with a different choice of $\epsilon$. Then, by the same argument as in \cite[Section~6]{donaldson14}, the holomorphic section 
$\sigma_k=\sigma_{k,1}-\sigma_{k,2}$ of $L^k$ satisfies the following properties, which in turn yield the desired estimate for the Bergman kernel $\rho_k$:  
\begin{itemize}
    \item $|\sigma_k(x)|=1+\epsilon(kr^2)$;
    \item there exists a constant $C$ such that, for $kr^2$ sufficiently large,  
    \[
    |\langle\sigma_k,\tau\rangle_{L^{2,\#}}|
    \;\leq\; \frac{C\|\tau\|_{L^{2,\#}}}{kr^2}
    \]
    for any holomorphic section $\tau$ vanishing at $x$;
\item $\|\sigma_k\|_{L^{2,\#}}=1-\frac{S_{\omega}}{4\pi k}+O((kr^2)^{-2})$.
\end{itemize}

For higher-order estimates of the Bergman kernel, we use Tian's peaked section method \cite{tian90} and its refinements in \cite{Ruan,Lu,liulu}. While a detailed application of \cite{Lu,liulu} may yield all high order estimate, we restrict ourselves here to a basic version, which only derives the bound as stated in Theorem \ref{effective bergman function bound}, but this is sufficient for our later purposes.

For a point $x \in \Omega_r$, we consider the rescaled metric and line bundle
\[
g_r = \lceil r^{-2}\rceil g, \qquad L_r = L^{\lceil r^{-2}\rceil}.
\]
In the following, we study sections of $L_r^m$ for sufficiently large $m\in \mathbb{N}$, and let $L^{2,\#}$ denote the $L^2$ norm with respect to the rescaled metric $mg_r$. 
As discussed above, H\"ormander's $L^2$ estimate applies as usual with the weight constructed in Theorem~\ref{global weights}, provided $m \gg (\log r^{-1})^\beta$ for some fixed $\beta$. 

We use the notation in \cite{Lu,liulu}.  
Let $P=(p_1,p_2)$ be a multiindex and let $|P|=p_1+p_2$ and $P!=p_1!p_2!$.    
We order the multiindex $P$ lexicographically as follows:  
\begin{itemize}
    \item by total degree $|P|=p_1+p_2$;
    \item if the total degree is equal, by the first differing component of $P$;
\end{itemize}
Let $\sigma(P):\mathbb{N}_{\geq 0}^2 \to \mathbb{N}_{\geq 0}$ denote this ordering, and $R(j)$ its inverse. 
By \cite[Lemma 3.2]{Ruan}, which refines the results of \cite{tian90}, 
for multiindices $P$ with $|P|\leq 4$, there exists a uniform constant $C$ and holomorphic sections $s_{m,P}$ of $L_r^m$ for $m$ sufficiently large, such that:
\begin{itemize}
    \item $\|s_{m,P}\|_{L^{2,\#}}=1+O(m^{-10})$;  
    \item there exist local $K$-coordinates $(z_1,z_2)$ around $x$ and a local $K$-frame $e_L$ of $L$ such that  
\begin{equation}\label{peaked section}
    s_{m,P} \;=\; \bigl(\lambda_{m,P}\,z^P + o(|z|^5)\bigr)\, e_L^{\otimes m},
\end{equation}
where the constant is given by
\begin{equation}\label{lambda def}
    \lambda_{m,P}^2
    = \frac{1}{m^2}\left(
        \int_{|z|\leq \tfrac{\log m}{\sqrt m}}
        |z^P|^2\, |e_L|^{2m}\, \omega^2
    \right)^{-1};
\end{equation}

    \item for two multiindices $P\neq P'$, with $|P|,|P'|\leq 4$, we have
    \begin{equation}\label{eq-almost orthogonal for peak section}
        |\langle s_{m,P}, s_{m,P'} \rangle_{L^{2,\#}}| \leq C\,m^{-1-\frac{|\,|P|-|P'|\,|}{2}};
    \end{equation}
    \item for any section $\tau$ of $L^m$ vanishing to order $q\geq 5$, 
    \begin{equation}\label{eq--almost orthogonal}
        |\langle s_{m,P}, \tau \rangle_{L^{2,\#}}|\leq C\,m^{-1-\frac{|q-|P||}{2}}\|\tau\|_{L^{2,\#}}.
    \end{equation} 
\end{itemize}Here, $O(m^{-10})$ denotes a term bounded by $C' m^{-10}$, where the constant $C'$ depends only on the local geometry of $X$ near $x$ and is therefore uniformly bounded. It is also worth noting that, although there is no uniform pointwise control on the remainder term $o(\cdot)$ in the expansion \eqref{peaked section}, one still has a uniform $L^2$-norm estimate together with control on its vanishing order. This is sufficient for our purposes, since in later applications we use this expansion only to compute derivatives at the fixed point $x$ and to evaluate global $L^2$-inner products.  The definitions of $K$-coordinates and $K$-frames can be found in \cite[Proposition 2.1]{Ruan} and \cite[Definition A.2]{liulu}. These notions are essential for the proofs in these papers. Here we only use the fact that
\begin{equation}\label{property of K-coordinate}
   \nabla |e_L|^2(x)=0,\quad  \frac{\partial^2}{\partial z_i\partial z_j}|e_L|^2(x)=0, \quad \ii\partial\pp \left(\sum_i |z_i|^2 |e_L|^2\right)(x)=0.
\end{equation}
Following the argument in \cite[Section 4]{Lu}, we can expand the constant $\lambda_{m,P}$ up to certain order of $m$:
\begin{equation}\label{expansion of lambda00}
    \lambda_{m,(0,0)}^2=1+\frac{S_{\omega}}{4\pi}m^{-1}+O(m^{-2});
\end{equation}
\begin{equation}\label{expansion of lambda01}
\begin{aligned}
     \lambda^2_{m,(1,0)}&=m\left(1+\frac{1}{m}\int_{\mathbb C^2}|z_1|^2(\frac{1}{4}R_{i\bar jk\bar l}z_i\bar z_jz_k\bar z_l-R_{i\bar j}z_i\bar z_j)\frac{dV_0}{(2\pi)^2}+O(m^{-2})\right)^{-1}\\
     &=m+\frac{S_{\omega}}{4\pi}+O(m^{-1});
\end{aligned}
\end{equation}
\begin{equation}\label{expansion of lambda10}
\begin{aligned}
     \lambda^2_{m,(0,1)}&=m\left(1+\frac{1}{m}\int_{\mathbb C^2}|z_2|^2(\frac{1}{4}R_{i\bar jk\bar l}z_i\bar z_jz_k\bar z_l-R_{i\bar j}z_i\bar z_j)\frac{dV_0}{(2\pi)^2}+O(m^{-2})\right)^{-1}\\
&=m+\frac{S_{\omega}}{4\pi}+O(m^{-1}).
\end{aligned}
\end{equation}
Moreover we can improve the estimate in \eqref{eq-almost orthogonal for peak section} for $s_{m,(1,0)}$ and $s_{m,(0,1)}$:
\begin{equation}
\begin{aligned}
      \langle s_{m,(1,0)},s_{m,(0,1)}\rangle_{L^{2,\#}}&=\frac{1}{m}\int_{\mathbb C^2}z_1\bar z_2(\frac{1}{4}R_{i\bar jk\bar l}z_i\bar z_jz_k\bar z_l-R_{i\bar j}z_i\bar z_j)\frac{dV_0}{(2\pi)^2}+O(m^{-2})\\
      &=O(m^{-2})
\end{aligned}
\end{equation}
Then we can construct an $L^2$-orthonormal basis as follows.  
First, for $j \geq 15$, let $\{s_{m,j}\}$ be an orthonormal basis for the subspace of holomorphic sections vanishing to order at least $5$ at $x$.  
Next, apply the Gram--Schmidt process to the family 
\[
\{\, s_{m,P}, s_{m,j} \mid |P|\leq 4 ,j\geq 15 \,\},
\] 
to obtain an orthonormal basis $\{s_{m,j}\}_{j=0}^{d_m}$ by write
\[
    s_{m,i} \;=\; (1+\epsilon_{ii})\, s_{m,R(i)} \;+\; \sum_{j>i}^{14}\epsilon_{ij}\, s_{m,R(j)}+\sum_{j\geq 15}\epsilon_{m,j}s_{m,j},
    \qquad 0\leq i \leq 14,
\]
where $d_m=\dim H^0(X,L^m)\leq Cm^2$.
By the properties of peak sections discussed above \eqref{peaked section}--\eqref{eq--almost orthogonal}, 
we obtain estimates for the error terms $\epsilon_{ij}$.
In what follows, we record only those estimates that are sufficient for our later proof, 
without aiming for their sharpest form.
\begin{equation}\label{estimate on error terms}
\begin{aligned}
      &  |\epsilon_{00}|\leq \frac{C}{m^{3}}.\quad  |\epsilon_{0j}|\leq \frac{C}{m^{3/2}} \text{ for } j=1,2\quad |\epsilon_{0j}|\leq \frac{C}{m^{2}} \text{ for } j\geq 3 \\
      &     |\epsilon_{1j}|\leq \frac{C}{m^2} \text{ for } j=1,2\quad   |\epsilon_{1j}|\leq \frac{C}{m^{3/2}} \quad \text{ for } j=3,\cdots, 14 \\
      & |\epsilon_{22}|\leq \frac{C}{m^3}  \quad   |\epsilon_{2j}|\leq \frac{C}{m^{3/2}} \text{ for } j=3,\cdots,14 \\
    &|\epsilon_{ij}|\leq \frac{C}{m^3} \text{ for } i=0,1,2 \text{ and } j\geq 15.
\end{aligned}
\end{equation}
Then we can estimate the Bergman kernel $\rho_m=\sum_{i=0}^{d_m}|s_{m,i}|^2$ at the point $x$. 
We have
\begin{equation}
    |\nabla_{g_r}\rho_m|(x)\leq \frac{C}{m^{3/2}}\left(|\nabla_{g_r} s_{m,R(1)}|+|\nabla_{g_r}s_{m,R(2)}|\right)\leq \frac{C}{m}.
\end{equation}
\begin{equation}
\begin{aligned}
       |\nabla^2_{g_r}\rho_m|(x)\leq & \left|\sum_{i=0}^2\nabla^2_{g_r}|s_{m,R(i)}|^2\right|+C\sum_{i=0}^2\left|\epsilon_{ii}\nabla^2_{g_r} |s_{m,R(i)} |^2\right|\\
&+C\sum_{i=3}^5|\epsilon_{0i}\nabla^2_{g_r}\langle s_{m,R(i)},s_{m,R(0)}\rangle|+\sum_{i,j=1}^2\left|\epsilon_{0i}\epsilon_{0j}\nabla^2_{g_r} \langle s_{m,R(i)},s_{m,R(j)}\rangle\right|\\ &+2\left|\epsilon_{12}\nabla^2_{g_r}\langle s_{m,R(1)},s_{m,R(2)}\rangle\right|
\end{aligned}
\end{equation}Note that by the construction of the peaked section \eqref{peaked section}--\eqref{property of K-coordinate}, and by the estimate \eqref{expansion of lambda00}--\eqref{expansion of lambda10}, we know that
\begin{equation}
    \sum_{i=0}^2\nabla^2_{g_r}|s_{m,R(i)}|^2(x)=O(m^{-1}), \quad \left|\nabla^2_{g_r}\langle s_{m,R(i)},s_{m,R(j)}\rangle\right|(x)\leq Cm, \text{ for } i,j\in \{1,2\}.
\end{equation} Then by the estimate in \eqref{estimate on error terms} on the $\epsilon_{ij}$, we get 
\begin{equation}
    |\nabla^2_{g_r}\rho_m|(x)\leq \frac{C}{m}.
\end{equation}

Taking into account the rescaling of the metric $g_r=\lceil r^{-2}\rceil g$ 
and the corresponding change of polarization 
$L_r = L^{\lceil r^{-2}\rceil},$ (so $k=m\lceil r^{-2}\rceil$),
we then arrive at the desired estimate.
\end{proof}

With the preparations in place, the proof of the Zariski openness proceeds along 
the lines of \cite{donaldson14}; for the reader's convenience, we provide a brief 
sketch here, highlighting the main differences.

\begin{proof}[Proof of Theorem \ref{Zariski openness under sob bound}.]
First, by replacing the polarization with a sufficiently high power, we may assume the line bundle is very ample and that all higher cohomology groups vanish, 
so that the dimension of $H^0(X_s, L_s)$ is a fixed number $N+1$ for all $s \in S$, given by the Riemann–Roch formula (we can achieve this since the dimension of cohomology group is upper-semicontinuous \cite[Theorem 12.8]{hartshorne}). 
Choosing a basis for this space of sections, the embedding of $X_s$ corresponds to a point in the Chow variety 
$\mathrm{Chow}=\operatorname{Chow}(N,d)$ of cycles of degree $d$ in $\mathbb{P}^N$. 
For $s\in X_s$, let $\Gamma_s$ denote the $SL(N+1)$-orbit of the point $[X_s]$ in $\operatorname{Chow}$.

As we can do induction on the dimension of the base $S$, it is sufficient to show that if $S^*$ is non-empty, then it contains a non-empty Zariski subset of $S$. Then Donaldson introduced a subset $GD$ of $S$, where the name stands for ``generic degeneration". A point $\sigma\in GD$ \cite{donaldson14} if and only for any $(\sigma,[W])$ in the closure of 
   $ \bigcup_{s\in S, Z\in \Gamma_s}(s,[Z]),$ it satisfies
   \begin{itemize}
        \item  the cycle $[W]\in \overline{\Gamma}_{\sigma}$;
       \item if $[W]\in \overline{\Gamma}_{\sigma}\setminus\Gamma_{\sigma}$, then it can be realized as the central fiber of a non-trivial family of degenerations. 
   \end{itemize}
      
 This set is introduced to rule out the phenomenon of “splitting of orbits,” 
namely that the closure of a union of group orbits may be strictly larger than 
the union of the closures of the individual orbits. Following \cite{donaldson14}, we now explain the notion of a family of (analytic) 
degenerations. Let $U$ be an open subset of $S$, let $\Delta$ denote the unit disc 
in $\mathbb{C}$, and $\Delta^*$ the punctured disc. A flat family of degenerations 
parameterized by $U$ consists of a normal variety $\mathcal{X}$, a flat morphism 
\[
\mathcal{X} \to U \times \Delta
\]
together with a line bundle $\mathcal{L} \to \mathcal{X}$, which is very ample when 
restricted to each fiber. Moreover, the restriction of $(\mathcal{X},\mathcal{L})$ 
to $U \times \Delta^*$ is isomorphic to the polarized family obtained by first 
restricting $(\underline{X}, L)$ to $U$ and then pulling it back to $U \times \Delta^*$. 
In particular, for each $s \in U$, the restriction of $(\mathcal{X},\mathcal{L})$ to 
$\{s\} \times \Delta$ defines a degeneration of $(X_s,L_s)$. Let $\mathcal{X}_s$ denote this degeneration of $X_s$. $\mathcal{X}$ is said to be non-trivial, if for each $s\in U$, $\mathcal{X}_s$ is not isomorphic to a product.

Note that in the original definition, $\mathcal{X}$ is not required to be normal. 
However, in the applications below we will only consider the case where the fiber of $\mathcal{X}$ over $(\sigma,0)$ is normal. 
By \cite[Theorem 1.100 and Theorem 1.101]{GLS}, there exists a small neighborhood $U'$ of $\sigma$ such that the fibers of $\mathcal{X}$ over $(u,0)$ are normal for all $u \in U'$, and $\mathcal{X}$ itself is normal. 
Hence, without loss of generality, we include the normality requirement in the definition. Moreover, in \cite[Definition~3]{donaldson14}, it is not stated explicitly that the degeneration must be non-trivial when $[W]\in \overline{\Gamma}_{\sigma}\setminus\Gamma_{\sigma}$, but this follows from the argument given in \cite[Section~5.4]{donaldson14}.

 It is proved in \cite[Proposition~1]{donaldson14} that $GD$ contains a nonempty Zariski open subset $GD' \subset S$. 
Our goal is to show
\begin{equation}\label{contain a zariski open subset}
    GD' \subset S^*.
\end{equation}
Since $GD'$ is Zariski open and $S^* \neq \emptyset$ is open in the Euclidean topology, we know that
\begin{equation}
    GD' \cap S^* \neq \emptyset.
\end{equation}
Moreover, $GD'$ is connected in the Euclidean topology \cite[Theorem~7.1]{shafarevich2013}. 
Therefore, to prove \eqref{contain a zariski open subset}, it suffices to show that 
$GD' \cap S^*$ is closed in $GD'$ with respect to the Euclidean topology. 
In fact, we will establish the stronger statement that if a sequence $s_j \in S^*$ converges (in the Euclidean topology) to some $\sigma \in GD$, then $\sigma$ also lies in $S^*$.

Write $(X_j,L_j)=(X_{s_j},L_{s_j})$, then choosing $L^2$-orthonormal basis of $H^0(X_j,L_j)$, we get embeddings $T_j:X_j\rightarrow \mathbb{CP}^N$. Taking a subsequence, we can assume $T_j(X_j)$ converges to some algebraic cycle $W$, which by Proposition \ref{algebraic convergence} is isomorphic to $X_{\infty}$, the polarized Gromov-Hausdorff limit of $X_j$. In particular, $W$ is normal and with only quotient singularities. There are two cases due to the assumption $\sigma\in GD$:
\begin{itemize}
    \item $[W]\in \Gamma_{\sigma}$. Then this would imply $X_{\infty}$ is biholomorphic to $X_{\sigma}$. Then the cscK orbifold metric on $X_{\infty}$ would be smooth and therefore $\sigma\in S^*$.
    \item  $[W]\in \overline{\Gamma}_{\sigma}\setminus\Gamma_{\sigma}$. Moreover we get a non-trivial family of degenerations $\mathcal{X}$ parametrized by $U$, where $U$ is small neighborhood of $\sigma$ in $S$. We want to derive a contradiction in this case. 

    For an analytic degeneration $(\mathcal X,\mathcal L)\to\Delta$ of a given polarized manifold $(X,L)$, one can still define certain  Futaki type numerical invariant as the usual case of test configurations; see \cite[Section 2.3]{donaldson14} and \cite[Proposition 11]{szekelyhidi2015}. Denote this by $\widetilde{\mathrm{Fut}}(\mathcal X)$.
  Choose $j$ large so that $s_j \in U$. Since the restriction of $\mathcal{X}$ to $s_j \times \Delta$ is non-trivial, the results of \cite{Stoppa2009,szekelyhidi2015} and \cite[Proposition 5]{donaldson14} imply 
\begin{equation}
    \widetilde{\mathrm{Fut}}(\mathcal{X}_{s_j}) > 0.
\end{equation}
 On the other hand, as noted by Donaldson, the differential-geometric inputs ensure that 
\begin{equation}\label{futaki non-positive}
\widetilde{\mathrm{Fut}}(\mathcal{X}_{\sigma}) \leq 0.
\end{equation}
However, $\widetilde{\mathrm{Fut}}(\mathcal{X}_{s})$ is constant for $s \in U$ from the definition, yielding a contradiction.
\end{itemize}
In the following, we give the proof of \eqref{futaki non-positive}. Since Theorem~\ref{algebraic convergence} holds, \cite[Lemma~3]{donaldson14} remains applicable. 
Thus, by the argument in \cite[Section 3.1 and Section 3.2]{donaldson14}, the inequality \eqref{futaki non-positive} reduces to showing that
\begin{equation}\label{inequ to be proved}
   \limsup_{k\rightarrow\infty} k\int_X \Big|\frac{\rho_k}{P(k)} \, \omega^2 - \omega_{FS,k}^2\Big| \leq 0,
\end{equation}
where 
\begin{itemize}
    \item $X = X_s$ for some $s \in S^*$, and $\omega \in c_1(L_s)$ is the cscK metric;
    \item $L^2$-orthonormal holomorphic sections are used to define the embedding $T_k: X \to \mathbb{P}^{N_k}$, 
    and $\omega_{FS,k}$ is defined as $k^{-1}$ times the pull-back of the Fubini--Study metric via $T_k$;
    \item $\rho_{k}=\rho_{k,X} = \sum_{\alpha} |s_\alpha|^2=1+\frac{S_{\omega}}{4\pi}k^{-1}+\eta_k$ is the Bergman kernel function;
    \item $P(k) = \frac{\dim H^0(X,L^k)}{\int_X \frac{\omega^2}{2} k^2} = 1 + \frac{S_{\omega}}{4\pi} k^{-1} + \alpha_2 k^{-2}$.
\end{itemize}
By the definition of the Fubini-Study metric and the Bergman kernel, we have 
\begin{equation}
    \omega_{FS,k}=\omega+(2\pi k)^{-1}\ii\partial\pp \log \rho_k.
\end{equation}
As in \cite{donaldson14}, we write $\omega_{FS,k}^2=(1+F_1)\omega^2$ and $\frac{\rho_k}{P(k)}=1+F_2$. Then one can easily show that for any $r>0$, 
\begin{equation}\label{estimate of chow weights}
     \int_X \Big|\frac{\rho_k}{P(k)} \, \omega^2 - \omega_{FS,k}^2\Big| \leq 2\int_{\Omega_r}(|F_1|+|F_2|)\frac{\omega^2}{2}+2\vol(Z_r).
\end{equation}
Let $\beta$ be the constant given in Theorem \ref{effective bergman function bound}. Then for $k$ large, we let $r_0\in (0,10^{-1})$ the number such that $k=r_0^{-2}(\log r_0^{-1})^{\beta}$ and choose $r=r_0$ in \eqref{estimate of chow weights} to estimate the left hand side. The second term can be bounded by 
\begin{equation}\label{estimate for the second term}
    Cr_0^4\leq C\frac{(\log r_0^{-1})^{4\beta}}{k^2}\leq C\frac{(\log k)^{4\beta}}{k^2}.
\end{equation} To estimate the first term in \eqref{estimate of chow weights}, we use the Fubini theorem as in \cite{donaldson14}. Note that the estimate in Theorem \ref{effective bergman function bound} implies that for any $r\geq r_0$, $|F_2|\leq Ck^{-2}r^{-4}$ and 
\begin{equation}
\begin{aligned}
        |F_1|\leq & C \left(\frac{|\Delta \log \rho_k|}{k}+\frac{|\partial\pp \log \rho_k|^2}{k^2}\right)\leq   Ck^{-2}r^{-4}
\end{aligned}
\end{equation}Write $|F_1|+|F_2|=k^{-2}F$, then we have $F\leq Cr^{-4}$ on $\Omega_r$.
Then we can estimate
\begin{equation}\label{estimate for the first term}
\begin{aligned}
        \int_{\Omega_{r_0}}(|F_1|+|F_2|)\frac{\omega^2}{2}&= Ck^{-2}\int_0^{r_0^{-4}}\vol(\{x\in \Omega_{r_0}\mid F(x)\geq Ct\})dt\\
        &\leq Ck^{-2}\int_0^{r_0^{-4}}\vol(\Omega_{r_0}\cap Z_{Ct^{-1/4}})dt\\
        &\leq Ck^{-2}(1+\log r_0^{-4})\leq Ck^{-2}(1+\log k).
\end{aligned}
\end{equation}
Then the desired inequality \eqref{inequ to be proved} follows from \eqref{estimate for the first term} and \eqref{estimate for the second term}.
\end{proof}


\section{Non-collapsing implies Sobolev constant bound}\label{sec-noncollapsing imply sobolev}
As mentioned in the introduction, we are going to prove that the \emph{a priori} stronger Sobolev bound is equivalent to the weaker volume non-collapsing condition. The key input here is the compactness results established in \cite{tian-viaclovsky2}, which avoid the assumption on the Sobolev constant bound.
We remark here that the polarization assumption is used only to derive topological finiteness and a uniform $L^2$-bound on the curvature via Matsusaka’s big theorem. 
The following argument remains valid as long as one imposes a uniform bound on the first Betti number and a uniform $L^2$-bound on the curvature, which is required for the compactness results in \cite{tian-viaclovsky1,tian-viaclovsky2}.

\begin{theorem}[\cite{tian-viaclovsky2}]\label{thm-compact for non-collapsing}
   Given any sequence $(X_j,\omega_j,p_j)\in\widetilde{\mathcal K}(V,\kappa)$, the following holds:
    \begin{itemize}
        \item  by passing to some subsequence, it converges to a polarized cscK orbifold in the polarized $\hat C^{\infty}$-Cheeger-Gromov sense; 
        \item given any sequence of scales $\lambda_j\to\infty$, by passing to some subsequence,  the blow-up sequence $(X_j,\lambda_j^2\omega_j,p_j)$ converges to an ALE scalar-flat orbifold in the pointed $\hat C^{\infty}$-Cheeger-Gromov sense.
\item  Moreover, both the number of orbifold points and the orders 
of the local orbifold groups on the limit are bounded by a quantity depends only $V$ and $\kappa$. 
    \end{itemize}
\end{theorem}

We include here some explanation of the derivation of Theorem \ref{thm-compact for non-collapsing}. There are two major observations used in \cite{tian-viaclovsky2} to overcome the difficulty of lacking a priori Sobolev bound. Firstly, under the volume non-collapsing condition, a \textit{non-quantitative} $\epsilon$-regularity theorem \cite[Section 3]{tian-viaclovsky2} can be proved by a blow-up argument via a point selection process. Here being non-quantitative means that the geometry (say local Sobolev bound or curvature) is uniformly bounded when energy is small, but not bounded in terms of energy. But then one can run PDE arguments in \cite{tian-viaclovsky0} to upgrade into a quantitative $\epsilon$-regularity theorem. This allows us to understand the geometry in the regular region. Secondly, near singular set, the derivation of a Sobolev constant bound relies only on the information of metrics at the $C^0$-level, which in turn only depends on the knowledge of cone structure. For instance, to prove the removable singularity theorem in this setting, we first show that the tangent cone is uniquely isometric to $\mathbb C^2/\Gamma$. This only requires geometric input, namely tangent cones are flat and volume non-collapsing, but no further analytic ingredients. From the knowledge of tangent cones, it's standard to see there is a $C^0$-extension of metric across singularity \cite[Proposition 5.10]{DS1}, so in particular there is a uniform Sobolev bound around curvature singularity. From this we can go to higher regularities, as established in previous theory \cite{tian-viaclovsky1}. The same idea applies, also to the ALE structure at infinity as well as neck structural result.

With Theorem~\ref{thm-compact for non-collapsing} in hand, we can apply essentially the same argument as in the proof of Theorem \ref{thm-bubble decomposition} to obtain the bubbling tree structure for any sequence $\{X_j\} \subset \tilde{K}(V,\kappa)$. 
The only new ingredient is the following gap theorem for ALE scalar-flat K\"ahler orbifolds, which is used in the proof of Lemma~\ref{lem-no intermediate scales} and need to be re-established without assuming a universal bound on Sobolev constant.
\begin{lemma}
    There exists $\epsilon_0=\epsilon_0(\kappa)$ such that the following holds. Given $(Z,\omega_Z)$ an ALE K\"ahler scalar-flat orbifold with at most one orbifold point, and $\mathrm{Vol}(B(q,r))\geq\kappa r^4$ for any $B(q,r)\subset Z$. Then if $\Vert \mathrm{Rm}(\omega_Z)\Vert_{L^2(Z)}<\epsilon_0$, then $(Z,\omega_Z)$ is a flat cone.
\end{lemma}
\begin{proof}
    Let $\Gamma_o$ be the orbifold group at the unique orbifold singularity, and $\Gamma_\infty$ be the orbifold group at infinity. By Gau\ss-Bonnet-Chern formula and signature theorem, we have
    \begin{equation}
        \begin{aligned}
            &\chi(Z)=\frac{1}{8\pi^2}(\int_Z\vert W^-\vert^2-\frac{1}{2}\int_Z\vert \mathrm{Ric}\vert^2)+1-\frac{1}{\vert\Gamma_o\vert}+\frac{1}{\vert\Gamma_\infty\vert},\\
            &\sigma(Z)=-\frac{1}{12\pi^2}\int_Z\vert W^-\vert^2-\eta(S^3/\Gamma_o)+\eta(S^3/\Gamma_\infty).
        \end{aligned}
        \end{equation}

        The K\"ahler condition implies that $\Gamma_o$ and $\Gamma_\infty$ are all contained in $U(2)$. Since the volume ratio is uniformly bounded from below by $\kappa$, their cardinality is uniformly bounded. Therefore, there are only finitely many isomorphism classes of such groups \cite[Table 1]{lock-viaclovsky16}. In particular, the signature theorem implies that $\int_Z\vert W^-\vert^2$ is discrete, so by choosing $\epsilon_0$ small, we can assume that $W^-\equiv 0$. Now Gau\ss-Bonnet-Chern formula tells us $\int_Z\vert \mathrm{Ric}\vert^2$ is also discrete, so has to be zero for $\epsilon_0$ small. Therefore $(Z,\omega_Z)$ is flat, hence being a flat cone.
\end{proof}

Therefore, given a sequence $X_j \subset \tilde{K}(V,\kappa)$, after passing to a subsequence, Theorem~\ref{thm-bubble decomposition} and Corollary~\ref{cor-no intermediate scales} apply to $X_j$. 
In what follows, we use the same notation as in Section~\ref{sec--bubble}.

In the proof of Theorem~\ref{local sobolev}, which will be given below, 
we need to show that the neck region supports a universal Sobolev bound. 
For this purpose, we establish a weak structure theorem, which can be constructed directly without referring to the refined analysis in \cite{ozuch2022}. 
Let $Z_I$ be a bubble limit obtained in Theorem~\ref{thm-bubble decomposition}, 
and let $\Gamma_I \subset U(2)$ be the finite subgroup such that the tangent cone at infinity of $Z_I$ is given by
\[
\mathcal{C}_\infty(Z_I) = \mathbb{C}^2 / \Gamma_I .
\]
Then we have the following weak approximation along the neck region. 
Here the term \textit{``weak''} means that the metric is close to a conical metric, but without any quantitative rate of convergence.

\begin{proposition}\label{prop-weak approximation}
    There exist $1<K_0<K_1$, such that we can find smooth embeddings
   \begin{equation}
       \Phi_{j,I}: A_{p_{j}}(K_0 \lambda_{j,I}^{-1},K_1^{-1}\lambda^{-1}_{j,P(I)})\to\mathbb C^2/\Gamma_I,
   \end{equation} such that for $k\in \mathbb N_{\geq 0}$
   \begin{equation}
      \sup_{s\geq R}s^k\left((|\nabla^k_{g_{\Gamma}}((\Phi_{j,I}^{-1})^*J_j-J_{\Gamma_I})|_{g_{\Gamma_I}}+ |\nabla^k_{g_{\Gamma}}((\Phi_{j,I}^{-1})^*(\lambda_{j,I}^2g_j)-g_{\Gamma_I})|_{g_{\Gamma_I}}\right)=\Psi(R^{-1}|k).
   \end{equation}
\end{proposition}
\begin{proof}
    This follows from Corollary~\ref{cor-no intermediate scales} and a gluing procedure; for example, see \cite[Lemma~3.1]{SunZ}.
\end{proof}

Now we are ready to prove our main theorem in this section.
\begin{theorem}\label{local sobolev}
    There exist constants $C_S$ and $r_0$, depends only on $V$ and $\kappa$ such that for any $(X,\omega)\in \widetilde{\mathcal{K}}(V,\kappa)$, we have
    \begin{equation}
        C_{\omega}(r_0)\leq C_S.
    \end{equation}
\end{theorem}

\begin{proof}
   We argue by contradiction. 
Suppose there exists a sequence $r_j \to 0$ and points $p_j \in X_j$ such that the local Sobolev constants 
$C_S(p_j, r_j) \to \infty$. We consider the rescaled  metric $$\tilde{g}_j=r_j^{-2}g_j$$ and we get functions $f_j\in C^{\infty}_c(B_{\tilde{g}_j}( p_j,1))$ with $\|f_j\|_{L^4}=1$ and $\|\nabla f_j\|_{L^2}\rightarrow 0$, where all the norm and derivative are taken with respect to the rescaled metric $\tilde{g}_j$. By passing to a subsequence we know that $(X_j,\tilde{g}_j,p_j)$ converges to an ALE scalar-flat K\"ahler orbifold $(Z_0,x_0)$ in the pointed $\hat C^{\infty}$--Cheeger--Gromov sense. 
    If the convergence to $Z_0$ is smooth, then we would get a contradiction, since Sobolev inequality holds for ALE scalar-flat K\"ahler manifold. Otherwise we need to deal with the case where $Z_0$ has curvature singularities. The idea is to choose suitable cut-off functions along neck region, such that the new function has definite $L^4$-norm, but with small $L^2$-gradient. We arrive again at a contradiction if such process continues to the deepest bubble, where the convergence is smooth.

    According to the discussion above, by passing to a subsequence, we can assume the existence of bubble tree and the structural result for neck regions for convergence around $p_j$. Let $Z_I$ denote all bubble limits for this sequence and let $N$ denote the number of bubbles. For simplicity, we may assume each $Z_I$ contains only one curvature singularity, denoted by $x_I$, so the bubble tree is non-branching. The general case can be proved in the same manner. We may index $\mathcal T$ now as $\{1,2,\cdots, N\}$. Note that the Sobolev inequality holds for ALE scalar-flat K\"ahler orbifold $Z_I$ and let 
    $$C=2N(\max_{0\leq I\leq N}\{C_S(Z_I)\}+C_S(\mathbb R^4)+1).$$

    Now let $p_{j,1}\in X_j$ be the sequence of points converging to $x_1$, and $\lambda_{j,1}\to\infty$ be the sequence of scales, such that $(X_j,\lambda_{j,1}^2\tilde{g}_j,p_{j,1})$ converges to $(Z_1, x_1)$. Then we claim that for any $r>0$, 
    \begin{equation}\label{eq-pass to small scale}
       \lim_{j\rightarrow\infty}\|f_j\|_{W^{1,2}(B_{\tilde{g}_j}(p_{j,1},1)\setminus B_{\tilde{g}_j}(p_{j,1},r))}=0. 
    \end{equation}
To see this, we may assume that $p_{j,1}$ converges to $x_{0,1} \in Z_0$ under the pointed Gromov--Hausdorff convergence
\(
(X_j, \tilde{g}_j, p_j) \to (Z_0, x_0).
\) Then we choose cut-off functions $\chi_{\epsilon,r}$ such that $\chi_{\epsilon,r}\equiv 0$ on $B_{g_{Z_0}}(x_{0,1},\epsilon)$ and $\chi_{\epsilon,r}\equiv 1$ outside $B_{g_{Z_0}}(x_{0,1},r/2)$, where $0<\epsilon\ll r$. We can thus identify $\chi_{\epsilon,r} f_j$ as a function on $Z_0$, so by applying Sobolev inequality on $Z_0$, we see that
\begin{equation}
    \begin{aligned}
        \|f_j\|_{L^4(B_{\tilde{g}_j}(p_{j,1},1)\setminus B_{\tilde{g}_j}(p_{j,1},r))}&\leq  (1+\Psi(j^{-1}))\Vert f_j\Vert_{L^4( A_{x_{0,1}}(r/2,1))}\leq C\Vert\chi_{\epsilon,r} f_j\Vert_{L^4(B_{g_{Z_0}}(x_{0,1},1))}\\
         &\leq C\Vert\nabla(\chi_{\epsilon,r} f_j)\Vert_{L^2(B_{g_{Z_0}}(x_{0,1},1))}\\
         &\leq C(\Vert\nabla\chi_{\epsilon,r}\Vert_{L^4( A_{x_{0,1}}(\epsilon,r))}\Vert f_j\Vert_{L^4(A_{x_{0,1}}(\epsilon,r))}
         +\Vert\nabla f_j\Vert_{L^2( A_{x_{0,1}}(\epsilon,1))}).
    \end{aligned}
\end{equation}
Note that given any $r>0$, using the logarithmic cut-off trick, which will be exploited frequently below, we can achieve 
\begin{equation}
    \Vert\nabla \chi_{\epsilon,r}\Vert_{L^4}=\Psi(\epsilon|r).
\end{equation}
More precisely, we take 
\begin{equation}
 \chi_{\epsilon,r}(x)= \begin{cases}0, & \text { if } d_{Z_0}(x_0,x)\leq \epsilon  \\ \frac{\log d_{Z_0}(x_0,x)-\log \epsilon}{-\log \sqrt{\epsilon}}, & \text { if } d_{Z_0}(x_0,x)\in [\epsilon,\sqrt{\epsilon}], \\
 1, & \text { if } d_{Z_0}(x_0,x)\geq \sqrt{\epsilon}.
 \end{cases}
\end{equation}
Then, using the assumption that $\|f\|_{L^4}=1$,  $\|\nabla f_j\|_{L^2} \to 0$, we first let $j \to \infty$ and then let $\epsilon \to 0$. Applying H\"older's inequality, we obtain \eqref{eq-pass to small scale}.

By structural theorem of neck region, we know that there exists $K_0$ and $K_1$ such that for any $j\geq j_0$, the metric $\tilde{g}_j$ on the annulus $ A_{p_{j,1}}(K_0\lambda_{j,1}^{-1}, K_1^{-1})$ is close to a cone metric in the sense of Proposition \ref{prop-weak approximation} . 
We write 
\begin{equation}
    f_j=\rho_{j,1}f_j+(1-\rho_{j,1})f_j,
\end{equation}where $\rho_{j,1}$ is a cut-off function supported on $B_{\tilde{g}_j}(p_{j,1},r_1)$ and equals 1 on $B_{\tilde{g}_j}(p_{j,1},r_0)$ for some $K_0\lambda_{j,1}^{-1}<r_0<r_1<K_1^{-1}$ such that 
\begin{equation}
    \|\nabla_{\tilde{g}_j}\rho_{j,1}\|_{L^4}\leq \frac{1}{(2C)^{N+1}}.
\end{equation}
To achieve this, we use the logarithmic cut-off trick again. Then by \eqref{eq-pass to small scale} we know that with respect to the metric $\tilde{g}_j$, we have
\begin{equation}
\|\rho_{j,1}f_j\|_{L^4}\geq 1-\Psi(j^{-1}),
\end{equation}and 
\begin{equation}
    \|\nabla(\rho_{j,1}f_j)\|_{L^2}\leq \frac{1}{(2C)^{N+1}}+\Psi(j^{-1}).
\end{equation}
Let $g_{j,1}=\lambda_{j,1}^2\tilde{g}_j$, and we write 
\begin{equation}
    \lambda_{j,1}^{-1}\rho_{j,1}f_j=\chi_j \lambda_{j,1}^{-1}\rho_{j,1}f_j+(1-\chi_j)\lambda_{j,1}^{-1}\rho_{j,1}f_j,
\end{equation}where $\chi_j$ is a cut-off function supported on $B_{g_{j,1}}(p_{j,1},K)$ for a large constant $K> K_0$ and equals 1 on $B_{g_{j,1}}(p_{j,1},K_0)$  and such that 
\begin{equation}
\|\nabla_{g_{j,1}}\chi_j\|_{L^{4}}\leq \frac{1}{(2C)^{N+1}}.
\end{equation}
Since the metric is close to a cone metric along the neck region, for sufficiently large $j$ its Sobolev constant can be chosen to be less than $C_S(\mathbb{R}^4)$.
Then we know that for $j$ sufficiently large, with respect to the metric $g_{j,1}$, we have
\begin{equation}
\|(1-\chi_j)\lambda_{j,1}^{-1}\rho_{j,1}f_j\|_{L^4}\leq 2C_{S}(\mathbb R^4)\|\nabla ((1-\chi_j)\lambda_{j,1}^{-1}\rho_{j,1}f_j))\|_{L^2}\leq \frac{1}{(2C)^{N}}+\Psi(j^{-1}).
\end{equation}
Let $f_{j,1}=\chi_{j,1}\lambda_{j,1}^{-1}\rho_{j,1}f_j$. Then it is supported on $B_{g_{j,1}}(p_{j,1},K)$ and with respect to the metric $g_{j,1}$, we have
\begin{equation}
    \|f_{j,1}\|_{L^4}\geq 1-\frac{1}{(2C)^{N}}+\Psi(j^{-1})
\end{equation} and 
\begin{equation}
    \|\nabla f_{j,1}\|_{L^2}\leq \frac{1}{(2C)^{N}}+\Psi(j^{-1}).
\end{equation}
We can now repeat this process and get that for $j$ large and for each $k=1,\cdots,N$ (recall that we make the assumption that each $Z_I$ contains only a unique orbifold point for $I=1,\cdots,N-1$ and $Z_N$ is smooth), we have  functions $f_{j,k}$ supported on $B_{g_{j,k}}(p_{j,k}, K)$ for some large constant $K$ (independent of $j$), such that with respect to the metric $g_{j,k}$, we have
\begin{equation}
        \|f_{j,k}\|_{L^4}\geq 1-\sum_{i=1}^k\frac{1}{(2C)^{N+1-k}}+\Psi(j^{-1})
\end{equation}and 
\begin{equation}
    \|\nabla f_{j,k}\|_{L^2}\leq \frac{1}{(2C)^{N+1-k}}+\Psi(j^{-1})
\end{equation}
Since $Z_N$ is smooth and we have smooth convergence, then we will get a contradiction when $j$ large by the choice of the constant $C$.
\end{proof}

\noindent \textit{Proof of Theorem \ref{thm--non-collapsing implies Sob bound}}
This follows from Theorem \ref{local sobolev}, and a standard argument using cut-off functions and the volume ratio upper bound of geodesic balls.
\qed

\begin{remark}
    The idea of using bubbling analysis to derive Sobolev bounds also appeared in \cite[Theorem 1.2]{chen-wang12}, where a uniform isoperimetric constant bound was derived for Ricci flows.
\end{remark}



\section{Discussions and questions}\label{sec--discussion}

\subsection{Partial $C^0$-estimate and convergence within a Hilbert scheme}\label{began function lower bound}
As proved in Section~\ref{convergence in a flat family}, the key to realizing the Gromov--Hausdorff convergence in a flat family is to verify that the dimension of holomorphic sections does not jump in the limit. This property is a priori \textit{weaker} than the partial $C^0$-estimate. Moreover, as noted in Remark~\ref{gradient estimate imply lower bound}, if a suitable gradient estimate for holomorphic sections is available, then the no-jumping property of numerical invariants further implies a uniform lower bound for the Bergman function.

In the setting considered in this paper, due to the absence of a uniform Ricci curvature bound, 
a global gradient estimate for holomorphic sections seems to be  unavailable. Consequently, obtaining a 
uniform lower bound for the Bergman function becomes a subtle problem, even though we have achieved that every 
holomorphic section on the limit arises as a limit of holomorphic sections from 
the converging sequence. 
Therefore, Theorem~\ref{algebraic convergence} may be interpreted as stating that there are 
enough holomorphic sections coming from limits to identify the Gromov--Hausdorff limits with 
the algebraic limits, even if the approximating sequence may lack a uniform lower bound for the 
Bergman function. 
Moreover, the following question still appears to be of independent interest:
\begin{problem}
    Can we find a uniform integer $k_0$ and positive number $b_0$ depending only on $V$ and $C_S$, such that  for all $(X,\omega_X)\in\mathcal K(V,C_S)$
    \begin{equation}
        \inf_{X}\rho_{k_0,X}\geq b_0^2?
    \end{equation}
\end{problem}

\

In this paper, the proof of Theorem~\ref{thm-dim constant}, which shows that the dimension of holomorphic sections does not jump, is global in nature and relies on special topological features of four-manifolds, such as the Chern--Gau\ss--Bonnet formula and the signature theorem. It is natural to ask whether a more analytic argument exists, which would be important in the study of local convergence of cscK metrics and other problems in geometric analysis. We make some attempt here, trying to bypass the topological argument above, although the argument below is \textit{not} enough to establish the full conclusion as before.

Making use of the quasi-psh functions constructed in Theorem~\ref{thm--psh weight} and using only the H\"ormander's $L^2$-estimate, 
we can show that, on the limit space, any section with sufficiently high vanishing order at the orbifold points 
can be realized as the limit of sections on the approximating sequence. 
In what follows, we denote by $S = X_{\infty}^{\mathrm{Sing}}$ the set of orbifold points on $X_{\infty}$, 
and by $\mathcal{I}_S$ the ideal sheaf of functions vanishing along $S$.
\begin{proposition}\label{main analytic input}
    There exist integers $k_0, N_0$, such that for any $k\geq k_0$ and  $s_\infty\in H^0(X_\infty, L^k_\infty\otimes\mathcal{I}_S^{N_0})$, there exists a subsequence of holomorphic sections $s_{j_l}\in H^0(X_{j_l},L^{k}_{j_l})$ converging to $s_\infty$ in $C^{\infty}_{loc}$.
\end{proposition}
\begin{proof} We use the same notations as that in Theorem \ref{global weights}.
  Then for any $k\geq 2k_0$, we have
        \begin{equation}
            \Ric_{\omega_j}+\sqrt{-1}\partial\pp u_j+k\omega_j\geq(k-k_0+\frac12)\omega_j\geq\frac{k}{2}\omega_j.
        \end{equation}
        From this we can apply H\"ormander's estimate \cite[Chapter VIII, Theorem 6.1]{Demailly} to the line bundle $L_j^k\otimes K_{X_j}^{-1}$ using the Hermitian metric $h_j^{\otimes k}\cdot e^{-u_j}\otimes\omega_j$ over the complete K\"ahler manifold $(X_j,\omega_j)$. 

        Again to simplify notation, we will assume throughout that 
$X_{\infty}$ has a single orbifold point $p_{\infty}$.  Passing to some subsequence, we can find scales $\epsilon_m\to 0^+$. such that for any $j\geq m$ under the identification by $(\chi_j, \hat{\chi}_j)$, we have
        \begin{equation}
            \lambda_{j,1}^{-1}\leq  \epsilon_m^2,\ \vert A_j-A_\infty\vert_{L^{\infty}(X_j\setminus B(p_{j,1},\epsilon_m))}+\vert J_j-J_\infty\vert_{L^{\infty}(X_j\setminus B(p_{j,1},\epsilon_m))}\leq\epsilon_m,
        \end{equation}
        where $\lambda_{j,1}$ is the scale corresponding to $p_{j,1}$, the root vertex in the tree $\mathcal{T}$; see Theorem \ref{thm-bubble decomposition}.


        Then we choose a sequence of cut-off functions $\chi_j$, such that $\chi_j\equiv 1$ on $X_j\setminus B(p_{j,1},2\epsilon_j)$, $\chi_j\equiv 0$ on $B(p_{j,1},\epsilon_j)$ and $\vert\nabla\chi_j\vert_{\omega_j}\leq\frac{2}{\epsilon_j}$. For any section $s_\infty\in H^0(X_\infty, L_\infty^k\otimes\mathcal I_{S}^{N_0})$, where $N_0$ is to be determined below, via the bunlde isomorphism $\hat{\chi}_j$, we can view $\chi_j\cdot s_\infty$ as an almost holomorphic section to $X_j$. Let $\eta_j=\pp_j(\chi_j\cdot s_\infty)$, where $\pp_j=\pp_{J_j}$ denotes the complex structure on $X_j$.
        
        \noindent\textbf{Claim}: There exists constant $C$ and $\alpha>0$ such that for $k$ and $j$ large,
        \begin{equation}
\int_{X_j}\vert\eta_j\vert^2_{h_j^{\otimes k}}e^{-u_j}\omega_j^2\leq Ck\epsilon_j^\alpha.
        \end{equation}

         Granted this, we can find $\zeta_j$ such that $\pp_j\zeta_j=\eta_j$ with estimate 
         \begin{equation}
             \int_{X_j}\vert\zeta_j\vert^2_{h_j^{\otimes k}}\omega_j^2\leq \int_{X_j}\vert\zeta_j\vert^2_{h_j^{\otimes k}}e^{-u_j}\omega_j^2\leq \frac{2}{k} \int_{X_j}\vert\eta_j\vert^2_{h_j^{\otimes k}}e^{-u_j}\omega_j^2\leq C\epsilon_j^{\alpha}.
         \end{equation}
         By standard elliptic estimate, we also have high order estimate for $\zeta_j$ outside any definite sized small balls centered at $p_{j,1}$. Therefore by passing to some subsequence, we get holomorphic sections $s_j=\chi_j\cdot s_\infty-\zeta_j$  converging to $s_\infty$ both in $C^{\infty}_{loc}$ and globally in $L^2$.

         To estimate our integrand, we split it into two parts. Note that we have
         \begin{equation}
             \eta_j=\pp_j\chi_j\wedge s_\infty+\chi_j\cdot\pp_j s_\infty,
         \end{equation}
         so
         \begin{equation}
             \begin{aligned}
                 \int_{X_j}\vert\eta_j\vert^2_{h_j^{\otimes k}}e^{-u_j}\omega_j^2\lesssim &\int_{A_{p_{j,1}}(\epsilon_j,2\epsilon_j)}\vert\nabla\chi_j\vert^2\vert s_\infty\vert^2 e^{-u_j}\omega_j^2\\
                 &+\int_{X_j\setminus B(p_{j,1},\epsilon_j)}\chi_j^2\vert(\pp_j-\pp_\infty)s_\infty\vert^2e^{-u_j}\omega_j^2.
             \end{aligned}
         \end{equation}
         The first term can be estimated as follows:
         \begin{equation}
             \begin{aligned}
                 \int_{\mathcal A_{p_{j,1}}(\epsilon_j,2\epsilon_j)}\vert\nabla\chi_j\vert^2\vert s_\infty\vert^2 e^{-u_j}\omega_j^2&\lesssim\epsilon_j^{-2}{\int_{ A_{p_{j,1}}(\epsilon_j,2\epsilon_j)}\ d(\cdot,p_{j,1})^{2N_0-L}\omega_j^2}\\
                 &\lesssim\epsilon_j^{2N_0-L-2}\cdot \mathrm{Vol}(X_j,\omega_j)\sim\epsilon_j^{2N_0-L-2}.
             \end{aligned}
         \end{equation}
         Here we use the fact that $\lambda_{j,1}^{-1}\ll\epsilon_j$, so all the distances $d(\cdot,p_{j,\mathcal I})$ are comparable to $d(\cdot,p_{j,1})$ on the annulus $\mathcal A_{p_{j,1}}(\epsilon_j,2\epsilon_j)$.
         For the second term, we can bound it by 
         \begin{equation}
             \begin{aligned}
                & k\left(\vert J_j-J_\infty\vert_{L^{\infty}(X_j\setminus B(p_{j,1},\epsilon_j))}^2+\lesssim\vert A_j-A_\infty\vert_{L^{\infty}(X_j\setminus B(p_{j,1},\epsilon_j))}^2\right)\int_{X_j\setminus B(p_{j,1},\epsilon_j)}\vert s_\infty\vert^2 e^{-u_j}\omega_j^2\\
&\lesssim k\epsilon_j^2(1+\int_{ A_{p_{j,i}}(\epsilon_j,1)}d(\cdot,p_{j,1})^{2N_0-L}\omega_j^2)\lesssim k\epsilon_j^2(1+\epsilon_j^{2N_0-L}).
             \end{aligned}
         \end{equation}
         The claim follows by choosing $N_0>\frac12(L+2)$.
         
         \end{proof}


Without assuming Theorem~\ref{thm-dim constant}, one can still relate the Gromov–Hausdorff limit $X_\infty$ to algebraic limits of $X_j$ \cite{donaldson2010}. 
Let $T_{k,j}$ be the Kodaira embedding of $X_j$ via $L_j^k$ using an $L^2$-orthonormal basis. By Lemma~\ref{L-infinity estimate} and Proposition~\ref{main analytic input}, a subsequence converges in $C^\infty_{\mathrm{loc}}$ to 
\[
T_k : X_\infty^{\mathrm{reg}} \to \mathbb{CP}^{N_k},
\] 
a holomorphic embedding that coincides with the Kodaira map defined by $V_k$, where $V_k \subset H^0(X_\infty, L_\infty^k)$ consisting of sections arising as limits from $X_j$.
 By compactness of Hilbert scheme, we may assume $T_{k,j}(X_j)$ converges to a scheme ${W}_k$ in the analytic topology.

 As noted in \cite{donaldson2010}, a fundamental difficulty is that ${W}_k$ may fail to be irreducible. This can be overcome by proving a uniform lower bound on Bergman kernels, a major achievement established in \cite{DS1}. Given Proposition \ref{main analytic input}, we show that this irreducibility issue is the only obstruction for realizing convergence within a flat family, even without assuming a uniform lower bound on Bergman kernels.

Recall that $P$ denotes the Hilbert polynomial of $(X_j,L_j)$, which we have assumed to be independent of $j$, and $P_\infty$ denotes the Hilbert polynomial of $(X_\infty,L_\infty)$, where $L_\infty$ is assumed to be a line bundle.
\begin{proposition}
    For any given $k\in \mathbb N_{>0}$,  ${W_k}$ is irreducible if and only if $P(k)=P_{\infty}(k)$. 
\end{proposition}

\begin{proof}
The direction from $P(k) = P_\infty(k)$ to the irreducibility of ${W}_k$ is established in the proof of Theorem~\ref{algebraic convergence} in Section~\ref{convergence in a flat family}. We give the proof of the converse direction.  

Firstly, by the uniform estimates of holomorphic sections on the regular region, ${W}_k$ 
is generically reduced. Let $\widetilde{W}_k$ denote its reduced scheme structure. The key fact that $W_k$ is irreducible implies that $\mathrm{Vol}(W_k)=\mathrm{Vol}(\widetilde{W}_k)=\mathrm{Vol}(X_\infty)$. On the other hand, Proposition \ref{main analytic input} implies that $T_k$ is an embedding on $X_\infty^{\mathrm{reg}}$, hence defines a birational map from $X_\infty$ to $\widetilde{W}_k$. Considering the induced birational map from the normalization of $\widetilde{W}_k$ to $X_{\infty}$, and noting that $X_\infty^{\mathrm{reg}}$ has codimension 2, we apply negativity lemma to compare the volume. It follows that $T_k$ extend across $X_\infty^{\mathrm {sing}}$, and indeed is the normalization of $\widetilde{W}_k$. As a result, $W_k$ has only isolated singularities as schemes.

However, at this stage it's possible that $W_k$ contains some embedded points. We choose a morphism from a smooth curve to the Hilbert scheme and pull back the universal family of Hilbert scheme. We assume that the central fiber is $W_k$, with general fibers smooth and contains \textit{some} $X_j$. By taking the normalization of the total space, we can get a flat family with normal total space
\begin{equation}
    \pi:\mathcal{X}\rightarrow C,
\end{equation}with $\mathcal{X}_0$ reduced and irreducible. It's clear that $X_{\infty}$ is the normalization of $\mathcal{X}_0$.
We can now apply a result of Koll\'ar \cite[Lemma 14.2]{kollar95}, or equivalently, a theorem of Mumford \cite{mumford78} on the non-deformability of certain isolated non-normal surface singularities. According to their result, we must have $X_\infty=\mathcal{X}_0$.  Since $\mathcal X\xrightarrow[]{\pi} C$ is flat, the Hilbert polynomial of $X_\infty$ is the same as all $X_j$.
\end{proof}

\subsection{CscK metrics converging in non-$\mathbb Q$-Gorenstein family}\label{in non-q-gorenstein family} 

It was proved in \cite{DS1,OSS16} that any Gromov--Hausdorff limit $X_\infty$ of smooth Kähler--Einstein Fano manifolds is $\mathbb{Q}$-Gorenstein smoothable. More precisely, there exists a flat projective morphism from a variety $\mathcal X$ to a smooth curve germ $(C,0)$ with smooth general fibers and central fiber $X_\infty$, such that $K_{\mathcal X}$ is $\mathbb Q$-Cartier. Examples of K\"ahler-Einstein Fano manifolds converging in a $\mathbb Q$-Gorenstein family have been constructed in \cite{spotti2014deformations,SSY16,LWX19}. For general cscK metrics, the gluing construction in \cite{BR15} and the openness results in \cite{PTT23} provides examples of cscK metrics converging in a $\mathbb{Q}$-Gorenstein family.

Note that for a sequence of non-collapsed polarized cscK surfaces $(X_j,\omega_j)$ converging to an orbifold $(X_\infty,\omega_\infty)$, by \eqref{c1 and ric L2} and Proposition \ref{prop--energy identity}, we know that whenever the Ricci curvature exhibits $L^2$-concentration along the sequence, or equivalently, when some bubble $Z_I$ in Theorem \ref{thm-bubble decomposition} is not Ricci-flat, then $c_1^2$ must jump on the limit. Consequently, by passing to some subsequence, any $X_j$ for $j$ large cannot fit into a $\mathbb{Q}$-Gorenstein family with central fiber $X_\infty$\footnote{Note that in this case, $X_\infty$ can still be $\mathbb Q$-Gorenstein smoothable, or even be degenerate limit of other smooth cscK metrics in some $\mathbb Q$-Gorenstein family; see the second example below. The notion of cscK metrics converging in a $\mathbb Q$-Gorenstein family is apparently stronger than requiring the central fiber to be $\mathbb Q$-Gorenstein smoothable.}. The converse is also true, namely if there is no non-Ricci-flat bubbles, then we can find a $\mathbb Q$-Gorenstein family $(\mathcal X,\mathcal L)$ with central fiber $(\mathcal X_0,\mathcal L_0)\cong (X_\infty,L_\infty)$. Indeed, by \cite[Theorem 3.5]{KSB}, for a flat family $\mathcal{X}$ over a smooth curve with central fiber having quotient singularities (equivalently, log terminal) and general fiber smooth, the total space $\mathcal{X}$ is $\mathbb{Q}$-Gorenstein if and only if $c_1(\mathcal X_t)^2$ is constant.

According to the discussion above, one can expect examples of cscK metrics on surfaces degenerating in non-$\mathbb{Q}$-Gorenstein families, although, to the best of the authors’ knowledge, no explicit construction is currently known. In this section, we propose two approaches to the construction of such examples.

\

The first approach is suggested by the Zariski openness established in Theorem~\ref{Zariski openness in controlled cone}. The following example is motivated by \cite[Section 5.5]{donaldson2010}, which goes back to Zariski \cite{zariski62}. Let $X_0$ be the projective surface obtained by blowing up $\mathbb{P}^2$ at five points $\{p_i\}_{i=1}^5$, where four points $p_2,\dots,p_5$ lie on a projective line $\ell$ and $p_1$ does not, and let $X_1$ be the blow-up of $\mathbb{P}^2$ at five points in general position.  
On $X_1$, consider the ample $\mathbb{Q}$-line bundle  
\begin{equation}
    L_1 = 3H - \sum_{i=1}^4E_i + \delta(H-E_5)
\end{equation}
where $H$ denotes the pull-back of $\mathcal{O}(1)$, each $E_i$ is a $(-1)$-curve from the blow-up, and $0<  \delta \ll 1$ is rational. On $X_0$, define $L_0$ with the same expression. Then $L_0$ is big and nef but not ample, since $L_0\cdot C=0$ for the proper transform $C$ of $\ell$. As $C^2=-3$, there exists a birational contraction  
\begin{equation}
    \mu: X_0 \to W_0,
\end{equation}
where $W_0$ has an orbifold point locally modeled on $\mathbb{C}^2/\mathbb{Z}_3$. This is a log terminal but non-canonical singularity. Moreover, $W_0$ admits an ample $\mathbb{Q}$-line bundle, still denoted $L_0$, whose pull-back to $X_0$ recovers the big and nef line bundle $L_0$.  

Note that $X_1$ can be viewed as the blow-up of a degree-$5$ del Pezzo surface, which admits a Kähler–Einstein metric by \cite{tianYau}. The class $c_1(L_1)$ is a small perturbation of the pull-back of the anti-canonical class. Building on the equivalence between the properness of the Mabuchi functional and the existence of cscK metrics established in \cite{chen-cheng}, it was shown in \cite[Section 5]{Sze-RCD} that there exists a cscK metric in $c_1(L_1)$ for $\delta$ sufficiently small; see also \cite{BJT} for a more general setting. Note that so far we fix \textit{one} $X_1$, and the resulting $\delta$ depends on the geometry of this particular surface. However, since $\mathrm{Aut}(X_1,L_1)$ is discrete and $c_1(L_1)$ lies in the controlled cone, Theorem~\ref{Zariski openness in controlled cone} applies. Therefore, one can choose $\delta$ uniform and construct a flat family $(X_t,L_t)$ through $(X_1,L_1)$, such that $(X_t,L_t)$ degenerates to $(W_0,L_0)$ in a non-$\mathbb{Q}$-Gorenstein family (as $c_1^2$ jumps in the limit). By picking up some sequence $t_j\to 0$, we can obtain a polarized Gromov–Hausdorff limit $(X_\infty,L_\infty)$ of $(X_{t_j},L_{t_j})$. It seems natural to expect that this limit is exactly $(W_0,L_0)$, which would provide a desired example.

\

The second approach is based on a possible generalization of the results in \cite{PTT23}.  
Let \( X \) be a projective surface with log terminal singularities, admitting an orbifold cscK metric, and suppose that \( X \) admits a local non-\(\mathbb{Q}\)-Gorenstein smoothing. Such examples do exist; for instance, one may take
\(
X = (\mathbb{P}^1 \times \mathbb{P}^1)/\mathbb{Z}_4.
\)
Indeed, since the anticanonical divisor \( -K_X \) is big,  by \cite[Proposition~3.1]{HP10}, there is no local-to-global obstruction to deformations, and hence \( X \) admits a non-\(\mathbb{Q}\)-Gorenstein deformation.
This gives rise to a three-dimensional total space $\mathcal{X}\rightarrow \Delta$ with $\mathcal X_0=X$. By \cite[Proposition~3.1]{Artin}, the general fiber of $\mathcal{X}$ is smooth. 
Although $\mathcal{X}$ is not $\mathbb{Q}$-Gorenstein, \cite[Theorem~3.5]{KSB} shows that it admits a $\mathbb{Q}$-Gorenstein modification 
\[
\phi : \widetilde{\mathcal{X}} \to \mathcal{X},
\] 
which is an isomorphism on the smooth fibers of $\mathcal{X}$ (this is the so-called $\mathbb{Q}$-factorial modification, whose existence in higher dimensions was proved in \cite{BCHM}).  

Fix a relative polarization $\mathcal{L}$ on $\mathcal{X}$ such that there exists a cscK metric with orbifold singularities in $c_1(\mathcal L_0)$ on the central fiber $\mathcal{X}_0$; for example, by perturbing the anticanonical class. One may try to generalize the result in \cite{PTT23}, and show that the properness of the Mabuchi functional is an open condition for the family $(\widetilde{\mathcal{X}},\phi^*\mathcal{L})$. 
Combining this with the work of Chen--Cheng \cite{chen-cheng}, we can then construct cscK metrics on nearby smooth fibers, which is expected to converge to the cscK metric on $(X,\mathcal L_0)$ in the Gromov-Hausdorff topology. To address the issue of nontrivial holomorphic vector fields, one may either restrict attention to varieties \(X\) without such vector fields, or, following \cite[Section~7]{BR15}, work in an equivariant setting.

\subsection{Construction of K-moduli}\label{further remarks}
It is known that the moduli space of polarized cscK metrics form a Hausdorff complex analytic space \cite{FS90,DN25}. With our Zariski openness result in place, and following the framework of \cite{Odaka13},  we obtain the following.
\begin{theorem}[=Theorem \ref{thm-moduli space}]
 Smooth polarized cscK surfaces $(X,L)$ with fixed Hilbert polynomial, finite automorphism group, and polarization lying in a controlled cone form a separated  moduli algebraic space; moreover, this moduli space has only quotient singularities.
\end{theorem}
\begin{proof}
The proof follows from the same argument as that in \cite{Odaka13}. By \cite{keelMori}, we have a coarse moduli space which is an algebraic space and the separatedness follows from the uniqueness of the cscK metrics \cite{donaldson2001} and the uniqueness of limits as Riemannian manifolds. See also \cite{BM85,BB2017,chen-tian} and the references therein for more discussion on the uniqueness of canonical metrics.

We show that the moduli space has only quotient singularities in the following. Note that the controlled condition \ref{controlled cone for cscK} forces $X$ is a blow-up of $\mathbb P^2$ or Hirzebruch surfaces, and therefore we have $H^2(X, \mathcal{O}_X)=H^0(X,K_X)=0$ and $$H^2(X, T_X)=0$$ by \cite[Example 2.4.11]{sernesi2006}. From the exact sequence of the Atiyah extension of $L$ \cite[Section 3.3.3]{sernesi2006}
\begin{equation}
    0 \rightarrow \mathcal{O}_X \rightarrow \mathcal{E}_L \rightarrow T_X \rightarrow 0,
\end{equation} we know that 
\begin{equation}
    H^2(X, \mathcal{E}_L)=0.
\end{equation}
Therefore the moduli algebraic stack (in the sense of Deligne-Mumford) is smooth as the deformation is unobstructed, which implies that moduli space has only quotient singularities.
\end{proof}

The next question is the structure of compactified moduli space. When the polarization is contained in the controlled cone, the compactness theory in \cite{tian-viaclovsky1,tian-viaclovsky2} provides a natural Gromov-Hausdorff compactification. This Gromov--Hausdorff compactification is expected to  \cite{odaka10,Odaka13,OSS16} admit a projective scheme structure and coincide with the moduli space of polarized K-polystable surfaces constructed via algebraic K-stability theory.  This strategy, namely to construct and compactify K-moduli using Riemannian convergence theory, works well in the early stage of moduli theory for Fano manifolds \cite{OSS16,odaka15,LWX19}. We hope our work in this paper would provide some preliminary tools for moduli theory of non-collapsed cscK metrics beyond the Fano category. Moreover, one might also try to construct some explicit examples of compactified K-moduli as in \cite{OSS16}. We propose the following problem:

\begin{problem}
 Construct projective algebro-geometric moduli spaces that coincides with the Gromov-Hausdorff compactification for smooth polarized cscK surfaces \((X,L)\) with a fixed Hilbert polynomial, with finite automorphism group and with \(c_1(L)\) lying in the controlled cone. 
\end{problem}

\bibliographystyle{alpha}
\bibliography{ref.bib}

\end{document}